\documentclass{compositio}
\usepackage[active]{srcltx}
\usepackage{amsmath}\usepackage{amssymb}\usepackage{amscd}\usepackage{amsthm}\usepackage{amsfonts}
\usepackage{cite}

\newcommand{\C}{\mathbb{C}}
\newcommand{\R}{\mathbb{R}}
\newcommand{\Z}{\mathbb{Z}}
\newcommand{\eps}{\varepsilon}

\newcommand{\cc}{\mathrm{cc}}
\newcommand{\qc}{\mathrm{c}}

\newcommand{\bGamma}{\boldsymbol{\Gamma}}
\newcommand{\btheta}{\boldsymbol{\theta}}
\newcommand{\bT}{\boldsymbol{T}}

\newcommand{\Tab}{\mathrm{Tab}}
\newcommand{\STab}{\mathrm{STab}}
\newcommand{\ASTab}{\mathrm{ASTab}}
\newcommand{\cB}{\mathcal{B}}
\newcommand{\sv}{\mathsf{v}}
\newcommand{\bze}{\boldsymbol{0}}

\newcommand{\Seq}{\mathrm{Seq}}
\newcommand{\bS}{\boldsymbol{S}}
\newcommand{\oO}{\mathcal{O}}
\newcommand{\un}[1]{\underline{#1}}
\newcommand{\unn}[1]{\underline{\underline{#1}}}
\newcommand{\oH}{\overline{H}}
\newcommand{\cH}{\mathcal{H}}
\newcommand{\cP}{\mathcal{P}}

\newcommand{\ovr}{\overline{r}}

\newcommand{\Ve}{\mathrm{Vert}}
\newcommand{\pl}{\mathrm{pl}}

\newtheorem{theorem}{\textbf{Theorem}}[section]
\newtheorem{corollary}[theorem]{\textbf{Corollary}}
\newtheorem{lemma}[theorem]{\textbf{Lemma}}
\newtheorem{proposition}[theorem]{\textbf{Proposition}}
\newtheorem*{theo-intro}{\textbf{Theorem}}

\theoremstyle{definition}
\newtheorem{definition}[theorem]{\textbf{Definition}}

\theoremstyle{remark}
\newtheorem{remark}[theorem]{\textbf{Remark}}
\newtheorem{remarks}[theorem]{\textbf{Remarks}}
\newtheorem{example}[theorem]{\textbf{Example}}

\newcounter{pic}

\newenvironment{pic}[1][\bf Fig. \arabic{pic}]{
        
        \refstepcounter{pic}\noindent\textbf{#1.}${}$\hspace{5pt}${}$\it}{}

\numberwithin{equation}{section}

\begin{document} 

\title[Hecke algebras and Young tableaux]{Young tableaux and representations of Hecke algebras of type ADE}

\author{L. Poulain d'Andecy}
\email{loic.poulain-dandecy@univ-reims.fr}
\address{Universit\'e de Reims Champagne-Ardenne, UFR Sciences exactes et naturelles, Laboratoire de Math\'ematiques EA 4535 Moulin de la Housse BP 1039, 51100 Reims, France}

\classification{20C08, 20C30, 05E10, 20F55}

\keywords{Hecke algebras; Weyl groups; simply-laced root systems; affine Hecke algebras; Jucys--Murphy elements; skew partitions; Young tableaux; seminormal representations; calibrated representations}

\date{June, 2014}

\thanks{Part of this research was supported by ERC-advanced grant no.~268105 when the author was in Korteweeg-de Vries Institute for Mathematics in Amsterdam. It is a pleasure to thank Marcelo Gon\c{c}alves de Martino and Eric Opdam for interesting discussions.}

\begin{abstract} 
We introduce and study some affine Hecke algebras of type ADE, generalising the affine Hecke algebras of GL. We construct irreducible calibrated representations and describe the calibrated spectrum. This is done in terms of new families of combinatorial objects equipped with actions of the corresponding Weyl groups. These objects are built from and generalise the usual standard Young tableaux, and are controlled by the considered affine Hecke algebras. By restriction and limiting procedure, we obtain several combinatorial models for representations of finite Hecke algebras and Weyl groups of type ADE. Representations are constructed by explicit formulas, in a seminormal form.
\end{abstract}

\maketitle

\section{Introduction}\label{sec-intro}

\paragraph{\textbf{1.}} 
The classical approach to the representation theory of the symmetric group involves the combinatorial notion of partitions and associated standard tableaux. A classical explicit realisation of the irreducible representations is often referred to as the ``seminormal form" and was already given by A. Young \cite{Yo} (see also \cite{Rut}). In this realisation, a distinguished (``seminormal") basis of the irreducible representation spaces is indexed by standard tableaux and is such that a certain set of commuting elements, the Jucys--Murphy elements, acts as diagonal matrices \cite{Mu}.

Analogues of the seminormal representations of the symmetric group have been given explicitly, using combinatorial constructions, in many other situations. We refer to \cite{AK,AMR,Ch1,Ch2,CP,CP2,En,Ho,HM,IO,IO2,Na1,Na2,OP,OP2,Pu,Ra,SV,Wa}. In all these situations, analogues of Jucys--Murphy elements are present and act semisimply in the representations.
The combinatorial constructions originate in an identification of the sequences of eigenvalues of these Jucys--Murphy elements with combinatorial objects (in the spirit of the work of A. Okounkov and A. Vershik for the symmetric group \cite{OV}). We note that all these situations have a type A flavour and that the combinatorics is related to the classical one (one considers for example strict tableaux, multi-tableaux or up-down tableaux). In particular, the group naturally acting on the set of combinatorial objects is the symmetric group.

The symmetric group is the Weyl group of the irreducible root system of type A. In this paper, we will consider irreducible root systems of simply-laced types. The simply-laced types are the types A, D and E, and we will call ``of type ADE" any object associated to such a root system  (for examples: Weyl group, finite Hecke algebra, affine Hecke algebra, ...). 

\paragraph{\textbf{2.}} In most of the situations quoted above, an ``affine" version of the algebras under consideration controls the eigenvalues of these Jucys--Murphy elements, and in turn the combinatorics involved. The ``affine" algebras controlling the Jucys--Murphy elements of the symmetric group and of the finite Hecke algebra of type A are actually particular examples of affine Hecke algebras (degenerate, or graded, for the symmetric group). They are affine Hecke algebras associated to $GL(n)$. In general, affine Hecke algebras are associated to root data coming from algebraic reductive groups, and a well-known motivation is that their representation theory forms a part of the representation theory of reductive groups defined over $p$-adic fields.

An affine Hecke algebra comes equipped with a distinguished commutative subalgebra. In the case of $GL(n)$, the Jucys--Murphy elements of the symmetric group and of the finite Hecke algebra of type A are obtained as images (in a well-chosen quotient) of well-chosen generators of this commutative subalgebra. This is the reason for the fact that representations of affine Hecke algebras of $GL(n)$ appear in the study of these Jucys--Murphy elements. In fact, in this setting, only a part of the representation theory of affine Hecke algebras is used. Namely, the representations involved are the ones in which the elements of the distinguished commutative subalgebra act semisimply. Such representations are called \emph{calibrated} (sometimes also completely splittable \cite{Kl,Ruf}, or in another context, homogeneous \cite{KR}).

For the affine Hecke algebras of GL, the irreducible calibrated representations are classified and described combinatorially in terms of skew partitions \cite{Ch1,Ch2,Ra} (see also \cite{Ruf}). In the degenerate case, by restriction to the symmetric group, one recovers the so-called ``skew" representations. They were introduced and studied for example in \cite{Ro,JP}, originally in connections with the Littlewood--Richardson rule. In the non-degenerate case, one obtains by restriction the analogues of skew representations for the finite Hecke algebra of type A, whose decompositions into irreducibles also involve Littlewood--Richardson coefficients \cite{Ra}.

Calibrated representations were classified in \cite{Ra2} for affine Hecke algebras associated to the weight lattice of a root system of arbitrary type. Here, we will follow a different line of generalisation by considering different affine Hecke algebras. As far as affine Hecke algebras are concerned, we generalise the construction for the affine Hecke algebra of GL in terms of skew partitions. Besides, our point of view on the affine Hecke algebras considered here is also that, in the spirit explained above, they are algebraic structures controlling some sort of combinatorics, that we apply to representations, in this case, of finite Hecke algebras and Weyl groups of type ADE.

\paragraph{\textbf{3.}} One goal of this paper is to provide a whole class of combinatorial models for representations of finite Hecke algebras and Weyl groups of type ADE.  We will mostly work at the level of some affine Hecke algebras of type ADE and actually construct irreducible (calibrated) representations of these affine Hecke algebras. Representations of finite Hecke algebras are then obtained simply by restriction, and representations of Weyl groups are obtained by a limiting procedure.

Main ingredients in the construction are new sets of combinatorial objects equipped with actions of the corresponding Weyl groups. These combinatorial objects are built from and generalise the usual Young tableaux associated to (skew) partitions. Representations are constructed by explicit formulas, generalising the seminormal representations of the symmetric group associated to (skew) partitions.

Another goal of this paper is to introduce and study some affine Hecke algebras (of type ADE). 
The combinatorial constructions evoked above are controlled by those affine Hecke algebras. We emphasize their fundamental role by introducing them as the starting point of our study. They generalise in a natural sense the affine Hecke algebra of GL. As for GL, they are associated to reductive groups (of type ADE) which are not semisimple. However, as an independent motivation to consider them, they admit certain quotients by a distinguished central element which are isomorphic to affine Hecke algebras associated to semisimple groups (the most studied cases). This will be more detailed below.

For each type (A, D or E), our construction depends on a chosen vertex of the corresponding Dynkin diagram. The classical construction for the symmetric group with the usual partitions then corresponds to type A with the chosen vertex being an extremity of the Dynkin diagram. So, already for the symmetric group, we obtain new combinatorial constructions: those associated to a vertex which is not one of the two extremities of the Dynkin diagram.

\paragraph{\textbf{4.}} We will now describe in more details the main objects and the organisation of the paper. We first introduce some affine Hecke algebra of type ADE (Section \ref{sec-def}). Given an irreducible root system $R$ of type ADE, we take an arbitrary vertex $\sv$ of the Dynkin diagram of $R$, and construct an affine Hecke algebra $H(R,\sv)$. The algebra $H(R,\sv)$ is associated to a free $\Z$-module, denote $L_{R,\sv}$, obtained by extending the root lattice of $R$ by one dimension; the action of the Weyl group on it depends on the choice of $\sv$.

An important step is the choice of a basis of $L_{R,\sv}$. This will be the key to the combinatorial approach, as explained later. We also use this basis to work out explicitly a finite presentation by generators and relations of the algebra $H(R,\sv)$. This presentation helps to see that the algebra $H(R,\sv)$ contains several (up to 3) distinguished subalgebras isomorphic to affine Hecke algebras of GL, and further how they are ``glued" together. It reflects a certain covering, depending on $\sv$, of the Dynkin diagram of $R$ by several (up to 3) subdiagrams of type A.

It turns out that the free $\Z$-module $L_{R,\sv}$ contains an distinguished element, denoted $\Delta_{R,\sv}$, invariant under the action of the Weyl group. Furthermore, the quotient module $L_{R,\sv}/\Z \Delta_{R,\sv}$ turns out to coincide with the following one: the root lattice of $R$ extended by adding the fundamental weight corresponding to the vertex $\sv$. This $\Z$-module, included between the root lattice and the weight lattice, automatically gives rise to an affine Hecke algebra, that we denote $\oH(R,\sv)$, which now corresponds to a semisimple group, unlike $H(R,\sv)$. It follows from our construction that this algebra $\oH(R,\sv)$ is isomorphic to the quotient of $H(R,\sv)$ over a relation $C_{R,\sv}=1$, where $C_{R,\sv}$ is a central element of $H(R,\sv)$. Therefore, we will obtain informations on the representation theory of $\oH(R,\sv)$ immediately from our results on $H(R,\sv)$. This could be an independent motivation for considering the algebras $H(R,\sv)$.

As $\sv$ varies among the set of vertices of the Dynkin diagram of $R$, we obtain different $\Z$-modules included between the root lattice of $R$ and the weight lattice of $R$. We note that all such $\Z$-modules (for types ADE) are obtained this way with the following exceptions: the root lattice of type $A_n$ and the weight lattice of type $D_n$ with $n$ even. To sum up, the algebras $\oH(R,\sv)$, when $R$ and $\sv$ vary, exhaust the set of affine Hecke algebra of type ADE associated with semisimple root data, with these two exceptions.

\paragraph{\textbf{5.}} A crucial notion in our study is that of an action and a truncated action of the Weyl group on sequences of numbers (Section \ref{sec-act}). These sequences are identified with characters of the commutative subalgebra of $H(R,\sv)$, and this is done using the basis of $L_{R,\sv}$ chosen earlier. One goal will be to determine the characters appearing in calibrated representations (\emph{i.e.} the calibrated spectrum) and to interpret them as combinatorial objects. The sequences form the intermediary step between characters and combinatorial objects. The action of the Weyl group of $R$ is explicitly calculated on the chosen basis and is then translated as an action of the Weyl group on sequences (and later on tableaux). This action generalises the permutation action of the symmetric group. 

Roughly speaking, the truncated action on a given sequence consists in ``killing" the action of some simple reflections. It is not an action of the Weyl group anymore, however the important property is that the braid relations are still satisfied. This allows us to define a notion of orbits under this truncated action. Some of these truncated orbits will parametrise the representations constructed later.

The new combinatorial objects, the tableaux of type $(R,\sv)$, are defined in Section \ref{sec-tab}. The action of the Weyl group is described explicitly, in a combinatorial way.  A new feature of the action of the Weyl group, compared to the usual action of the symmetric group on tableaux, is that not only boxes are exchanged, but some of them sometimes also move to a different position. This results in the possibility that, starting from a tableau of type $(R,\sv)$, one obtains, after applying an element of the Weyl group, something which is not a tableau of type $(R,\sv)$. This leads us to the notion of admissibility for a standard tableau of type $(R,\sv)$, the precise definition being that its truncated orbit must contain only standard tableaux of type $(R,\sv)$. Examples are given for several choices of $R$ and $\sv$.

Representations of $H(R,\sv)$ are constructed in Section \ref{sec-rep} with the help of standard tableaux of type $(R,\sv)$. The representations are parametrised by admissible truncated orbits and the representation spaces have a basis indexed by the elements of a given admissible truncated orbits (by admissibility, these are standard tableaux of type $(R,\sv)$). The action of the generators is given explicitly (in a ``seminormal" form) and the representations are irreducible and calibrated.

 As a corollary of the construction of representations associated to admissible truncated orbits, we obtain a complete description of the calibrated spectrum of $H(R,\sv)$. Roughly speaking, the calibrated spectrum is the set of characters of the commutative subalgebra of $H(R,\sv)$ which support a calibrated representation. The result is that the calibrated spectrum of type $(R,\sv)$ is in bijection with the set of admissible standard tableaux of type $(R,\sv)$. In a sense, this sums up the interpretation of these combinatorial objects in terms of representation theory.

It is immediate to determine which representations pass to the quotient and become irreducible calibrated representations of $\oH(R,\sv)$. This property is seen directly on the standard tableaux. A description of the calibrated spectrum of the quotient algebra $\oH(R,\sv)$ also follows immediately from its description as a quotient of $H(R,\sv)$ by a central element, together with general considerations on the calibrated spectrum of a quotient affine Hecke algebra. It is described as a subset of the set of admissible standard tableaux of type $(R,\sv)$, namely those satisfying a condition on their contents given explicitly, and related to the expression of the fundamental weight corresponding to $\sv$ in terms of simple roots.

\paragraph{\textbf{6.}} Let us make connections with the classical situation of $GL(n)$. Among the algebras $H(R,\sv)$, the affine Hecke algebra of $GL(n)$ appears when one takes $R=A_{n-1}$ and $\sv$ an extremity of the Dynkin diagram. In this case, the $\Z$-module $L_{R,\sv}$ is the usual free $\Z$-module with basis $\epsilon_1,\dots,\epsilon_n$ equipped with the permutation action of the symmetric group (the simple roots are $\epsilon_{i+1}-\epsilon_i$ and $\epsilon_1$ is the basis vector added to the root lattice). The distinguished
invariant is $\epsilon_1+\dots+\epsilon_n$ and the quotient algebra $\oH(R,\sv)$ is the affine Hecke algebra associated to the weight lattice (that is, associated to $SL(n)$).

For $GL(n)$, our notion of tableaux of type $(R,\sv)$ coincides evidently with the usual notion of Young tableaux associated to skew shapes. The truncated action in this case is defined as follows. Consider the simple transposition exchanging $i$ and $i+1$. Then its action on a given standard tableau is ``killed" if $i$ and $i+1$ are adjacent. A fact underlying the representation theory of the symmetric group is that this process respects the braid relations, a result that we generalise to any type $(R,\sv)$. We check that, with our definitions, a truncated orbit for $GL(n)$ contains all standard tableaux of a given skew shape. Moreover, here any standard tableau is obviously admissible. To sum up, our construction recovers, for $GL(n)$, the constructions of \cite{Ch1,Ch2,Ra}. In this situation, the general result concerning the quotients $\oH(R,\sv)$ becomes that the representations passing to the affine Hecke algebra of $SL(n)$ are those corresponding to shapes such that the product of the contents of all boxes is equal to 1.

From the point of view of representation theory, the notion of being of the same shape for two standard tableaux is replaced, for a general type $(R,\sv)$, by the notion of belonging to the same truncated orbit. In this sense, the (admissible) truncated orbits replace the usual skew shapes. We note that the notion of shape is not relevant anymore since the boxes can move under the action of the Weyl group. This emphasizes the importance of the truncated action which was used to define the truncated orbits, and whose properties ultimately rely on the fact that the braid relations are preserved. Furthermore, the truncated action is necessary to define the notion of admissibility which could not have been defined using the full action of the Weyl group. We note that the notion of admissibility was not present (or, more precisely, is trivial) in the classical $GL(n)$ situation.
 
\paragraph{\textbf{7.}} Finally, in a last part (Section \ref{sec-res-lim}), the restriction to finite Hecke algebras (of type ADE) is considered. The representations, indexed by admissible truncated orbits, are already constructed. We also obtain representations of Weyl groups (of type ADE) by a limiting procedure in the formulas for the action of the generators. We consider the obtained representations as analogues of the skew representations of the symmetric group and of the finite Hecke algebra of type A. We note that, corresponding to a given type $R$, we obtain several different families of combinatorial constructions (depending of the vertex $\sv$).

The obtained representations of finite Hecke algebras and Weyl groups are not irreducible in general. Nevertheless, we obtain a sufficient condition for these representations to be irreducible, which is expressed combinatorially in terms of tableaux. We name this property as ``to be of level 1" since it corresponds to considering representations which pass to a certain ``cyclotomic" quotient of level 1. The term cyclotomic for these quotients, and the ones of higher levels, is motivated by the remark below.

Again to clarify the connections with the $GL(n)$ situation, we note that the cyclotomic quotients (of any level) exactly coincide, in the $GL(n)$ situation, with the well-studied Ariki--Koike algebras \cite{AK}, also called cyclotomic Hecke algebras. Then the property of being of level 1 for a classical tableau is simply that its shape must be a usual partition (not skew). It is well-known that in this case, the corresponding representations of the finite Hecke algebra of type A (or of the symmetric group) are irreducible. This is the property we generalise to any type $(R,\sv)$. We note that a further property in the $GL(n)$ situation is that all irreducible representations of the finite Hecke algebra are obtained as such level 1 representations, or equivalently, the cyclotomic quotients of level 1 are isomorphic to the finite Hecke algebra. This is not true anymore for a general type $(R,\sv)$ and the understanding of the algebraic structure of the cyclotomic quotients of any level (analogues of the Ariki--Koike algebras) would require a further study.

\vskip 0.5cm
\setcounter{tocdepth}{2}
\tableofcontents

\section{Preliminaries on root data and affine Hecke algebras}\label{sec-prel}

\subsection{Root data}\label{subsec-prel1}

Let $R$ be a finite reduced root system and fix $\{\alpha_i\}_{i=1,\dots,n}$ a set of simple roots for $R$. We denote by $Q_R$ the root lattice:
\[Q_R:=\bigoplus_{i=1}^n\Z\alpha_i\ .\]
We denote by $R^{\vee}$ the dual root system of coroots of $R$ and we identify coroots in $R^{\vee}$ with elements of $\text{Hom}_{\Z}(Q_R,\Z)$ via the natural pairing between roots and coroots.

Let $T$ be a finitely-generated free $\Z$-module with a given inclusion $Q_R\subseteq T$ of $\Z$-modules. Assume moreover that we have elements $\alpha^{\vee}\in\text{Hom}_{\Z}(T,\Z)$, for $\alpha\in R$, such that the set of maps $\{\alpha^{\vee}\}_{\alpha\in R}\,$, restricted to $Q_R$, forms the dual root system $R^{\vee}$. The data $Q_R\subseteq T$ and $R^{\vee}\subset \text{Hom}_{\Z}(T,\Z)$ form a root datum.

By extension of scalars, the coroots $\alpha^{\vee}$ are $\R$-linear forms on $\R\otimes_{\Z}T$. To each root $\alpha\in R$ is associated the reflection $r_{\alpha}$ in $\R\otimes_{\Z} T$ given by
\begin{equation}\label{act-W0}
r_{\alpha}(x):=x-\alpha^{\vee}(x)\,\alpha\ \ \ \ \ \ \text{for any $x\in \R\otimes_{\Z} T$.}
\end{equation}
The group generated by $r_{\alpha}$, $\alpha\in R$, is identified with the Weyl group $W_0(R)$ associated to $R$. Let $r_1,\dots,r_n$ be the reflections in $W_0(R)$ corresponding to the simple roots $\alpha_1,\dots,\alpha_n$ and let $m_{ij}$ be the order of the element $r_ir_j$ in $W_0(R)$ (we have $m_{ii}=1$). The group $W_0(R)$ is generated by $s_1,\dots,s_n$ and a set of defining relations is
\[(r_ir_j)^{m_{ij}}=1\ \ \ \ \ \text{for $i,j=1,\dots,n$\,.}\]
The Weyl group $W_0(R)$ acts on $T$ via Formula (\ref{act-W0}) since $\alpha^{\vee}(x)\in\Z$ if $x\in T$ and $Q_R\subseteq T$.

\begin{example}\label{ex-prel}
The free $\Z$-module $P_R:=\text{Hom}_{\Z}(Q_{R^{\vee}},\Z)$ is called the weight lattice and is naturally identified with a subset of $\R\otimes_{\Z}Q_R$. Namely, the weight lattice $P_R$ consists of elements $\omega\in\R\otimes_{\Z}Q_R$ such that $\alpha^{\vee}(\omega)\in\Z$ for any $\alpha\in R$. The fundamental weights $\omega_1,\dots,\omega_n$ associated to the set of simple roots $\{\alpha_i\}_{i=1,\dots,n}$ are defined by $\alpha_i^{\vee}(\omega_j)=\delta_{i,j}$, $i,j=1,\dots,n$. We then have
\[P_R=\bigoplus_{i=1}^n\Z\omega_i\ .\]
We have obviously $Q_R\subseteq P_R$. Let $L$ be a $\Z$-module such that $Q_R\subseteq L\subseteq P_R$. By definition of $P_R$, the restrictions of the maps $\alpha^{\vee}$ to $L$ belong to $\text{Hom}_{\Z}(L,\Z)$. So finally, for such a $\Z$-module $L$, the natural inclusions $Q_R\subseteq L$ and $R^{\vee}\subset \text{Hom}_{\Z}(L,\Z)$ form a root datum. \hfill$\triangle$
\end{example}

\subsection{Affine Hecke algebras}\label{subsec-prel2}

As we will only consider irreducible simply-laced types, we only give the definition of one-parameter (affine) Hecke algebras.

\paragraph{\textbf{Finite Hecke algebras.}} Let $q$ be an indeterminate. The finite Hecke algebra $H_0(R)$ is the associative $\C [q,q^{-1}]$-algebra with unit generated by elements $g_1,\dots,g_n$ with the defining relations:
\begin{equation}\label{rel-H0}
\begin{array}{ll}
g_i^2=(q-q^{-1})g_i+1 & \text{for}\ i\in\{1\dots,n\}\,,\\[0.2em]
\underbrace{g_ig_jg_i\dots}_{m_{ij}\,terms}=\underbrace{g_jg_ig_j\dots}_{m_{ij}\,terms}\ \ \  & \text{for $i,j\in\{1,\dots,n\}$ with $i\neq j$}\,.
\end{array}
\end{equation} 

For any element $w\in W_0(R)$, let $w=r_{a_1}\dots r_{a_k}$ be a reduced expression for $w$ in terms of the generators $r_1,\dots,r_n$. We define $g_w:=g_{a_1}\dots g_{a_k}\in H_0(R)$. This definition does not depend on the reduced expresion for $w$ and and it is a standard fact that the set of elements
\begin{equation}\label{base-H0}
\{g_w\,,\ \ \ w\in W_0(R)\}
\end{equation} 
forms a $\C [q,q^{-1}]$-basis for $H_0(R)$ (\emph{e.g.} \cite{GP}). The specialization of the algebra $H_0(R)$ at $q=\pm1$ is the group algebra $\C W_0(R)$.

\paragraph{\textbf{Affine Hecke algebras.}}

 Let $\bigl(Q_R\subseteq T,\  R^{\vee}\!\subset\text{Hom}_{\Z}(T,\Z)\bigr)$ be a root datum as before. 
The affine Hecke algebra, denoted by $H(R,T)$, associated to the root datum is the associative $\C [q,q^{-1}]$-algebra with unit generated by elements 
$$g_1,\dots,g_{n}\ \ \ \ \text{and}\ \ \ \ X^x,\ \ x\in T\ ,$$
subject to the defining relations $X^0=1$, $X^xX^{x'}=X^{x+x'}$ for any $x,x'\in T$, and Relations (\ref{rel-H0}) together with
\begin{equation}\label{rel-Lu}
g_iX^x-X^{r_i(x)}g_i=(q-q^{-1})\frac{X^x-X^{r_i(x)}}{1-X^{-\alpha_i}}\ ,
\end{equation}
for any $x\in T$ and $i=1,\dots,n$. It follows from (\ref{act-W0}) that the fraction in the right hand side above is a polynomial in the (commuting) generators $X^{x'}$, $x'\in T$. Explicitly, we have 
\[\frac{X^x-X^{r_i(x)}}{1-X^{-\alpha_i}}=\left\{\begin{array}{ll}0 & \text{if $\alpha_i^{\vee}(x)=0$\,,}\\
X^x\displaystyle\sum_{k=0}^{\alpha_i^{\vee}(x)-1}X^{-k\alpha_i} & \text{if $\alpha_i^{\vee}(x)\in\Z_{>0}$\,,}\\
-X^{r_i(x)}\displaystyle\sum_{k=0}^{-\alpha_i^{\vee}(x)-1}X^{-k\alpha_i} & \text{if $\alpha_i^{\vee}(x)\in\Z_{<0}$\,.}
\end{array}\right. \]
We record here the special cases of Formula \ref{rel-Lu} that will be sufficient for the sequel:
\begin{equation}\label{rel-Lu2}
g_iX^x=X^{x}g_i\ \ \ \text{if $\alpha_i^{\vee}(x)=0$}\ \ \ \ \ \ \text{and}\ \ \ \ \ \ g_iX^xg_i=X^{x+\alpha_i}\ \ \ \text{if $\alpha_i^{\vee}(x)=-1$\,,}
\end{equation}
where we used that $g_i^{-1}=g_i-(q-q^{-1})$.

The following set of elements is a $\C [q,q^{-1}]$-basis of $H(R,T)$ (\emph{e.g.} \cite{Lu}):
\begin{equation}\label{base-aff}
\{X^xg_w\ ,\ \ \ \text{$x\in T$ and $w\in W_0(R)$}\}\,.
\end{equation}
It follows in particular that the (multiplicative) group formed by elements $X^x$, $x\in T$, of $H(R,T)$ can be identified with the (additive) group $T$. It follows also that the subalgebra generated by $g_1,\dots,g_n$ is isomorphic to the finite Hecke algebra $H_0(R)$.

\begin{remark}\label{rem-fin-pres}
An affine Hecke algebra $H(R,T)$ is finitely generated and finitely presented. Indeed, let $\delta_1,\dots,\delta_k\in T$ be elements spanning $T$ over $\Z$. Then, it is obvious that $H(R,T)$ is generated by $g_1,...,g_n$ together with $X^{\pm\delta_1},\dots,X^{\pm\delta_k}$. Moreover, it is straightforward to check that, for $i=1,\dots,n$, the set of relations (\ref{rel-Lu}), where $x\in T$, is implied by the set of relations (\ref{rel-Lu}), where $x\in\{\delta_1,\dots,\delta_k\}$. \hfill $\triangle$
\end{remark}

\paragraph{\textbf{Parabolic root subdata.}} Let $\mathcal{P}$ be a parabolic root subsystem of $R$ corresponding to a subset $S_{\mathcal{P}}\subseteq\{1,\dots,n\}$. Let $T_{\mathcal{P}}\subseteq T$ be a free $\Z$-submodule of $T$ containing the simple roots $\alpha_i$ for $i\in S_{\mathcal{P}}$. Then $T_{\mathcal{P}}$ contains the whole root lattice of $\mathcal{P}$ and the restrictions to $T_{\mathcal{P}}$ of the maps $\alpha^{\vee}$, $\alpha\in\mathcal{P}$, give an inclusion of the dual root system $\mathcal{P}^{\vee}$ in $\text{Hom}_{\Z}(T_{\mathcal{P}},\Z)$. Thus 
$$\bigl(Q_{\mathcal{P}}\subseteq T_{\mathcal{P}},\ \mathcal{P}^{\vee}\!\subset\text{Hom}_{\Z}(T_{\mathcal{P}},\Z)\bigr)$$
is a root datum to which an affine Hecke algebra $H(\mathcal{P},T_{\mathcal{P}})$ is associated. It follows at once from the definitions that the subalgebra of $H(R,T)$ generated by $g_i$, $i\in S_{\mathcal{P}}$, and $X^x$, $x\in T_{\mathcal{P}}$, is a quotient of $H(\mathcal{P},T_{\mathcal{P}})$. From the knowledge of the basis (\ref{base-aff}), we have that this subalgebra is actually isomorphic to $H(\mathcal{P},T_{\mathcal{P}})$.

\paragraph{\textbf{Quotient root data.}} Let $I$ be a $\Z$-submodule of $T$ such that $I\subseteq \text{ker}(\alpha^{\vee})$ for any $\alpha\in R$; in other words, any element of $I$ is fixed by the Weyl group $W_0(R)$. In particular $Q_R\cap I=\{0\}$. We note $[x]$ the class of an element $x\in T$ in the quotient $T/I$.

Then $Q_R\ni\alpha\mapsto [\alpha]\in T/I$ gives an inclusion $Q_R\subseteq T/I$. Further, by assumption, the maps $\alpha^{\vee}$, $\alpha\in R$, factor through $T/I$ and yield an inclusion $R^{\vee}\subset\text{Hom}_{\Z}(T/I,\Z)$. Assume moreover that $T/I$ is a free $\Z$-module. Thus 
\begin{equation}\label{quot-RD}
\bigl(Q_R\subseteq T/I,\ R^{\vee}\!\subset\text{Hom}_{\Z}(T/I,\Z)\bigr)
\end{equation}
is a root datum to which an affine Hecke algebra $H(R,T/I)$ is associated. It follows from the definitions that $H(R,T/I)$ is a quotient of $H(R,T)$, the surjective morphism being given on the generators by $H(R,T)\ni g_i\mapsto g_i\in H(R,T/I)$ and $H(R,T)\ni X^x\mapsto X^{[x]}\in H(R,T/I)$. Again, from the knowledge of 
the basis (\ref{base-aff}), it follows that $H(R,T/I)$ is isomorphic to the quotient of $H(R,T)$ over the relations $X^x=1$ for $x\in I$. Note moreover that the elements $X^x$, $x\in I$, are central in $H(R,T)$.

\subsection{Calibrated spectrum of affine Hecke algebras}\label{subsec-prel3}

In all the paper, we will consider only finite-dimensional representations of affine Hecke algebras.

Let $\bigl(Q_R\subseteq T,\  R^{\vee}\!\subset\text{Hom}_{\Z}(T,\Z)\bigr)$ be a root datum and $H(R,T)$ the associated affine Hecke algebra. We denote $\C(q)H(R,T):=\C(q)\otimes_{\C[q,q^{-1}]}H(R,T)$.
\begin{definition}
A representation $V$ of the affine Hecke algebra $\C(q)H(R,T)$ is called \emph{calibrated} if the elements $X^x$, $x\in T$, act as diagonalizable endomorphisms of $V$.
\end{definition}

Recall that the (additive) Abelian group $T$ is identified with the (multiplicative) Abelian group formed by elements $X^x$, $x\in T$. The set $\text{Hom}(T,\C(q)^{\times})$ is the set of (multiplicative) group homomorphisms from $T$ to $\C(q)^{\times}$ (\emph{i.e.} the set of $\C(q)$-characters of $T$). 

\begin{definition}\label{def-C-spec}
We denote by $\text{C-Spec}\bigl(H(R,T)\bigr)$ the set of elements $\chi\in\text{Hom}(T,\C(q)^{\times})$ such that there exist a calibrated representation $V$ of $\C(q)H(R,T)$ and a non-zero vector $v\in V$ satisfying
\[X^{x}(v)=\chi(X^x)\,v\ \ \ \ \text{for any $x\in T$\,.}\]
We call the set $\text{C-Spec}\bigl(H(R,T)\bigr)$ the \emph{calibrated spectrum} of the affine Hecke algebra $H(R,T)$.
\end{definition}
In other words, the calibrated spectrum $\text{C-Spec}\bigl(H(R,T)\bigr)$ is the subset of $\text{Hom}(T,\C(q)^{\times})$ consisting of all the characters of $T$ appearing when we decompose the calibrated representations of $\C(q)H(R,T)$ as direct sums of characters of $T$ (note that by definition a calibrated representation of $\C(q)H(R,T)$ is semisimple as a representation of $T$).

We consider a quotient root datum as in (\ref{quot-RD}). For later use, we formulate below the general result relating the calibrated spectrum of the quotient affine Hecke algebra $H(R,T/I)$ with the calibrated spectrum of $H(R,T)$.
\begin{proposition}\label{prop-c-spec}
We have:
\[\text{C-Spec}\bigl(H(R,T/I)\bigr)=\{\chi\in\text{C-Spec}\bigl(H(R,T)\bigr)\ \text{such that}\ \chi(X^x)=1\ \text{for}\ x\in I\}\ .\]
\end{proposition}
\begin{proof} We recall that $H(R,T/I)$ is the quotient of $H(R,T)$ by the relations $X^x=1$ for $x\in I$.

Let $V$ be a calibrated representation of $\C(q)H(R,T/I)$. By composition with the canonical surjective morphism from $H(R,T)$ to its quotient $H(R,T/I)$, the representation $V$ can be seen as a representation of $\C(q)H(R,T)$, which is obviously calibrated. The characters $\chi$ of $T$ appearing in this representation satisfy $\chi(X^x)=1$ for $x\in I$. This proves the inclusion from left to right.

Next, let $V$ be a calibrated representation of $\C(q)H(R,T)$ and let $\{v_1,\dots,v_N\}$ be a basis of $V$ of common eigenvectors for elements $X^x$, $x\in T$. Each vector $v_i$ corresponds to a character $\chi_i\in\text{C-Spec}\bigl(H(R,T)\bigr)$ as in Definition \ref{def-C-spec}. Order the basis vectors such that, for some $k\geq0$, the characters $\chi_i$ for $i\in\{1,\dots,k\}$ satisfy $\chi_i(X^x)=1$ for all $x\in I$, and the characters $\chi_{k+1},\dots,\chi_N$ do not.

Let $W$ be the subspace of $V$ spanned by the vectors $v_{k+1},\dots,v_N$. It is easy to see that $W$ coincides with the subspace spanned by all vectors $v-X^x(v)$, where $v\in V$ and $x\in I$. This way, we see that $W$ is stable by $H(R,T)$ since elements $X^x$, $x\in I$, are central in $H(R,T)$. Therefore, if non-zero, $V/W$ becomes a representation of $\C(q)H(R,T)$, obviously calibrated (a basis of common eigenvectors is formed by $v_1+W,\dots,v_k+W$). Moreover, the elements $X^x$, $x\in I$, act as the identity on $V/W$, so that $V/W$ becomes a calibrated representation of  $\C(q)H(R,T/I)$. This proves the inclusion from right to left.
\end{proof}

\begin{remark}\label{rem-cal}
The set of characters in $\text{Hom}(T/I,\C(q)^{\times})$ is naturally a subset of $\text{Hom}(T,\C(q)^{\times})$. The proposition can be formulated equivalently as $\text{C-Spec}\bigl(H(R,T/I)\bigr)=\text{C-Spec}\bigl(H(R,T)\bigr)\cap\text{Hom}(T/I,\C(q)^{\times})$. Besides, we note that the proof can be made shorter if we assume that all characters in $\text{C-Spec}\bigl(H(R,T)\bigr)$ appear in irreducible calibrated representations of $H(R,T)$. This will turn out to be true for the algebras to which we will apply the proposition (see Proposition \ref{prop-rep-quot}). \hfill$\triangle$
\end{remark}

\paragraph{\textbf{Characters and sequences.}} Let $\cB=(\delta_1,\dots,\delta_k)$ be an ordered $\Z$-basis of the free $\Z$-module $T$. Let $\Seq_{k}:=\bigl(\C(q)^{\times}\bigr)^{k}$ be the set of sequences of $k$ elements in $\C(q)$ all different from $0$. 

As $\cB$ is a basis of $T$, a character of $T$ is uniquely determined by its values on elements of $\cB$. Therefore, the set of characters $\text{Hom}(T,\C(q)^{\times})$ is identified with the set of sequences $\Seq_{k}$ by the following bijection
\begin{equation}\label{bij-seq} 
\text{Hom}(T,\C(q)^{\times})\ni\chi\ \ \longleftrightarrow\ \ \Bigl(\chi(X^{\delta_1}),\dots,\chi(X^{\delta_k})\Bigr)\in\Seq_k\ .
\end{equation}
This identification depends on the choice of the ordered basis $\cB$.

\begin{definition}
 We denote $\text{C-Eig}(\cB)$ the subset of $\Seq_{k}$ consisting of sequences corresponding to characters in the calibrated spectrum of $H(R,T)$, via the identification (\ref{bij-seq}).
\end{definition}
Thus, for any ordered basis $\cB$ of $T$, the identification (\ref{bij-seq}) restricts by definition to
\[\text{C-Spec}\bigl(H(R,T)\bigr) \ \ \longleftrightarrow\ \  \text{C-Eig}(\cB)\ ,\]
and identifies the calibrated spectrum $\text{C-Spec}\bigl(H(R,T)\bigr)$ with the set $\text{C-Eig}(\cB)$. To save space, we omit to indicate the algebra $H(R,T)$ in the notation $\text{C-Eig}(\cB)$. This should not cause any confusion as it will always be clear which affine Hecke algebra we will be considering.

\section{Affine Hecke algebras $H(R,\sv)$ and $\oH(R,\sv)$}\label{sec-def}

\subsection{Labellings of Dynkin diagrams}\label{subsec-def-R1}

From now on, let $R$ be an irreducible root system of simply-laced type (that is, $R=A_n$, $n\geq1$, or $R=D_n$, $n\geq 4$, or $R=E_n$, $n=6,7,8$). We consider the labelling of vertices of the Dynkin diagram of $R$ as shown in Figure \ref{Cox-ADE} and refer to it as the ``standard" labelling.

\begin{center}
\setlength{\unitlength}{0.01cm}
\begin{picture}(1550,900)(0,380)
\put(100,1180){$A_n$\ ($n\geq1$):}
\put(500,1200){\circle*{15}}
\put(490,1160){\scriptsize{$1$}}
\put(500,1200){\line(1,0){150}}
\put(650,1200){\circle*{15}}
\put(640,1160){\scriptsize{$2$}}
\put(650,1200){\line(1,0){50}}
\put(750,1200){$\ldots$}
\put(850,1200){$\ldots$}
\put(950,1200){\line(1,0){50}}
\put(1000,1200){\circle*{15}}
\put(970,1160){\scriptsize{$n-1$}}
\put(1000,1200){\line(1,0){150}}
\put(1150,1200){\circle*{15}}
\put(1140,1160){\scriptsize{$n$}}
\put(100,930){$D_n$\ ($n\geq4$):}
\put(470,1050){\scriptsize{$1$}}
\put(500,1050){\circle*{15}}
\put(500,1050){\line(1,-1){100}}
\put(470,840){\scriptsize{$2$}}
\put(500,850){\circle*{15}}
\put(500,850){\line(1,1){100}}
\put(600,950){\circle*{15}}
\put(590,910){\scriptsize{$3$}}
\put(600,950){\line(1,0){150}}
\put(750,950){\circle*{15}}
\put(740,910){\scriptsize{$4$}}
\put(750,950){\line(1,0){50}}
\put(850,950){$\ldots$}
\put(950,950){$\ldots$}
\put(1050,950){\line(1,0){50}}
\put(1100,950){\circle*{15}}
\put(1070,910){\scriptsize{$n-1$}}
\put(1100,950){\line(1,0){150}}
\put(1250,950){\circle*{15}}
\put(1240,910){\scriptsize{$n$}}
\put(100,680){$E_n$\ ($n=6,7,8$):}
\put(500,700){\circle*{15}}
\put(490,720){\scriptsize{$1$}}
\put(500,700){\line(1,0){150}}
\put(650,700){\circle*{15}}
\put(640,720){\scriptsize{$3$}}
\put(650,700){\line(1,0){150}}
\put(800,700){\circle*{15}}
\put(790,720){\scriptsize{$4$}}
\put(800,700){\line(0,-1){150}}
\put(800,550){\circle*{15}}
\put(800,510){\scriptsize{$2$}}
\put(800,700){\line(1,0){150}}
\put(950,700){\circle*{15}}
\put(940,720){\scriptsize{$5$}}
\put(950,700){\line(1,0){50}}
\put(1050,700){$\ldots$}
\put(1150,700){$\ldots$}
\put(1250,700){\line(1,0){50}}
\put(1300,700){\circle*{15}}
\put(1290,720){\scriptsize{$n$}}
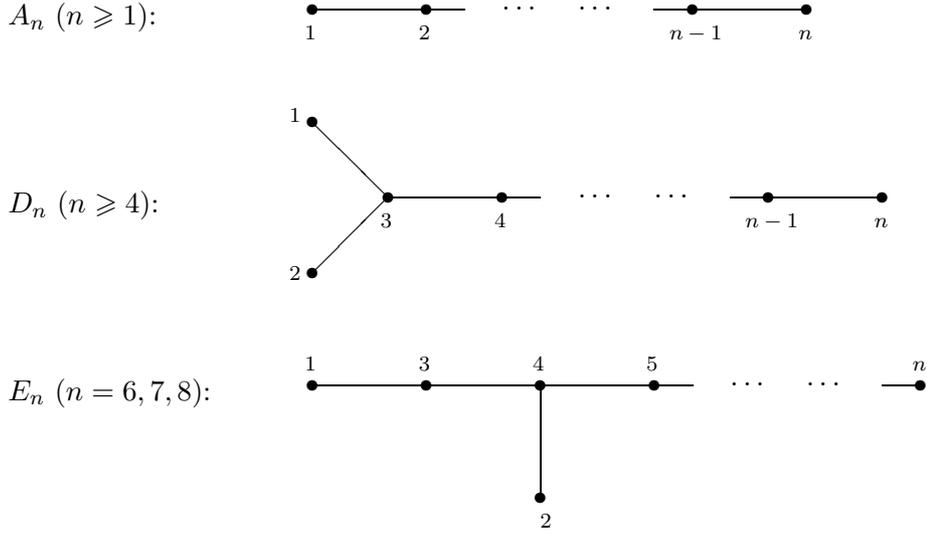
\put(200,420){\begin{pic}\label{Cox-ADE} ``Standard" labelling of Dynkin diagrams of type ADE
\end{pic}}
\end{picture}
\end{center}

\paragraph{\textbf{New labelling of the Dynkin diagram.}} We choose a vertex of the Dynkin diagram of the root system $R$ and denote it by $\sv$. Given $R$ and $\sv$, we fix once and for all a new labelling of the set of vertices of the Dynkin diagram of $R$ such that (as shown in Fig. \ref{Cox-label}): 
\begin{itemize}
\item First, the distinguished vertex $\sv$ is labelled by $1$.
\item Then, $(1,2,\dots,l)$ is the shortest path from $\sv$ to an extremity of the Dynkin diagram.
\item If $\sv$ is not itself an extremity, then $(1,\un{2},\dots,\un{l'})$ is the shortest path from $\sv$ to another extremity of the Dynkin diagram.
\item Finally, in type $D$ or $E$, there is a trivalent vertex. Up to interchanging the two first paths, we can assume that this trivalent vertex is already labelled by $k$ for $1\leq k\leq l-1$. Then $(1,\dots,k,\unn{k+1}\dots,\unn{l''})$ is the shortest path from $\sv$ to the third extremity of the Dynkin diagram.
\end{itemize} 
\begin{remark} We make the following remarks/conventions:

$\bullet$ By convention, if $R=A_n$, we consider that $l=k=l''$, that is, in the partition above, we have $\{\unn{k+1},\dots,\unn{l''}\}=\emptyset$. 

$\bullet$ If $\sv$ is itself an extremity, then we consider that $l'=1$, that is, in the partition above, we have $\{\un{2},\dots,\un{l'}\}=\emptyset$

$\bullet$ It can happen that $k=1$. This happens whenever $\sv$ is a trivalent vertex, that is, when $R=D_n$ (respectively, $R=E_n$) and $\sv$ is the vertex labelled by $3$ (respectively, by $4$) in the standard labelling.
\hfill$\triangle$
\end{remark}

We will make no distinction between a vertex and its label in the new labelling. We denote by $\Ve_R$ the set of vertices of the Dynkin diagram of $R$. The new labelling results in the following partition of the set $\Ve_R$:
\begin{equation}\label{part}
\Ve_R=\{1,\dots,k,k+1,\dots,l\}\sqcup\{\un{2},\dots,\un{l'}\}\sqcup\{\unn{k+1},\dots,\unn{l''}\}\ .
\end{equation}

\begin{center}
\setlength{\unitlength}{0.01cm}
\begin{picture}(1550,600)(-150,-530)
\put(-10,30){\scriptsize{$\un{l'}$}}
\put(0,0){\circle*{15}}
\put(0,0){\line(1,0){50}}
\put(100,0){$\ldots$}
\put(200,0){\line(1,0){50}}
\put(240,30){\scriptsize{$\un{2}$}}
\put(250,0){\circle*{15}}
\put(250,0){\line(1,0){150}}
\put(380,30){\scriptsize{$\sv=1$}}
\put(400,0){\circle*{15}}
\put(400,0){\line(1,0){150}}
\put(540,30){\scriptsize{$2$}}
\put(550,0){\circle*{15}}
\put(550,0){\line(1,0){50}}
\put(650,0){$\ldots$}
\put(750,0){\line(1,0){50}}
\put(790,30){\scriptsize{$k$}}
\put(800,0){\circle*{15}}
\put(800,0){\line(0,-1){150}}
\put(820,-155){\scriptsize{$\unn{k+1}$}}
\put(800,-150){\circle*{15}}
\put(800,-150){\line(0,-1){50}}
\put(800,-280){$\vdots$}
\put(800,-350){\line(0,-1){50}}
\put(820,-405){\scriptsize{$\unn{l''}$}}
\put(800,-400){\circle*{15}}
\put(800,0){\line(1,0){150}}
\put(940,30){\scriptsize{$k+1$}}
\put(950,0){\circle*{15}}
\put(950,0){\line(1,0){50}}
\put(1050,0){$\ldots$}
\put(1150,0){\line(1,0){50}}
\put(1190,30){\scriptsize{$l$}}
\put(1200,0){\circle*{15}}
\put(200,-500){\begin{pic}\label{Cox-label} New labelling of the Dynkin diagram
\end{pic}}
\end{picture}
\end{center}

It will be very convenient to use the following notations:
\begin{equation}\label{notation}
\un{1}:=1\,,\ \ \ \ \ \ \ \unn{k}:=k\ \ \ \ \ \ \ \text{and}\ \ \ \ \ \ \ \left\{\begin{array}{ll}
\un{b}-1:=\un{b-1}\,, & \text{for $b=2,\dots,l'$,}\\[0.5em]
\unn{c}-1:=\unn{c-1}\,,\quad & \text{for $c=k+1,\dots,l''$.}
\end{array}\right.
\end{equation}

The new labelling of the Dynkin diagram induces a new labelling of the simple roots of $R$, of the simple coroots of $R$, of the simple reflections of $W_0(R)$ and of the generators of $H_0(R)$, by the set $\Ve_R$ as in (\ref{part}). They are now respectively denoted
\[\{\alpha_i\}_{i\in\Ve_R}\ ,\ \ \ \ \ \ \{\alpha^{\vee}_i\}_{i\in\Ve_R}\ ,\ \ \ \ \ \ \ \{r_i\}_{i\in\Ve_R}\ ,\ \ \ \ \ \ \ \{g_i\}_{i\in\Ve_R}\ .\]
We recall that, for $i,j\in\Ve_R$, we have $\alpha_i^{\vee}(\alpha_i)=2$ and, for $i\neq j$, $\alpha_j^{\vee}(\alpha_i)=-1$ if the vertices $i$ and $j$ are connected by an edge, while $\alpha_j^{\vee}(\alpha_i)=0$ otherwise.

\subsection{Extensions of the root lattice}\label{subsec-span}

\paragraph{\textbf{The $\Z$-module $L_{R,\sv}$.}} We extend the root lattice $Q_R$ into a larger free $\Z$-module by adding a vector denoted $\eps_{R,\sv}$, which we set to be $\Z$-linearly independent from $Q_R$. Thus, we consider the following free $\Z$-module:
\begin{equation}\label{def-L}
L_{R,\sv}:=\Z\eps_{R,\sv}\oplus Q_R\ .
\end{equation}
The coroots are maps from $Q_R$ to $\Z$ and we extend them to maps from $L_{R,\sv}$ to $\Z$ by setting:
\begin{equation}\label{def-eps}
\text{for $i\in\Ve_R$\,,}\ \ \ \ \ \ \alpha_i^{\vee}(\eps_{R,\sv}):=\left\{\begin{array}{ll}
-1\ \ & \text{if $i=1$\,,}\\[0.4em]
0 & \text{otherwise.}
\end{array}\right.
\end{equation}
This defines an inclusion $R^{\vee}\subset\text{Hom}_{\Z}(L_{R,\sv},\Z)$ which, together with the natural inclusion $Q_R\subset L_{R,\sv}$, forms a root datum.

\paragraph{\textbf{The basis $\cB_{R,\sv}$.}} We set, using notations (\ref{notation}),
\begin{equation}\label{form-basis-inv}
\delta_0:=\eps_{R,\sv}\ \ \ \ \ \quad\text{and}\quad\ \ \ \ \ \delta_{i}:=\delta_{i-1}+\alpha_{i}\,,\ \ \ \ \text{for $i\in\Ve_R$\,.}
\end{equation}
This defines recursively a set of vectors $\{\delta_j\}_{j\in\Ve_R\cup\{0\}}$. More explicitly, we have:
\begin{equation}\label{form-basis}
\begin{array}{ll}
\delta_{a}=\eps_{R,\sv}+\alpha_{1}+\dots+\alpha_{a}\,, & \text{for $a=0,1,\dots,l$,}\\[0.4em]
\delta_{\un{b}}=\eps_{R,\sv}+\alpha_{1}+\alpha_{\un{2}}\dots+\alpha_{\un{b}}\,, & \text{for $b=2,\dots,l'$,}\\[0.4em]
\delta_{\unn{c}}=\eps_{R,\sv}+\alpha_{1}+\dots+\alpha_{k}+\alpha_{\unn{k+1}}+\dots+\alpha_{\unn{c}}\,,\ \ \  & \text{for $c=k+1,\dots,l''$.}
\end{array}
\end{equation}
It follows at once that the set $\{\delta_{j}\}_{j\in\Ve_R\cup\{0\}}$ is a $\Z$-basis of the free $\Z$-module $L_{R,\sv}$. We will always use in the following the basis $\{\delta_{j}\}_{j\in\Ve_R\cup\{0\}}$ that we will order as follows:
\begin{equation}\label{basis}
\cB_{R,\sv}:=\left(\delta_{0},\delta_{1},\dots,\dots,\delta_{l},\,\delta_{\un{2}},\dots,\delta_{\un{l'}},\,\delta_{\unn{k+1}},\dots,\delta_{\unn{l''}}\right)\ .
\end{equation}
We emphasize that we consider $\cB_{R,\sv}$ as an ordered basis.

\subsection{Action of the Weyl group on $\cB_{R,\sv}$.} \label{subsec-act}

From Formulas (\ref{form-basis-inv})-(\ref{form-basis}), it is straightforward to calculate the action of the simple reflections $r_i$, $i\in\Ve_R$, of the Weyl group $W_0(R)$ on elements of the basis $\cB_{R,\sv}$. It results in the following formulas:
\begin{equation}\label{act-W1}
r_i(\delta_{i-1})=\delta_i\ \ \ \ \text{and}\ \ \ \ r_i(\delta_i)=\delta_{i-1}\,,\ \ \qquad\text{for $i\in\Ve_R$\,,} 
\end{equation}
and
\begin{equation}\label{act-W2}
\hspace{-1.7cm}\left\{\begin{array}{rcll}
r_{2}(\delta_{\un{b}}) & = & \delta_{\un{b}}+\delta_{2}-\delta_{1} & \text{for $b=2,\dots,l'$\,,}\\[0.4em]
r_{\un{2}}(\delta_{a}) &  = & \delta_{a}+\delta_{\un{2}}-\delta_{1}\qquad & \text{for $a=2,\dots,l$\,,}
\end{array}\right.
\end{equation}
\begin{equation}\label{act-W3}
\hspace{-0.2cm}\left\{\begin{array}{rcll}
r_{k+1}(\delta_{\unn{c}}) & = & \delta_{\unn{c}}+\delta_{k+1}-\delta_{k}\qquad & \text{for $c=k+1,\dots,l''$\,,}\\[0.4em]
r_{\unn{k+1}}(\delta_{a}) & = & \delta_{a}+\delta_{\unn{k+1}}-\delta_{k}\qquad & \text{for $a=k+1,\dots,l$\,,}
\end{array}\right.
\end{equation}
\begin{equation}\label{act-W4}
\hspace{-0cm}
\left\{\begin{array}{rcll} 
r_{\un{2}}(\delta_{\unn{c}}) &  = & \delta_{\unn{c}}+\delta_{\un{2}}-\delta_{1}\qquad &\text{for $c=k+1,\dots,l''$\,,}\\[0.4em]
r_{\unn{2}}(\delta_{\un{b}}) & = & \delta_{\un{b}}+\delta_{\unn{2}}-\delta_{1} & \text{for $b=2,\dots,l'$\ \ (if $k=1$)\,.}
\end{array}\right.
\end{equation}
Finally, if the action of a simple reflection $r_i$ on a basis element $\delta_j$ is not specified above, it means that it is trivial:
\begin{equation}\label{act-W5}
r_i(\delta_{j})=\delta_j\,,\ \ \qquad\text{otherwise.}
\end{equation}
Note that the second line in (\ref{act-W4}) is present only when $k=1$. If $k>1$ then $r_{\unn{k+1}}$ acts trivially on $\delta_{\un{b}}$ for $b=2,\dots,l'$.

\subsection{The affine Hecke algebra $H(R,\sv)$}\label{subsec-H}

Given $R$ and $\sv$ as before, we denote by $H(R,\sv)$ the affine Hecke algebra associated to the root datum $\bigl(Q_R\subset L_{R,\sv},\  R^{\vee}\!\subset\text{Hom}_{\Z}(L_{R,\sv},\Z)\bigr)$ defined in Section \ref{subsec-span}.

From the general definition of affine Hecke algebras, the algebra $H(R,\sv)$ is generated by $g_i$, $i\in\Ve_R$, and $X^x$, $x\in L_{R,\sv}$, subject to the defining relations as recalled in Section \ref{sec-prel}. We will make more explicit the relations by considering generators of $H(R,\sv)$ associated to the basis $\cB_{R,\sv}$ of the $\Z$-module $L_{R,\sv}$.

We define 
\[X_{j}:=X^{\delta_{j}}\,,\ \ \quad\text{for $j\in\Ve_R\cup\{0\}$.}\]
Since $\cB_{R,\sv}$ is a basis of $L_{R,\sv}$, the affine Hecke algebra $H(R,\sv)$ is generated by the following elements:
\[g_i,\ i\in\Ve_R\,\ \ \ \ \text{and}\ \ \ \ X_{j}^{\pm1},\ j\in\Ve_R\cup\{0\}\ .\]
Due to Remark \ref{rem-fin-pres}, the set of defining relations for $H(R,\sv)$ can be reduced to a finite set of relations involving the generators above. We first have the relations involving only the generators $g_i$, $i\in\Ve_R$, which are given in (\ref{rel-H0}). We also have the relations involving only the generators $X_j$, $j\in\Ve_R\cup\{0\}$:
\begin{equation}\label{rel-H-XX}
X_iX_j=X_jX_i\ \ \ \ \ \ \text{for any $i,j\in\Ve_R\cup\{0\}$.}
\end{equation} 
Then we have the relations between generators $g_i$, $i\in\Ve_R$ and $X_j$, $j\in\Ve_R\cup\{0\}$. They are given in a uniform manner by Relation (\ref{rel-Lu}), where $x$ runs over the basis elements in $\cB_{R,\sv}$. Using the action of $W_0(R)$ on the basis $\cB_{R,\sv}$ calculated in (\ref{act-W1})--(\ref{act-W5}), we give them explicitly. Note that only the formulas in (\ref{rel-Lu2}) are used. First, corresponding to (\ref{act-W1}), we have:
\begin{equation}\label{rel-H-gX1}
g_{i}X_{i-1}g_{i} = X_{i}\ ,\ \ \quad\text{for $i\in\Ve_R$\,.}
\end{equation}
Then, corresponding successively to (\ref{act-W2})--(\ref{act-W4}), we have:
\begin{equation}\label{rel-H-gX2}
\hspace{-1.9cm}\left\{\begin{array}{rcll}
g_{2}X_{\un{b}}g_{2} & = & X_{\un{b}}X_{2}X_{1}^{-1}  & \text{for $b=2,\dots,l'$\,,}\\[0.6em]
g_{\un{2}}X_{a}g_{\un{2}} & = & X_{a}X_{\un{2}}X_{1}^{-1} \qquad & \text{for $a=2,\dots,l$\,,}
\end{array}\right.
\end{equation}
\begin{equation}\label{rel-H-gX3}
\hspace{-0cm}\left\{\begin{array}{rcll}
g_{k+1}X_{\unn{c}}g_{k+1} & = & X_{\unn{c}}X_{k+1}X_{k}^{-1}\qquad & \text{for $c=k+1,\dots,l''$\,,}\\[0.6em]
g_{\unn{k+1}}X_{a}g_{\unn{k+1}} & = & X_{a}X_{\unn{k+1}}X_{k}^{-1} \quad & \text{for $a=k+1,\dots,l$\,,}
\end{array}\right.
\end{equation}
\begin{equation}\label{rel-H-gX4}
\hspace{-0.2cm}
\left\{\begin{array}{rcll}
g_{\un{2}}X_{\unn{c}}g_{\un{2}} & = & X_{\unn{c}}X_{\un{2}}X_{1}^{-1}  & \text{for $c=k+1,\dots,l''$\,,}\\[0.6em]
g_{\unn{2}}X_{\un{b}}g_{\unn{2}} & = & X_{\un{b}}X_{\unn{2}}X_{1}^{-1} \qquad &\text{for $b=2,\dots,l'$\ \ (if $k=1$)\,.}
\end{array}\right.
\end{equation}
Again, the second relation in (\ref{rel-H-gX4}) is present only if $k=1$. Lastly, there are cases where $g_i$ and $X_j$ commute. It happens if and only if $r_i(\delta_j)=\delta_j$. To give the precise list, let $\mathcal{CP}$ be the set consisting of pairs $(i,j)$, where $i\in\Ve_R$ and $j\in\Ve_R\cup\{0\}$, such that:
\begin{itemize}
\vspace{-0.1cm}
\item first, $j\notin\{i,i-1\}$;
\item second, $(i,j)$ is not one of the pairs appearing in left hand sides of (\ref{rel-H-gX2})--(\ref{rel-H-gX4}).
\end{itemize}
Then we have:
\begin{equation}\label{rel-H-gX5}
g_iX_{j}= X_{j}g_i\,,\ \ \quad \text{for any $(i,j)\in\mathcal{CP}$\,.}
\end{equation}
Note in particular that $g_iX_{0}= X_{0}g_i$ if $i\neq 1$.

\begin{remark}\label{rem-rel} 
We will see in the next subsection that the first set of relations (\ref{rel-H-gX1}) occurs inside subalgebras isomorphic to affine Hecke algebras of GL. Vaguely speaking, the remaining relations (\ref{rel-H-gX2})--(\ref{rel-H-gX5}) can be seen as ``mixed" relations expressing how these subalgebras are ``glued" together to form $H(R,\sv)$. Besides, Relations (\ref{rel-H-gX1})--(\ref{rel-H-gX5}) (or equivalently, Relations (\ref{act-W1})--(\ref{act-W5})) will be interpreted combinatorially in Section \ref{sec-tab}, in terms of an action of the Weyl group on some combinatorial objects to be introduced later. \hfill$\triangle$
\end{remark}

\subsection{Subalgebras isomorphic to $H(GL_{N+1})$}\label{subsec-GL}

\paragraph{\textbf{The case $R=A_N$ and $\sv$ an extremity.}} Let $N>0$. We consider the particular situation where $R=A_N$ and $\sv$ is an extremity of the Dynkin  diagram (say, $\sv=1$ in the standard labelling of Fig. \ref{Cox-ADE}). In this case, the new labelling of the Dynkin diagram of Fig. \ref{Cox-label} coincides with the standard one:
\begin{center}
\setlength{\unitlength}{0.01cm}
\begin{picture}(1550,50)(150,1160)
\put(500,1200){\circle*{15}}
\put(490,1160){\scriptsize{$1$}}
\put(500,1200){\line(1,0){150}}
\put(650,1200){\circle*{15}}
\put(640,1160){\scriptsize{$2$}}
\put(650,1200){\line(1,0){50}}
\put(750,1200){$\ldots$}
\put(850,1200){$\ldots$}
\put(950,1200){\line(1,0){50}}
\put(1000,1200){\circle*{15}}
\put(970,1160){\scriptsize{$N-1$}}
\put(1000,1200){\line(1,0){150}}
\put(1150,1200){\circle*{15}}
\put(1140,1160){\scriptsize{$N$}}
\end{picture}
\end{center}
In this case, we have $\{\un{2},\dots,\un{l'}\}=\{\unn{k+1},\dots,\unn{l''}\}=\emptyset$, that is, $l'=1$ and $l''=k=l$. The ordered basis $\cB_{A_N,1}$ in (\ref{basis}) of the $\Z$-module $L_{A_N,1}$ reads
\[\cB_{A_N,1}=(\delta_{0},\delta_1,\dots,\delta_N)\ ,\]
where $\delta_i=\eps_{A_N,1}+\alpha_1+\dots+\alpha_i$ for $i=0,1,\dots,N$. The Weyl group here is the symmetric group $S_{N+1}$ and it acts by permuting the $\delta_i$'s: the simple reflection $r_i$, $i\in\{1,\dots,N\}$, exchanges $\delta_i$ and $\delta_{i-1}$.

With the presentation of the preceding subsection, the algebra $H(A_N,1)$ is generated by elements (to avoid confusion later, we rename the generators: $g_i\leadsto \sigma_i$ and $X_j\leadsto Y_j$)
\[\sigma_1,\dots,\sigma_{N},Y_0^{\pm1},Y_1^{\pm1},\dots,Y_N^{\pm1}\ ,\]
and the relations between generators $\sigma_i$ and $Y_j$ are (Relations (\ref{rel-H-gX2})--(\ref{rel-H-gX4}) are not present here):
\begin{equation}\label{rel-GL2}
\begin{array}{ll}
\sigma_iY_{i-1}\sigma_i=Y_{i}\qquad& \text{for $i\in\{1,\dots,N\}$}\,,\\[0.2em]
\sigma_iY_j=Y_j\sigma_i& \text{for $i\in\{1,\dots,N\}$, $j\in\{0,1,\dots,N\}$ such that $j\neq i-1,i$\,.}
\end{array}
\end{equation}
This algebra is the usual affine Hecke algebras of $GL(N+1)$ and will thus be denoted in the following by $H(GL_{N+1})$.

\paragraph{\textbf{Distinguished subalgebras of $H(R,\sv)$.}} We return now to the general setting of $H(R,\sv)$. We use below the general property on parabolic root subdata recalled in Section \ref{sec-prel}.

We set:
\begin{equation}\label{def-P}
\cP_1:=\{1,2,\dots,l\}\,,\ \ \ \ \ \cP_2:=\{1,\un{2},\dots,\un{l'}\}\,,\ \ \ \ \ \cP_3:=\{1,\dots,k,\unn{k+1},\dots,\unn{l''}\}\,.
\end{equation}
Let $u\in\{1,2,3\}$. We consider the following subset of $\Ve_R$ (\emph{i.e.} of simple roots) and the following submodule of $L_{R,\sv}$:
\[\cP_u\ \ \ \ \quad\text{and}\ \ \ \ \quad\Z\eps_{R,\sv}\oplus\bigoplus_{j\in\cP_u}\Z\alpha_{j}\ ,\]
and we denote by $\cH_u$ the corresponding subalgebra of $H(R,\sv)$:
\begin{equation}\label{def-Hu}
\cH_u:=\langle X_{0}^{\pm1},\ \ g_j,X_j^{\pm1}\ (j\in\cP_u)\ \rangle\ .
\end{equation}
For $u=1,2,3$, we have that $\cH_u$ is isomorphic to $H(GL_{l_u+1})$, where $l_u$ is equal to $l$, $l'$ and $l''$ respectively. The isomorphisms are given on the generators by
\begin{equation}\label{iso-sub1}
\begin{array}{rcl}
\cH_1 & \cong & H(GL_{l+1})\\[0.2em]
g_{a} & \mapsto & \sigma_a\ \ \quad (a=1,\dots,l)\,,\\[0.2em]
X_{a} & \mapsto & Y_a\ \ \quad (a=0,1,\dots,l)\,,
\end{array}
\end{equation}
\begin{equation}\label{iso-sub2}
\begin{array}{c}
\cH_2\ \cong\ H(GL_{l'+1})\\[0.2em]
g_{1}\ \mapsto\ \sigma_1\qquad\qquad\qquad\text{and}\quad\ g_{\un{b}}\ \mapsto\ \sigma_b\ \ \quad (b=2,\dots,l')\,,\\[0.2em]
X_{a}\ \mapsto\ Y_a\ \ \ (a=0,1)\ \quad\text{and}\quad\ X_{\un{b}}\ \mapsto\ Y_b\ \ \quad (b=2,\dots,l')\,,
\end{array}
\end{equation}
\begin{equation}\label{iso-sub3}
\begin{array}{c}
\cH_3\ \cong\ H(GL_{l''+1})\\[0.2em]
g_{a}\ \mapsto\ \sigma_a\ \ \ (a=1,\dots,k)\ \quad\text{and}\ \quad g_{\unn{c}}\ \mapsto\ \sigma_c\ \ \ (c=k+1,\dots,l'')\,,\\[0.2em]
X_{a}\ \mapsto\ Y_a\ \ \ (a=0,1,\dots,k)\ \quad\text{and}\ \quad X_{\unn{c}}\ \mapsto\ Y_c\ \ \ (c=k+1,\dots,l'')\,.
\end{array}
\end{equation}

\subsection{The quotient affine Hecke algebra $\oH(R,\sv)$ of $H(R,\sv)$}\label{subsec-quot}

\paragraph{\textbf{Invariant for the action and quotient root datum.}} Recall that $\sv$ is a chosen vertex of the Dynkin diagram of $R$, which is labelled by $1$ in the new labelling.  The corresponding simple root  is $\alpha_1$ (in the new labelling that we always use). Let $\omega_{\sv}$ be the fundamental weight corresponding to $\alpha_{1}$. This is the element of $\R\otimes_{\Z}Q_R$ determined by $\alpha^{\vee}_{1}(\omega_{\sv})=1$ and $\alpha^{\vee}_{i}(\omega_{\sv})=0$ for $i\in\Ve_R\backslash\{1\}$. As an element of $\R\otimes_{\Z}Q_R$, the fundamental weight $\omega_{\sv}$ can be written as
\begin{equation}\label{weight}
\omega_{\sv}=\frac{1}{n_0}\sum_{i\in\Ve_R}n_i\,\alpha_i\ ,
\end{equation}
where $n_j$, $j\in\Ve_R\cup\{0\}$, are integers such that $n_0$ is the smallest possible positive integer (that is, $n_0>0$ and $\text{GCD}\bigl\{n_j\,,\ j\in\Ve_R\cup\{0\}\bigr\}=1$). We then define:
\begin{equation}\label{def-eps0}
\Delta_{R,\sv}:=n_0\eps_{R,\sv}+\sum_{i\in\Ve_R} n_i\,\alpha_i\in L_{R,\sv}\ .
\end{equation}
\begin{lemma}\label{lem-quot}
\begin{enumerate}
\item The element $\Delta_{R,\sv}\in L_{R,\sv}$ is invariant under the action of $W_0(R)$.
\item The quotient module $L_{R,\sv}/\Z\Delta_{R,\sv}$ is isomorphic as a $W_0(R)$-module to $Q_R+\Z\omega_{\sv}$.
\end{enumerate}
\end{lemma}
\begin{proof}
(i) By definition of the fundamental weight $\omega_{\sv}$ and its expression (\ref{weight}), we have:
\[\text{for $i\in\Ve_R$\,,}\ \ \ \ \ \ \alpha_{i}^{\vee}\Bigl(\sum_{i\in\Ve_R}\! n_i\,\alpha_i\Bigr)=\left\{\begin{array}{ll}
n_0\quad & \text{if $i=1$\,,}\\[0.4em]
0 & \text{otherwise.}
\end{array}\right.\]
 From the action of $\alpha_i^{\vee}$ on $\eps_{R,\sv}$ set in (\ref{def-eps}), it follows that $\alpha_i^{\vee}(\Delta_{R,\sv})=0$ for any $i\in\Ve_R$.

(ii) We define a $\Z$-linear map $\phi$ from $L_{R,\sv}$ to $Q_R+\Z\omega_{\sv}$ by:
\begin{equation}\label{phi}
\eps_{R,\sv}\mapsto-\omega_{\sv}\ \ \ \ \ \text{and}\ \ \ \ \ L_{R,\sv}\ni\alpha_i\mapsto\alpha_i\in Q_R+\Z\omega_{\sv}\ \ \ (i\in\Ve_R)\ .
\end{equation}
As $\alpha_i^{\vee}(\eps_{R,\sv})=\alpha_i^{\vee}(-\omega_{\sv})$ for $i\in\Ve_R$, we have that $\phi$ is an homomorphism of $W_0(R)$-modules. Moreover, it is obviously surjective. It remains to check that the kernel of $\phi$ coincides with $\Z\Delta_{R,\sv}$. By construction, we have $\phi(\Delta_{R,\sv})=0$. On the other hand, let $x:=c\eps_{R,\sv}+\sum_{i\in\Ve_R} c_i\alpha_i\in L_{R,\sv}$. We have:
\[\phi(x)=-c\omega_{\sv}+\sum_{i\in\Ve_R} c_i\alpha_i=0\ \ \quad\Longleftrightarrow\ \ \quad c_i=\frac{cn_i}{n_0}\ \ \ \text{for $i\in\Ve_R$.}\]
In particular, it implies that $\displaystyle\frac{cn_i}{n_0}\in\Z$ for all $i\in\Ve_R$, and thus that $\displaystyle\frac{c}{n_0}\in\Z$ since by assumption, we have  $\text{GCD}\bigl\{n_j\,,\ j\in\Ve_R\cup\{0\}\bigr\}=1$. We conclude that $x=\displaystyle\frac{c}{n_0}\Delta_{R,\sv}\in \Z\Delta_{R,\sv}$.
\end{proof}

\begin{remark}\label{quot-root}
As $\sv$ runs over the set of vertices of the Dynkin diagram of $R$, the $\Z$-module $Q_R+\Z\omega_{\sv}$ varies among the set of $\Z$-modules $L$ such that $Q_R\subseteq L\subseteq P_R$. In fact, all such $\Z$-modules $L$ are obtained this way, with two exceptions, which are the root lattice when $R=A_n$, and the weight lattice when $R=D_n$ and $n$ is even (this follows from the explicit formulas for the fundamental weights, see \emph{e.g.} \cite{Hu}). We also note that, for a given $R$, different choices of $\sv$ can produce the same $\Z$-modules; for example, in type $E_8$, the $\Z$-module $Q_R+\Z\omega_{\sv}$ is the root lattice for any $\sv$.\hfill$\triangle$
\end{remark}

\paragraph{\textbf{The algebra $\oH(R,\sv)$.}}
The $\Z$-module $Q_R+\Z\omega_{\sv}$ satisfies $Q_R\subseteq Q_R+\Z\omega_{\sv}\subseteq P_R$, and so, as recalled in Example \ref{ex-prel}, it defines a root datum. We define $\oH(R,\sv)$ to be the corresponding affine Hecke algebra associated to $Q_R+\Z\omega_{\sv}$.

We set:
\begin{equation}\label{def-C}
C_{R,\sv}:=X^{\Delta_{R,\sv}}\in H(R,\sv)\ .
\end{equation}

The following is a corollary of Lemma \ref{lem-quot}, where we apply the general setting of quotient root data recalled in Section \ref{sec-prel}.
\begin{corollary}\label{cor-quot}
\begin{enumerate}
\item The element $C_{R,\sv}$ is central in $H(R,\sv)$.
\item The quotient of $H(R,\sv)$ over the relation $C_{R,\sv}=1$ is isomorphic to $\oH(R,\sv)$.
\end{enumerate}
\end{corollary}

It is desirable to express $C_{R,\sv}$ in terms of the generators $X_j$, $j\in\Ve_R\cup\{0\}$, of $H(R,\sv)$ that we are using. To do this, we need to decompose the element $\Delta_{R,\sv}$ in the basis $\cB_{R,\sv}$ of $L_{R,\sv}$. Recall that
\[\Delta_{R,\sv}:=n_0\eps_{R,\sv}+\sum_{i\in\Ve_R} n_i\,\alpha_i\ ,\]
where $n_j$, $j\in\Ve_R\cup\{0\}$ are as explained after (\ref{def-eps0}). It is straightforward, using (\ref{form-basis-inv}),  to check that
\begin{equation}\label{form-Delta}
\Delta_{R,\sv}=\sum_{j\in\Ve_R\cup\{0\}} \kappa_j\,\delta_j\ \ \ \ \ \text{where $\,\kappa_j=n_j - \sum_{j=i-1}\! n_{i}$ ,}
\end{equation}
and where the sum in the expression for $\kappa_j$ is taken over $i\in\Ve_R$ satisfying $j=i-1$. For example, $\kappa_0=n_0-n_1$, $\kappa_1=n_1-n_2-n_{\un{2}}$ if $k>1$ and $\kappa_1=n_1-n_2-n_{\un{2}}-n_{\unn{2}}$ if $k=1$.

To sum up, the central element $C_{R,\sv}$ of $H(R,\sv)$ is given in terms of the generators by:
\begin{equation}\label{form-C}
C_{R,\sv}=\prod_{j\in\Ve_R\cup\{0\}} X_j^{\kappa_j}\ ,
\end{equation}
where the powers $\kappa_j$ are given in (\ref{form-Delta}), and the affine Hecke algebra $\oH(R,\sv)$ is the quotient of $H(R,\sv)$ over the relation $\prod_{j\in\Ve_R\cup\{0\}}X_j^{\kappa_j}=1$.

\begin{example}\label{ex-C-GL}
Let $R=A_N$ and $\sv=1$ in the standard labelling. The resulting affine Hecke algebra is $H(GL_{N+1})$; see Section \ref{subsec-GL}. In this case, the application of the formulas above yields the following expression for the central element:
\[C_{A_N,1}=Y_0Y_1\dots Y_N\ ,\]
since the fundamental weight $\omega_{\sv}$ is given by $\omega_{\sv}=\frac{1}{N+1}\bigl(N\alpha_1+(N-1)\alpha_2+\dots+\alpha_N\bigr)$. So here, the algebra $\oH_{A_N,1}$ is the quotient over $Y_0Y_1\dots Y_N=1$. It is the affine Hecke algebra associated to the weight lattice $P_{A_N}$, or in other words, the affine Hecke algebra associated to $SL_{N+1}$.\hfill$\triangle$
\end{example}

\section{Action and truncated action of the Weyl group on sequences}\label{sec-act}

We fix $R$ and $\sv$ as in the preceding section. 
We defined a free $\Z$-module $L_{R,\sv}$, on which the Weyl group $W_0(R)$ acts, with ordered basis
\begin{equation}\label{basis2}
\cB_{R,\sv}:=\left(\delta_{0},\delta_{1},\dots,\dots,\delta_{l},\,\delta_{\un{2}},\dots,\delta_{\un{l'}},\,\delta_{\unn{k+1}},\dots,\delta_{\unn{l''}}\right)\ .
\end{equation}
We recall that the ordering of the elements of $\cB_{R,\sv}$ is relevant
.  The action of the simple reflections $r_i$, $i\in\Ve_R$, on the basis elements is given explicitly in (\ref{act-W1})--(\ref{act-W5}).

\subsection{Sequences and truncated operators}\label{subsec-trunc}

We define $\Seq_{R,\sv}:=\bigl(\C(q)^{\times}\bigr)^{|\Ve_R|+1}$ and we label the elements of a sequence $\bS\in\Seq_{R,\sv}$ in the following way, in accordance with the labelling of elements of the ordered basis $\cB_{R,\sv}$:
\begin{equation}\label{string}
\bS=\bigl(s_{0},s_{1},\ldots,\ldots,s_{l},\,s_{\un{2}},\ldots,s_{\un{l'}},\,s_{\unn{k+1}},\ldots,s_{\unn{l''}}\bigr)\ .
\end{equation}
We will use the notation $\qc_j(\bS):=s_j$ for $j\in\Ve_R\cup\{0\}$. It will be sometimes visually convenient to write a sequence $\bS\in\Seq_{R,\sv}$ as an array as follows:
\begin{equation}\label{string-array}
\left(s_{0} \,,s_{1}\,,\begin{array}{ccccc}s_{2}\,, & \dots & \dots\,, & s_{k} \,, & \begin{array}{cccc}
s_{k+1} \,, & \dots & \dots\,, & s_{l}\\[0.2em]
s_{\unn{k+1}} \,, & \dots\,, & s_{\unn{l''}} &
\end{array}\\ 
s_{\un{2}} \,, & \dots\,, & s_{\un{l'}} & 
\end{array}\right)\ .
\end{equation}

\paragraph{\textbf{Truncated operators $\ovr_i$, $i\in\Ve_R$, on sequences.}} We recall from Section \ref{subsec-prel3} that the choice of the ordered basis $\cB_{R,\sv}$ of $L_{R,\sv}$ determines a bijection between the set of characters $\text{Hom}\bigl(L_{R,\sv},\C(q)^{\times}\bigr)$ and the set $\Seq_{R,\sv}$. The labelling of elements of a sequence $\bS\in\Seq_{R,\sv}$ was made such that the corresponding  character $\chi_S$ is given by
\[\chi_{\bS}(X^{\delta_j})=\qc_j(\bS)\,,\ \ \ \ \ \text{for $j\in\Ve_R\cup\{0\}$\,.}\]

The Weyl group $W_0(R)$ acts on $L_{R,\sv}$ and in turn on characters of $L_{R,\sv}\,$. Namely, for $\chi\in\text{Hom}\bigl(L_{R,\sv},\C(q)^{\times}\bigr)$ and $w\in W_0(R)$, the character $w(\chi)$ is given by
\[w(\chi)\bigl(X^{x}\bigr)=\chi\bigl(X^{w^{-1}(x)}\bigr)\ ,\ \ \ \ \ \text{for $x\in L_{R,\sv}$.}\]
\begin{definition}\label{def-ovr}
Let $\bS\in\Seq_{R,\sv}$ corresponding to the character $\chi_{\bS}$ of $L_{R,\sv}$.

\textbf{(i)} For $w\in W_0(R)$, we define $w(\bS)$ to be the sequence in $\Seq_{R,\sv}$ corresponding to the character $w(\chi_{\bS})$.

\textbf{(ii)} Let $\bze$ be a symbol. For $i\in\Ve_R$, we define the \emph{truncated operator} $\ovr_i$ on $\Seq_{R,\sv}\cup\{\bze\}$ by:
\begin{equation}\label{def-ovri}
\ovr_i(\bze):=\bze\ \ \ \qquad\text{and}\qquad\ \ \ \ovr_i(\bS):=\left\{\begin{array}{ll}
\bze & \text{if $\qc_i(\bS)=q^{\pm2}\qc_{i-1}(\bS)$\ ,}\\[0.4em]
r_i(\bS)\ \  & \text{otherwise.}
\end{array}\right.
\end{equation}
\end{definition}
We note that, for any $a_1,\dots,a_m\in\Ve_R$ and $\bS\in\Seq_{R,\sv}$, we have by definition:
\begin{equation}\label{r-ovr}
\ovr_{a_1}\dots \ovr_{a_m}(\bS)\neq\bze\ \ \ \quad\Longleftrightarrow\quad\ \ \ \ovr_{a_1}\dots \ovr_{a_m}(\bS)=r_{a_1}\dots r_{a_m}(\bS)\ .
\end{equation}

\subsection{Properties of truncated operators and truncated orbits.}

\begin{lemma}\label{lem-mon}
Let $i,j\in\Ve_R$ and $\bS\in\Seq_{R,\sv}$. We have
\begin{equation}\label{lem-cont}
\frac{\qc_i(\bS)}{\qc_{i-1}(\bS)}=\frac{\qc_i\bigl(r_j(\bS)\bigr)}{\qc_{i-1}\bigl(r_j(\bS)\bigr)}\ \ \ \ \ \ \qquad\text{if $m_{i,j}=2$\,,}
\end{equation}
\begin{equation}\label{lem-cont2}
\frac{\qc_i\bigl(r_j(\bS)\bigr)}{\qc_{i-1}\bigl(r_j(\bS)\bigr)}=\frac{\qc_j\bigl(r_i(\bS)\bigr)}{\qc_{j-1}\bigl(r_i(\bS)\bigr)}\ \ \ \ \ \quad\text{if $m_{i,j}=3$\,.}
\end{equation}  
\end{lemma}
\begin{proof} Let $i,j\in\Ve_R$ and let $\bS\in\Seq_{R,\sv}$. Let $\chi_{\bS}$ be the character corresponding to $\bS$. We recall that $\delta_i-\delta_{i-1}=\alpha_i$. Therefore, by definition, we have
\[\frac{\qc_i(\bS)}{\qc_{i-1}(\bS)}=\chi_{\bS}(X^{\alpha_i})\ .\]
The lemma then follows from the fact that $r_j(\alpha_i)=\alpha_i$ if $m_{i,j}=2$ and $r_j(\alpha_i)=r_i(\alpha_j)$ if $m_{i,j}=3$.
\end{proof}

\begin{proposition}\label{prop-mon} Let $i,j\in\Ve_R$ and $\bS\in\Seq_{R,\sv}$. We have:
\begin{itemize}
\item[\rm{(i)}] if $m_{i,j}=2$: \qquad $\ovr_i\ovr_j(\bS)\neq\bze$\ \quad$\Longleftrightarrow$\quad\ $\bze\notin\{\ovr_j(\bS),\,\ovr_i(\bS)\}$\,;
\vskip .2cm
\item[\rm{(ii)}]  if $m_{i,j}=3$: \qquad $\ovr_i\ovr_j\ovr_i(\bS)\neq\bze$\ \quad $\Longleftrightarrow$\quad\ $\bze\notin\{\ovr_j(\bS),\,\ovr_i(\bS),\,\ovr_i\ovr_j(\bS),\,\ovr_j\ovr_i(\bS)\}$\,.
\end{itemize}
\end{proposition}
\begin{proof} \rm{(i)} Assume $m_{i,j}=2$. From the definitions, we have $\ovr_i\ovr_j(\bS)\neq\bze$ if and only if $\ovr_j(\bS)\neq\bze$ and $\displaystyle\frac{\qc_i\bigl(r_j(\bS)\bigr)}{\qc_{i-1}\bigl(r_j(\bS)\bigr)}\neq q^{\pm2}$. The assertion of item (i) follows then from (\ref{lem-cont}).

\vskip .2cm
\rm{(ii)} Assume $m_{i,j}=3$. From the definitions, we have $\ovr_i\ovr_j\ovr_i(\bS)\neq\bze$ if and only if 
\[\ovr_i(\bS)\neq\bze\ \ \text{and}\ \ \ovr_j\ovr_i(\bS)\neq\bze\ \ \text{and}\ \ \displaystyle\frac{\qc_i\bigl(r_jr_i(\bS)\bigr)}{\qc_{i-1}\bigl(r_jr_i(\bS)\bigr)}\neq q^{\pm2}\ .\]
From (\ref{lem-cont2}), the last condition is equivalent to $\ovr_j(\bS)\neq\bze$. Moreover, it also follows from (\ref{lem-cont2}) that, if $\ovr_i(\bS)\neq\bze$ and $\ovr_j(\bS)\neq\bze$, then $\ovr_j\ovr_i(\bS)\neq\bze\ \Leftrightarrow\ \ovr_i\ovr_j(\bS)\neq\bze$.
\end{proof}

As a consequence of the proposition, we obtain that the braid relations are satisfied by the truncated operators.
\begin{corollary}\label{cor-mon}
The operators $\ovr_i$, $i\in\Ve_R$, on the set $\Seq_{R,\sv}\cup\{\bze\}$ satisfy:
\begin{equation}\label{rel-mon}
\underbrace{\ovr_i\ovr_j\ovr_i\dots}_{m_{ij}\,terms}=\underbrace{\ovr_j\ovr_i\ovr_j\dots}_{m_{ij}\,terms}\quad\ \ \ \ \text{for $i,j\in\Ve_R$ with $i\neq j$}\,.
\end{equation} 
\end{corollary}
\begin{proof}
By definition, the braid relations are satisfied by the operators $r_i$, $i\in\Ve_R$, on $\Seq_{R,\sv}$. Due to the remark (\ref{r-ovr}) after Definition \ref{def-ovr}, it remains to check that, for any $i,j\in\Ve_R$ with $i\neq j$, we have
\begin{equation}\label{eq-lem-mon1}\underbrace{\ovr_i\ovr_j\ovr_i\dots}_{m_{ij}\,terms}(\bS)=\bze\ \quad\Leftrightarrow\quad\ \underbrace{\ovr_j\ovr_i\ovr_j\dots}_{m_{ij}\,terms}(\bS)=\bze\ .
\end{equation}
This is an immediate consequence of Proposition \ref{prop-mon}.
\end{proof}

\paragraph{\textbf{Truncated orbits of sequences.}} 

For any element $w\in W_0(R)$ in the Weyl group, let $w=r_{a_1}\dots r_{a_m}$ be a reduced expression for $w$ in terms of the simple reflections $r_i$, $i\in\Ve_R$. We then define 
\[\overline{w}:=\ovr_{a_1}\dots \ovr_{a_m}\ .\]
Due to Corollary \ref{cor-mon}, the element $\overline{w}$ does not depend on the reduced expression chosen for $w$. By convention $\overline{w}:=\text{Id}_{\Seq_{R,\sv}\cup\{\bze\}}$ if $w=1$.
\begin{definition}\label{def-oO}
For $\bS\in\Seq_{R,\sv}$, we define
\[\oO_{\bS}:=\left\{\bS'\in\Seq_{R,\sv}\ |\ \bS'=\overline{w}(\bS)\ \ \ \text{for $w\in W_0(R)$}\,\right\}\,,\]
and we call $\oO_{\bS}$ the \emph{truncated orbit} of $\bS$.
\end{definition}
Note that by definition $\bS\in\oO_{\bS}$ and $\bze\notin\oO_{\bS}$. Moreover, we have, for $\bS\in\Seq_{R,\sv}$ and $w\in W_0(R)$, 
\[\overline{w}(\bS)\neq\bze\ \ \quad\Longleftrightarrow\quad\ \ \overline{w^{-1}}\bigl(\overline{w}(\bS)\bigr)=\bS\ ;\]
this is immediate from Definition \ref{def-ovr} if $w$ is a simple reflection, and in general, it follows at once by induction on the length of $w$. As a direct consequence, for $\bS,\bS'\in\Seq_{R,\sv}$, we have either $\oO_{\bS}=\oO_{\bS'}$ or $\oO_{\bS}\cap\oO_{\bS'}=\emptyset$.

\begin{remarks}\label{rem-orb}
\textbf{(i)} For $\bS\in\Seq_{R,\sv}$ the truncated orbit $\oO_{\bS}$ is a subset of the orbit of $\bS$ under the action (not truncated) of $W_0(R)$. More precisely, the $W_0(R)$-orbit of $\bS$ is partitioned into truncated orbits. In particular, the truncated orbits are finite sets, namely, $|\oO_{\bS}|\leq |W_0(R)|$ for any $\bS\in\Seq_{R,\sv}$.

\textbf{(ii)} The truncated orbit $\oO_{\bS}$ of a sequence $\bS\in\Seq_{R,\sv}$ as defined in Definition \ref{def-oO} coincides with the set of all $\bS'\in\Seq_{R,\sv}$ which can be obtained by repeated applications of the truncated operators $\ovr_i$, $i\in\Ve_R$. In other words, an alternative definition of $\oO_{\bS}$ could be
\[\oO_{\bS}=\left\{\bS'\in\Seq_{R,\sv}\ |\ \bS'=\ovr_{a_1}\dots \ovr_{a_m}(\bS)\ \ \text{for $m\geq 0$ and $a_1,\dots,a_m\in\Ve_R$}\,\right\}\,.\]
This is a direct consequence of Corollary \ref{cor-mon} together with the fact that if $\ovr_i\bigl(\ovr_i(\bS)\bigr)\neq\bze$ then $\ovr_i\bigl(\ovr_i(\bS)\bigr)=\bS$.
\hfill$\triangle$
\end{remarks}

\subsection{Explicit formulas for the action of $W_0(R)$ on sequences.}\label{subsec-act-exp}

The action of the simple reflections $r_i$, $i\in\Ve_R$, of the Weyl group $W_0(R)$ on a sequence $\bS$ can be read off directly from Formulas (\ref{act-W1})--(\ref{act-W5}). For complete explicitness and due to their central role in the rest of the paper, we write down the formulas.

Let $\bS\in\Seq_{R,\sv}$ written as an array as in (\ref{string-array}) and let $i\in\Ve_R$. The sequence $r_i(\bS)$ is given by the following procedure:
\begin{itemize}
\item First, in $\bS$, we exchange $s_i$ and $s_{i-1}$. The resulting sequence is $r_i(\bS)$ whenever $i$ is not one of the following label: $2,\ \un{2},\ k+1,\ \unn{k+1}\ $.
\item Now, if $i\in\{2,\ \un{2},\ k+1,\ \unn{k+1}\}$, it means that the sequence, in the array representation (\ref{string-array}), ``ramifies" after $s_{i-1}$ (we read a sequence from left to right). Then the rule to calculate $r_i(\bS)$ is as follows. We still exchange $s_i$ and $s_{i-1}$. Then after the ramification, in the branch which contained $s_i$ nothing more happens, while all coefficients in the other branches, after the ramification, are multiplied by $\displaystyle\frac{s_i}{s_{i-1}}$.
\end{itemize}

We show explicitly the outcome of the above procedure, when $i\in\{2,\ \un{2},\ k+1,\ \unn{k+1}\}$. For clarity, we consider first the situation $k>1$. Then we have:
\begin{equation}\label{act-seq1}
r_2(\bS)=\left(s_{0} \,,s_{2}\,,\hspace{-0.2cm}\begin{array}{cclcc}s_{1}\,, & \dots & \dots\,, & \hspace{-0.1cm}s_{k} \,, & \!\!\!\!\begin{array}{cclc}
 s_{k+1} \,, & \dots & \dots\,, & \hspace{-0.2cm}s_{l}\\[0.8em]
s_{\unn{k+1}} \,, & \dots\,, & s_{\unn{l''}} &
\end{array}\\ 
\displaystyle\frac{s_{2}}{s_{1}}s_{\un{2}} \,, & \dots\,, & \displaystyle\frac{s_{2}}{s_{1}}s_{\un{l'}} & 
\end{array}\right)\ ,
\end{equation}
\begin{equation}\label{act-seq2}
r_{\un{2}}(\bS)=\left(s_{0} \,,s_{\un{2}}\,,\hspace{-0cm}\begin{array}{cclcc}\displaystyle\frac{s_{\un{2}}}{s_{1}}s_{2}\,, & \dots & \dots\,, & \hspace{-0.1cm}\displaystyle\frac{s_{\un{2}}}{s_{1}}s_{k} \,, & \!\!\!\!\begin{array}{cclc}
\displaystyle\frac{s_{\un{2}}}{s_{1}} s_{k+1} \,, & \dots & \dots\,, & \hspace{-0.2cm}\displaystyle\frac{s_{\un{2}}}{s_{1}}s_{l}\\[1.4em]
\displaystyle\frac{s_{\un{2}}}{s_{1}}s_{\unn{k+1}} \,, & \dots\,, & \displaystyle\frac{s_{\un{2}}}{s_{1}}s_{\unn{l''}} &
\end{array}\\ 
s_{1} \,, & \dots\,, & s_{\un{l'}} & 
\end{array}\right)\ ,
\end{equation}
\begin{equation}\label{act-seq3}
r_{k+1}(\bS)=\left(s_{0} \,,s_{1}\,,\hspace{-0.2cm}\begin{array}{cclcc}s_{2}\,, & \dots & \dots\,, & \hspace{-0.1cm}s_{k+1} \,, & \!\!\!\!\begin{array}{cccc}
 \hspace{-0.2cm}s_{k} \,, & \dots & \dots\,, & \hspace{-0.1cm}s_{l}\\[0.8em]
\displaystyle\frac{s_{k+1}}{s_{k}}s_{\unn{k+1}} \,, & \dots\,, & \displaystyle\frac{s_{k+1}}{s_{k}}s_{\unn{l''}} &
\end{array}\\ 
s_{\un{2}} \,, & \dots\,, & s_{\un{l'}} & 
\end{array}\right)\ ,
\end{equation}
and $r_{\unn{k+1}}(\bS)$ is obtained similarly to (\ref{act-seq3}), by exchanging the role of $\{k+1,\dots,l\}$ and $\{\unn{k+1},\dots,\unn{l''}\}$.

If $k=1$ the situation is more symmetrical. We have :
\begin{equation}\label{act-seq4}
r_2(\bS)=\left(s_{0} \,,s_{2}\,,\begin{array}{ccccc}\hspace{-0.1cm}s_{1}\,, & \dots & \dots & \dots\,, & \hspace{-0.1cm}s_{l} \\[1em]
\displaystyle\frac{s_{2}}{s_{1}}s_{\unn{2}} \,, & \dots\,, & \displaystyle\frac{s_{2}}{s_{1}}s_{\unn{l''}} \\[1em]
\displaystyle\frac{s_{2}}{s_{1}}s_{\un{2}} \,, & \dots & \dots\,, & \displaystyle\frac{s_{2}}{s_{1}}s_{\un{l'}} & 
\end{array}\right)\ ,
\end{equation}
and $r_{\un{2}}(\bS)$, $r_{\unn{2}}(\bS)$ are obtained similarly, by permuting the role of $\{2,\dots,l\}$, $\{\un{2},\dots,\un{l'}\}$ and $\{\unn{2},\dots,\unn{l''}\}$.

\section{Classical placed skew shapes and tableaux}\label{sec-tab-cla}

In this section, let $N\in\Z_{>0}$. We define the combinatorial notion of placed skew shapes and their associated tableaux. We note that we use the short terminology ``tableaux (of size $N$)" for tableaux associated to placed skew shapes (of size $N$).  

\subsection{Definitions.}\label{subsec-PSS}

\paragraph{\textbf{Skew diagrams.}}
Let $\lambda\vdash N$ be a partition of $N$, that is, $\lambda=(\lambda_1,\dots,\lambda_l)$ is a family of  integers such that $\lambda_1\geq\lambda_2\geq\dots\geq\lambda_l>0$ and $\lambda_1+\dots+\lambda_l=N$. We say that $\lambda$ is a partition {\em of size} $N$ and set $|\lambda|:=N$. 

 A pair $(x,y)\in\mathbb{Z}^2$ is called a {\em node}. For a node $\theta=(x,y)$, the classical content of $\theta$ is denoted by $\cc(\theta)$ and is defined by $\cc(\theta):=y-x$\,.

The (Young) diagram of $\lambda=(\lambda_1,\dots,\lambda_l)$ is the set of nodes $(x,y)$ such that $x\in\{1,\dots,l\}$ and $y\in\{1,\dots,\lambda_x\}$. The diagram of $\lambda$ will be represented in the plan by a left-justified array of $l$ rows such that the $j$-th row contains $\lambda_j$ nodes for all $j=1,\dots,l$ (a node will be pictured by an empty box). We number the rows from top to bottom. We identify partitions with their diagrams and say that $(x,y)$ is a node of $\lambda$, or $(x,y)\in\lambda$, if $(x,y)$ is a node of the diagram of $\lambda$.

A skew partition consists of two partitions $\mu,\lambda$ such that, as sets of nodes, $\mu\subset\lambda$. We denote it by $\lambda/\mu$. The skew diagram of $\lambda/\mu$ consists of the sets of nodes which are in $\lambda$ and not in $\mu$. For example,
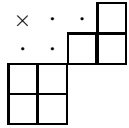
\[\begin{array}{cccc}
 \scriptstyle{\times} & \hspace{-0.35cm}\cdot & \hspace{-0.35cm}\cdot & \hspace{-0.35cm}\fbox{\phantom{\scriptsize{$2$}}} \\[-0.2em]
\cdot & \hspace{-0.35cm}\cdot & \hspace{-0.35cm}\fbox{\phantom{\scriptsize{$2$}}} & \hspace{-0.35cm}\fbox{\phantom{\scriptsize{$2$}}}\\[-0.2em]
\fbox{\phantom{\scriptsize{$2$}}} &\hspace{-0.35cm}\fbox{\phantom{\scriptsize{$2$}}} & &\\[-0.2em]
\fbox{\phantom{\scriptsize{$2$}}} &\hspace{-0.35cm}\fbox{\phantom{\scriptsize{$2$}}} & &
\end{array}\]
is the skew diagram corresponding to $\lambda/\mu$ with $\lambda=(4,4,2,2)$ and $\mu=(3,2)$. We sometimes put dots at some empty positions and we mark with $\times$ the point in $\Z^2$ with coordinate $(1,1)$ (we omit the symbol $\times$ only when $\mu$ is empty and when the top left box is in position $(1,1)$, which corresponds to a usual diagram). It will be convenient to also call a skew diagram any set of nodes (possibly with non-positive coordinates) which is the translated of a skew diagram along the diagonal. The size $|S|$ of a skew diagram $S$ is the number of nodes in $S$.

The connected components of a skew diagram $S$ are the minimal subsets of nodes $C_1,\dots,C_a\subset S$ such that $S=C_1\sqcup\dots\sqcup C_a$ (disjoint union) and two nodes $\theta_i$ and $\theta_j$ do not lie in the same diagonal nor in adjacent diagonals if $\theta_i\in C_i$ and $\theta_j\in C_j$ with $i\neq j$. In the example above, the skew diagram consists of the two connected components $\begin{array}{ll}
  & \hspace{-0.45cm}\Box \\[-0.7em]
\Box & \hspace{-0.45cm}\Box
\end{array} $ and $\begin{array}{ll}
 \Box & \hspace{-0.45cm}\Box \\[-0.7em]
\Box & \hspace{-0.45cm}\Box
\end{array} $.

We say that two skew diagrams are equivalent if one is obtained from the other by translating some connected components along the diagonal. For example, we have
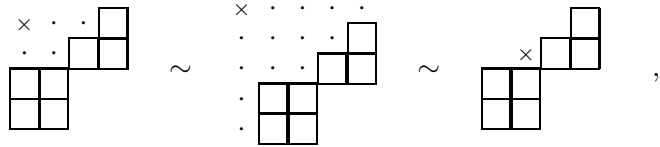
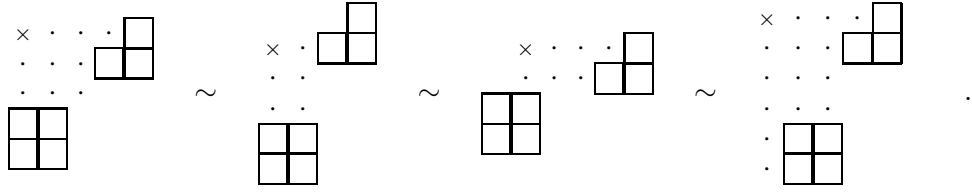
\[\begin{array}{cccc}
 \scriptstyle{\times} & \hspace{-0.35cm}\cdot & \hspace{-0.35cm}\cdot & \hspace{-0.35cm}\fbox{\phantom{\scriptsize{$2$}}} \\[-0.2em]
\cdot & \hspace{-0.35cm}\cdot & \hspace{-0.35cm}\fbox{\phantom{\scriptsize{$2$}}} & \hspace{-0.35cm}\fbox{\phantom{\scriptsize{$2$}}}\\[-0.2em]
\fbox{\phantom{\scriptsize{$2$}}} &\hspace{-0.35cm}\fbox{\phantom{\scriptsize{$2$}}} & &\\[-0.2em]
\fbox{\phantom{\scriptsize{$2$}}} &\hspace{-0.35cm}\fbox{\phantom{\scriptsize{$2$}}} & &
\end{array}
\ \ \sim\ \ 
\begin{array}{ccccc}
 \scriptstyle{\times} & \hspace{-0.25cm}\cdot & \hspace{-0.35cm}\cdot & \hspace{-0.35cm}\cdot & \hspace{-0.35cm}\cdot \\[-0.2em]
\cdot & \hspace{-0.25cm}\cdot & \hspace{-0.35cm}\cdot & \hspace{-0.35cm}\cdot & \hspace{-0.35cm}\fbox{\phantom{\scriptsize{$2$}}} \\[-0.2em]
\cdot &  \hspace{-0.25cm}\cdot & \hspace{-0.35cm}\cdot & \hspace{-0.35cm}\fbox{\phantom{\scriptsize{$2$}}} & \hspace{-0.35cm}\fbox{\phantom{\scriptsize{$2$}}}\\[-0.2em]
\cdot &  \hspace{-0.25cm}\fbox{\phantom{\scriptsize{$2$}}} &\hspace{-0.35cm}\fbox{\phantom{\scriptsize{$2$}}} & &\\[-0.2em]
\cdot &  \hspace{-0.25cm}\fbox{\phantom{\scriptsize{$2$}}} &\hspace{-0.35cm}\fbox{\phantom{\scriptsize{$2$}}} & &
\end{array}
\ \ \sim\ \ 
\begin{array}{cccc}
 & \hspace{-0.35cm} & & \hspace{-0.35cm}\fbox{\phantom{\scriptsize{$2$}}} \\[-0.2em]
 & \hspace{-0.35cm}\scriptstyle{\times} & \hspace{-0.35cm}\fbox{\phantom{\scriptsize{$2$}}} & \hspace{-0.35cm}\fbox{\phantom{\scriptsize{$2$}}}\\[-0.2em]
\fbox{\phantom{\scriptsize{$2$}}} &\hspace{-0.35cm}\fbox{\phantom{\scriptsize{$2$}}} & &\\[-0.2em]
\fbox{\phantom{\scriptsize{$2$}}} &\hspace{-0.35cm}\fbox{\phantom{\scriptsize{$2$}}} & &
\end{array}\ \ \ \ ,
\]
and
\[\begin{array}{ccccc}
 \scriptstyle{\times} & \hspace{-0.35cm}\cdot & \hspace{-0.35cm}\cdot &\hspace{-0.35cm}\cdot & \hspace{-0.35cm}\fbox{\phantom{\scriptsize{$2$}}} \\[-0.2em]
\cdot & \hspace{-0.35cm}\cdot & \hspace{-0.35cm}\cdot &\hspace{-0.35cm}\fbox{\phantom{\scriptsize{$2$}}} & \hspace{-0.35cm}\fbox{\phantom{\scriptsize{$2$}}}\\[-0.2em]
\cdot & \hspace{-0.35cm}\cdot & \hspace{-0.35cm}\cdot & &\\[-0.2em]
\fbox{\phantom{\scriptsize{$2$}}} &\hspace{-0.35cm}\fbox{\phantom{\scriptsize{$2$}}} & & &\\[-0.2em]
\fbox{\phantom{\scriptsize{$2$}}} &\hspace{-0.35cm}\fbox{\phantom{\scriptsize{$2$}}} & & &
\end{array}
\ \ \sim\ \ 
\begin{array}{cccc}
& &  & \hspace{-0.35cm}\fbox{\phantom{\scriptsize{$2$}}} \\[-0.2em]
\scriptstyle{\times} & \hspace{-0.35cm}\cdot & \hspace{-0.35cm}\fbox{\phantom{\scriptsize{$2$}}} & \hspace{-0.35cm}\fbox{\phantom{\scriptsize{$2$}}} \\[-0.2em]
\cdot & \hspace{-0.35cm}\cdot & &\\[-0.2em]
\cdot & \hspace{-0.35cm}\cdot & &\\[-0.2em]
\fbox{\phantom{\scriptsize{$2$}}} &\hspace{-0.35cm}\fbox{\phantom{\scriptsize{$2$}}} & &\\[-0.2em]
\fbox{\phantom{\scriptsize{$2$}}} &\hspace{-0.35cm}\fbox{\phantom{\scriptsize{$2$}}} & &
\end{array}
\ \ \sim\ \ 
\begin{array}{cccccc}
        &\hspace{-0.35cm}\scriptstyle{\times} & \hspace{-0.35cm}\cdot & \hspace{-0.35cm}\cdot &\hspace{-0.35cm}\cdot & \hspace{-0.35cm}\fbox{\phantom{\scriptsize{$2$}}} \\[-0.2em]
       &\hspace{-0.35cm}\cdot & \hspace{-0.35cm}\cdot & \hspace{-0.35cm}\cdot &\hspace{-0.35cm}\fbox{\phantom{\scriptsize{$2$}}} & \hspace{-0.35cm}\fbox{\phantom{\scriptsize{$2$}}}\\[-0.2em]
\fbox{\phantom{\scriptsize{$2$}}} &\hspace{-0.35cm}\fbox{\phantom{\scriptsize{$2$}}} & & & &\\[-0.2em]
\fbox{\phantom{\scriptsize{$2$}}} &\hspace{-0.35cm}\fbox{\phantom{\scriptsize{$2$}}} & & & &
\end{array}
\ \ \sim\ \ 
\begin{array}{ccccc}
 \scriptstyle{\times} & \hspace{-0.25cm}\cdot & \hspace{-0.35cm}\cdot &\hspace{-0.35cm}\cdot & \hspace{-0.35cm}\fbox{\phantom{\scriptsize{$2$}}} \\[-0.2em]
\cdot & \hspace{-0.25cm}\cdot & \hspace{-0.35cm}\cdot &\hspace{-0.35cm}\fbox{\phantom{\scriptsize{$2$}}} & \hspace{-0.35cm}\fbox{\phantom{\scriptsize{$2$}}}\\[-0.2em]
\cdot & \hspace{-0.25cm}\cdot & \hspace{-0.35cm}\cdot & &\\[-0.2em]
\cdot & \hspace{-0.25cm}\cdot & \hspace{-0.35cm}\cdot & &\\[-0.2em]
\cdot & \hspace{-0.25cm}\fbox{\phantom{\scriptsize{$2$}}} &\hspace{-0.35cm}\fbox{\phantom{\scriptsize{$2$}}} & &\\[-0.2em]
\cdot & \hspace{-0.25cm}\fbox{\phantom{\scriptsize{$2$}}} &\hspace{-0.35cm}\fbox{\phantom{\scriptsize{$2$}}} & &
\end{array}\ \ \ \ \ .
\]

\paragraph{\textbf{Placed skew shapes.}} Let $q$ be an indeterminate. 

A placed skew diagram $\bGamma$ consists of an integer $d\in\Z_{>0}$, a $d$-tuple of non-empty skew diagrams $\Gamma_1,\dots,\Gamma_d$ and a $d$-tuple of elements $\gamma_1,\dots,\gamma_d\in\C(q)\backslash\{0\}$ such that $\gamma_i\neq q^{2a}\gamma_j$ for any $i\neq j$ and $a\in\Z$. We write $\bGamma=\bigl\{(\Gamma_1,\gamma_1),\dots,(\Gamma_d,\gamma_d)\bigr\}$.

A placed skew diagram $\bGamma=\bigl\{(\Gamma_1,\gamma_1),\dots,(\Gamma_d,\gamma_d)\bigr\}$ will be represented by a $d$-tuple of skew diagrams $\Gamma_1,\dots,\Gamma_d$ with each place $\gamma_1,\dots,\gamma_d$ written above each corresponding skew diagram. The size $|\bGamma|$ of the placed skew diagram $\bGamma$ is the total number of nodes, that is, $|\bGamma|=|\Gamma_1|+\dots+|\Gamma_d|$. For example,
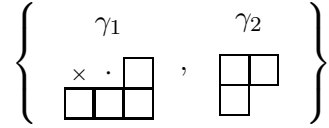
\[\left\{
\begin{array}{c}
\gamma_1 \\[0.5em]
\begin{array}{ccc}
\scriptstyle{\times} & \hspace{-0.35cm}\cdot & \hspace{-0.35cm}\fbox{\phantom{\scriptsize{$2$}}} \\[-0.2em]
\fbox{\phantom{\scriptsize{$2$}}} & \hspace{-0.35cm}\fbox{\phantom{\scriptsize{$2$}}} & \hspace{-0.35cm}\fbox{\phantom{\scriptsize{$2$}}}
\end{array}
\end{array}
,
\begin{array}{c}
\gamma_2 \\[0.5em]
\begin{array}{cc}
\fbox{\phantom{\scriptsize{$2$}}} & \hspace{-0.35cm}\fbox{\phantom{\scriptsize{$2$}}} \\[-0.2em]
\fbox{\phantom{\scriptsize{$2$}}} & 
\end{array}\\[-1.0em]
\ 
\end{array}
\right\}\]
represents a placed skew diagram of size 7.

We say that two placed skew diagrams of the form $\bigl\{(\Gamma_1,\gamma_1)\bigr\}$ and $\bigl\{(\Gamma'_1,\gamma'_1)\bigr\}$ are equivalent if $\gamma'_1=\gamma_1q^{2a}$ for some $a\in\Z$ and $\Gamma'_1$ is equivalent (as a skew diagram) to the translated of $\Gamma_1$ by $a$ steps to the left if $a\geq0$, and by $-a$ steps to the right if $a\leq0$. For example, we have
\begin{equation}\label{ex-eq}
\left\{\begin{array}{c}
\gamma_1 \\[0.5em]
\begin{array}{cc}
\fbox{\phantom{\scriptsize{$2$}}} & \hspace{-0.35cm}\fbox{\phantom{\scriptsize{$2$}}} \\[-0.2em]
\fbox{\phantom{\scriptsize{$2$}}} & 
\end{array}
\end{array}\right\}
\ \sim\ \ 
\left\{\begin{array}{c}
\gamma_1q^{-2} \\[0.5em]
\begin{array}{ccc}
\scriptstyle{\times} & \hspace{-0.25cm}\fbox{\phantom{\scriptsize{$2$}}} & \hspace{-0.35cm}\fbox{\phantom{\scriptsize{$2$}}} \\[-0.2em]
\cdot & \hspace{-0.25cm}\fbox{\phantom{\scriptsize{$2$}}} & 
\end{array}
\end{array}\right\}
\ \sim\ \ 
\left\{\begin{array}{c}
\gamma_1q^{2} \\[0.5em]
\begin{array}{cc}
\scriptstyle{\times} & \hspace{-0.35cm}\cdot \\[-0.2em]
\fbox{\phantom{\scriptsize{$2$}}} & \hspace{-0.35cm}\fbox{\phantom{\scriptsize{$2$}}} \\[-0.2em]
\fbox{\phantom{\scriptsize{$2$}}} & 
\end{array}
\end{array}\right\}\ \ \ \ \ ,
\end{equation}
where we noticed that, by equivalence under translation along the diagonal, translating a skew diagram to the left is equivalent to translating vertically downward. Note also that if $\gamma_1=\gamma'_1$ then $\bigl\{(\Gamma_1,\gamma_1)\bigr\}$ and $\bigl\{(\Gamma'_1,\gamma_1)\bigr\}$ are equivalent if and only if $\Gamma_1$ and $\Gamma'_1$ are equivalent as skew diagrams.

Then two placed skew diagrams $\bigl\{(\Gamma_1,\gamma_1),\dots,(\Gamma_d,\gamma_d)\bigr\}$ and $\bigl\{(\Gamma'_1,\gamma'_1),\dots,(\Gamma'_{d'},\gamma'_{d'})\bigr\}$ are said to be equivalent if $d=d'$ and there is some permutation $\pi$ such that $\bigl\{(\Gamma_a,\gamma_a)\bigr\}\sim\bigl\{(\Gamma'_{\pi(a)},\gamma'_{\pi(a)})\bigr\}$ for $a=1,\dots,d$. Note that the size of placed skew diagrams is conserved by the equivalence relation.
\begin{definition}\label{PSS-A}
A \emph{placed skew shape (of size $N$)} is an equivalence class of placed skew diagrams (of size $N$).\end{definition}

\paragraph{\textbf{Tableaux.}} A placed node is a pair $(\theta, \gamma)$, where $\theta\in\Z^2$ is a node and $\gamma\in\C(q)\backslash\{0\}$. A placed skew diagram can be seen as a set of placed nodes. A placed skew shape can also be seen as a set of placed nodes, up to equivalence as described above. 

A tableau (of shape $\bGamma$) is a bijection between the set of placed nodes of some placed skew shape $\bGamma$ and the set $\{0,1,\dots,|\bGamma|-1\}$. It is represented by filling the placed nodes of the skew diagrams representing $\bGamma$ by numbers $0,1,\dots,|\bGamma|-1$. The size $|\bT|$ of a tableau $\bT$ is the size of its shape.

\begin{remark}
One could define a tableau as a filling of the placed nodes of a placed skew diagram. Then our notion of tableaux corresponds to considering this definition up to equivalence, the equivalence being naturally induced from the equivalence of placed skew diagrams (the nodes are translated with their numbers in them). We will in all the paper consider placed skew diagrams and associated tableaux up to equivalence, very often without mentioning it anymore. For example, all our graphical representations of tableaux in later examples are to be considered up to equivalence.\hfill$\triangle$
\end{remark}

A tableau is standard if the numbers ascend along rows and down columns in every skew diagram. The notion of being standard or not for a tableau is well-defined since compatible with the equivalence relation on placed skew diagrams.

The $q$-content, or simply the content, of a placed node $\btheta=(\theta, \gamma)$ is $\qc(\btheta):=\gamma q^{2\cc(\theta)}=\gamma q^{2(y-x)}$, where $\theta=(x,y)$. Let $\bT$ be a tableau and, for $i=0,1,\dots,|\bT|-1$, let $\btheta_i$ be the placed node with number $i$. We set $\qc_i(\bT):=\qc(\btheta_i)$ for $i=0,1,\dots,|\bT|-1$. The contents $\qc_i(\bT)$ are well-defined since, for a placed node of a placed skew diagram, its content is conserved by the equivalence relation of placed skew diagram (see example (\ref{ex-eq})).

We call $\bigl(\qc_0(\bT),\dots,\qc_{|\bT|-1}(\bT)\bigr)$ the sequence of contents of $\bT$ and denote it by $\Seq(\bT)$. For example, the tableau
\[\left\{
\begin{array}{c}
\gamma_1 \\[0.5em]
\begin{array}{ccc}
\scriptstyle{\times} & \hspace{-0.35cm}\cdot & \hspace{-0.35cm}\fbox{\scriptsize{$1$}} \\[-0.2em]
\fbox{\scriptsize{$0$}} & \hspace{-0.35cm}\fbox{\scriptsize{$2$}} & \hspace{-0.35cm}\fbox{\scriptsize{$5$}}
\end{array}
\end{array}
,
\begin{array}{c}
\gamma_2 \\[0.5em]
\begin{array}{cc}
\fbox{\scriptsize{$3$}} & \hspace{-0.35cm}\fbox{\scriptsize{$6$}} \\[-0.2em]
\fbox{\scriptsize{$4$}} & 
\end{array}
\end{array}
\right\}\]
is standard and its sequence of contents is $\Seq(\bT)=(\gamma_1q^{-2},\,\gamma_1q^4,\,\gamma_1,\,\gamma_2,\,\gamma_2q^{-2},\,\gamma_1q^2,\,\gamma_2q^2)$.

For a placed skew shape $\bGamma$, we denote by $\STab(\bGamma)$ the set of standard tableaux of shape $\bGamma$, and we set
\[\STab(N):=\{\text{standard tableaux of size $N$}\}\ .\]

\subsection{Standard tableaux and sequences of contents}

Recall that $\Seq_N:=\bigl(\C(q)^{\times}\bigr)^N$. We recall the following characterization of standard tableaux in terms of their sequences of contents and we sketch a proof  (see, for example, \cite[Lemma 2.2 and parts of Theorem 4.1]{Ra}). 
\begin{proposition}\label{prop-tab-A}
The set $\STab(N)$ is in bijection with the set of sequences $(a_1,\dots,a_N)\in\Seq_N$ satisfying:
\[\text{($\star$)\ \  for any $i,j=1,\dots,N$ with $i<j$, if $a_i=a_j$ then $\{a_iq^2,a_iq^{-2}\}\subset\{a_{i+1},\dots,a_{j-1}\}$\,.}\]
\end{proposition}
\begin{proof}
We define a map $\phi$ from the set $\STab(N)$ to the set $\Seq_N$ by
\begin{equation}\label{bij-A}
\STab(N)\ni\bT\ \ \longmapsto\ \ \Seq(\bT)\in\Seq_N\ .
\end{equation}
For any $\bT\in\STab(N)$, the sequence $\phi(\bT)$ satisfies Condition ($\star$) because the following configurations $\begin{array}{ll}
 \Box & \hspace{-0.35cm}\cdot \\[-0.7em]
\,\cdot& \hspace{-0.45cm}\Box
\end{array} $, $\begin{array}{ll}
 \Box & \hspace{-0.45cm}\Box \\[-0.7em]
\,\cdot & \hspace{-0.45cm}\Box
\end{array} $ and $\begin{array}{ll}
 \Box & \hspace{-0.35cm}\cdot \\[-0.7em]
\Box & \hspace{-0.45cm}\Box
\end{array} $ can not appear in a skew diagram and because, in a standard tableau, whenever $\begin{array}{ll}
 \fbox{\scriptsize{$i_1$}} & \hspace{-0.35cm}\fbox{\scriptsize{$i_2$}} \\[-0.1em]
\fbox{\scriptsize{$i_3$}} & \hspace{-0.35cm}\fbox{\scriptsize{$i_4$}}
\end{array}$ appears, we must have $i_1<i_2,i_3<i_4$.

Next, we construct by induction on $N$ a map $\psi$ from the set of sequences in $\Seq_N$ satisfying Condition ($\star$) to the set $\STab(N)$ such that $\psi\circ\phi$ and $\phi\circ\psi$ are identity maps. 

Let $N=1$. We define $\psi$ by mapping $a_1\in\C(q)^{\times}$ to the standard tableau $\left\{\begin{array}{c} a_1\\[0.5em]\fbox{\scriptsize{$1$}}\end{array}\right\}$. We obviously have that $\phi\circ\psi$ is the identity map and moreover that $\psi(a_1)$ is the unique element of $\STab(1)$ with sequence of contents $(a_1)$. Therefore, we also have that $\psi\circ\phi$ is the identity map of $\STab(1)$.

Now we fix $(a_1,\dots,a_N)\in \Seq_N$ with $N>1$ satisfying ($\star$). Assume by induction that we have a unique element $\bT^{(N-1)}\in\STab(N-1)$ with sequence of contents $(a_1,\dots,a_{N-1})$. We construct below a standard tableau $\bT$ such that $\bT$ is obtained from $\bT^{(N-1)}$ by adding a placed node with the number $N$ in it and with content $a_N$. We treat several cases. For each case, it is straightforward to check  that $\bT$ is the unique element of $\Tab(N)$ with the required property. This implies, by induction hypothesis, that $\bT$ is the unique element of $\STab(N)$ with content string $(a_1,\dots,a_N)$. Then, setting $\psi(a_1,\dots,a_N):=\bT$, we conclude that $\psi\circ\phi$ and $\phi\circ\psi$ are identity maps.

\emph{Case (1).} Assume that, for any $i=1,\dots,N-1$, we have $a_i\notin\{q^{-2}a_N,a_N,q^2a_N\}$. To construct $\bT$, we add a new connected component $\fbox{\scriptsize{$N$}}$ in one of the diagram of $\bT^{(N-1)}$ if $a_N=q^{2b}a_i$ for some $i\in\{1,\dots,N-1\}$ and some $b\in\Z$; otherwise we set $\bT=\bT^{(N-1)}\cup\left\{\begin{array}{c} a_N\\[0.5em]\fbox{\scriptsize{$N$}}\end{array}\right\}\,$.

\emph{Case (2).} Assume that there are some $i,j\in\{1,\dots,N-1\}$ such that $a_i=q^{-2}a_N$ and $a_j=q^{2}a_N$. We choose the largest $i$ and $j$ with this property. Then one of the diagram of $\bT^{(N-1)}$ contains a part of the form $\begin{array}{ll}
 & \hspace{-0.35cm}\fbox{\scriptsize{$j$}} \\[-0.1em]
\fbox{\scriptsize{$i$}} & \hspace{-0.20cm}\cdot
\end{array}$, where there is no box in the position marked by the dot. To construct $\bT$, we add a placed node with number $N$ in the position marked by the dot.

\emph{Case (3).} Assume that, for any $j\in\{1,\dots,N-1\}$, we have $a_j\neq q^{2}a_N$, and that there is some $i\in\{1,\dots,N-1\}$ such that $a_i=q^{-2}a_N$. Then, by Condition ($\star$), we have $a_j\neq a_N$ for any $j\in\{1,\dots,N-1\}$. We claim that we can assume that $\bT^{(N-1)}$ contains a part of the form $\begin{array}{ll}
 \cdot & \hspace{-0.20cm}\cdot  \\[-0.1em] 
  \cdot & \hspace{-0.20cm}\cdot \\[-0.1em]
\fbox{\scriptsize{$i$}} & \hspace{-0.20cm}\cdot
\end{array}$, where there is no box above $\fbox{\scriptsize{$i$}}$ in the same column and no box in the next column to the right of $\fbox{\scriptsize{$i$}}$\,. Indeed, as $a_j\neq a_N$ for any $j\in\{1,\dots,N-1\}$, the positions just above $\fbox{\scriptsize{$i$}}$ and just to the right of $\fbox{\scriptsize{$i$}}$ must be both empty. In turn, as $\bT^{(N-1)}$ is a tableau, every position above $\fbox{\scriptsize{$i$}}$ in the same column must be empty as well, and so must be every position below the line of $\fbox{\scriptsize{$i$}}$ in the column right to it. As moreover $a_j\neq q^{2}a_N$ for any $j\in\{1,\dots,N-1\}$, we have that, if there is a box, say $\fbox{\scriptsize{$k$}}$\,, in the next column to the right of $\fbox{\scriptsize{$i$}}$ then it must lie not lower than the second line above $\fbox{\scriptsize{$i$}}$\,. Therefore we can translate the connected component containing $\fbox{\scriptsize{$k$}}$ one step to the south-east and thus we have $\bT^{(N-1)}$ of the desired form. Finally, to construct $\bT$, we add a placed node with number $N$ in the position just to the right of $\fbox{\scriptsize{$i$}}$ in the same line.

\emph{Case (4).} Assume that there is some $i\in\{1,\dots,N-1\}$ such that $a_i=q^{2}a_N$, and that, for any $j\in\{1,\dots,N-1\}$, we have $a_j\neq q^{-2}a_N$. This situation is equivalent to Case (3) by reflection through the main diagonal.
\end{proof}
\begin{corollary}\label{cor-tab-A}
Two standard tableaux are the same if and only if their sequences of contents coincide.
\end{corollary}

\subsection{Calibrated spectrum of $H(GL_{N+1})$}\label{subsec-rep-GL}

We consider the affine Hecke algebra $H(GL_{N+1})$ (corresponding to $R=A_N$ and $\sv$ an extremity); see Section \ref{subsec-GL}. We use the ordered basis $\cB_{A_N,1}=(\delta_{0},\delta_1,\dots,\delta_N)$ of the $\Z$-module $L_{A_N,1}$ to describe the calibrated spectrum. As explained in Section \ref{subsec-prel3}, given the basis $\cB_{A_N,1}$, a character in the calibrated spectrum $\text{C-Spec}\bigl(H(GL_{N+1})\bigr)$ is identified with a sequence of eigenvalues in $\text{C-Eig}(\cB_{A_N,1})\subset\Seq_{N+1}$.

We have the following description of the calibrated spectrum of the affine Hecke algebras $H(GL_{N+1})$ proved in \cite{Ra} (see also \cite{Ch1,Ch2}). We will use it as a first step to prove a generalisation for general $(R,\sv)$ in Section \ref{subsec-cal}.
\begin{theorem}\label{theo-GL}
The calibrated spectrum $\text{C-Spec}\bigl(H(GL_{N+1})\bigr)$ is in bijection with the set of standard tableaux of size $N+1$. More precisely, we have
\begin{equation}\label{phi-Sn}
\text{C-Eig}(\cB_{A_N,1})=\{\Seq(\bT)\ |\ \bT\in\STab(N+1)\}\ .
\end{equation}
\end{theorem}
The inclusion from right to left follows from an explicit construction  \cite{Ra} of calibrated representations of $\C(q)H(GL_{N+1})$. This construction is going to be a special case of the construction in Section \ref{sec-rep} for arbitrary $(R,\sv)$. A proof of the other inclusion can be found in \cite[``Step 4" of the proof of Theorem 4.1]{Ra}.

\section{Tableaux of type $(R,\sv)$ and admissibility}\label{sec-tab}

In this Section, we come back to our general setting, where we fix $R$ and $\sv$ as in Sections \ref{sec-def}-\ref{sec-act}.

\subsection{Tableaux of type $(R,\sv)$}

Let $\bT$ be a tableau of size $N+1$ and $m\in\{0,1,\dots,N\}$. We denote by $\bT^{\downarrow m}$ the tableau of size $m+1$ obtained from $\bT$ by keeping only the placed nodes containing $0,1,\ldots,m$.

\begin{definition}\label{def-tab-R} A \emph{tableau of type $(R,\sv)$} is a triplet $(\bT_1,\bT_2,\bT_3)$ such that:
\begin{itemize}
\vspace{-0.1cm}
\item $\bT_1$, $\bT_2$ and $\bT_3$ are tableaux of size, respectively, $l+1$, $l'+1$ and $l''+1$;
\item $\bT_1^{\downarrow 1}$ and $\bT_2^{\downarrow 1}$ coincide;
\item $\bT_1^{\downarrow k}$ and $\bT_3^{\downarrow k}$ coincide.
\end{itemize}
A tableau $(\bT_1,\bT_2,\bT_3)$ of type $(R,\sv)$ is standard if $\bT_1$, $\bT_2$ and $\bT_{3}$ are standard. We set:
\[\STab(R,\sv):=\{\text{standard tableaux of type $(R,\sv)$}\}\ .\]
\end{definition}

For a tableau $\bT=(\bT_1,\bT_2,\bT_3)$ of type $(R,\sv)$, we define the contents of $\bT$ as follows:
\begin{equation}\label{content}
\begin{array}{ll}
\qc_{a}(\bT):=\qc_a(\bT_1)\,,\quad & \text{for $a=0,1,\dots,l$,}\\[0.4em]
\qc_{\un{b}}(\bT):=\qc_b(\bT_2)\,, & \text{for $b=2,\dots,l'$,}\\[0.4em]
\qc_{\unn{c}}(\bT):=\qc_c(\bT_3)\,, & \text{for $c=k+1,\dots,l''$.}
\end{array}
\end{equation}
We define the sequence of contents $\bT$, and denote it by $\Seq(\bT)$, to be the following ordered collection of the contents of $\bT$:
\begin{equation}\label{content-string}
\bigl(\qc_0(\bT),\qc_{1}(\bT),\ldots,\ldots,\qc_{l}(\bT),\qc_{\un{2}}(\bT),\ldots,\qc_{\un{l'}}(\bT),\qc_{\unn{k+1}}(\bT),\ldots,\qc_{\unn{l''}}(\bT)\bigr)\ .
\end{equation}
Thus, the sequence of contents $\Seq(\bT)$ is naturally a sequence in $\Seq_{R,\sv}$ as defined in Section \ref{sec-act}. In particular, if $\bS=\Seq(\bT)$, then we have $\qc_j(\bS)=\qc_j(\bT)$ for any $j\in\Ve_R\cup\{0\}$.

\paragraph{\textbf{Graphical representation.}} We will represent a tableau $\bT=(\bT_1,\bT_2,\bT_3)$ of type $(R,\sv)$ by a triplet made of the usual representations of $\bT_1$, $\bT_2$ and $\bT_3$ with the following convention:
\vspace{-0.1cm}
\begin{itemize}
\item in $\bT_1$, the numbers in the placed nodes are denoted $0,1,\dots,k,k+1,\ldots,l$\,;
\item in $\bT_2$, the numbers in the placed nodes are denoted $0,1,\un{2},\dots,\un{l'}$\,;
\item in $\bT_3$, the numbers in the placed nodes are denoted $0,1,\dots,k,\unn{k+1},\dots,\unn{l''}$\,.
\end{itemize}
With this convention and according to Definition \ref{def-tab-R}, if two placed nodes of $\bT$ contain the same label, they coincide. Thus, for any $j\in\Ve_R\cup\{0\}$, the content $\qc_j(\bT)$ is the content of the placed node of $\bT$ containing the label $j$.

\begin{remarks}\label{rem-conv}
\textbf{(i)} Assume that we have $\{\un{2},\dots,\un{l'}\}=\emptyset$, that is, $l'=1$. By definition, $\bT_2$ is completely determined by $\bT_1$, namely $\bT_2=\bT_1^{\downarrow 1}$. In this situation, we will omit to represent $\bT_2$.

\textbf{(ii)} Similarly, assume that we have $\{\unn{k+1},\dots,\unn{l''}\}=\emptyset$, that is $l''=k=l$. In this situation,  $\bT_3=\bT_1$ by definition and we will omit to represent $\bT_3$.

\textbf{(iii)} Assume that we are in the $GL_{N+1}$ situation, that is, $R=A_N$ and $\sv$ is an extremity of the Dynkin diagram. Then both assumptions of items \textbf{(i)} and \textbf{(ii)} are satisfied, and a tableau of this type is simply a tableau as in Section \ref{sec-tab-cla}.
\hfill$\triangle$
\end{remarks}

\subsection{Action of $W_0(R)$ on $\STab(R,\sv)$ and admissible standard tableaux.}\label{subsec-adm}

Let $\bS=\bigl(s_{0},s_{1},\ldots,\ldots,s_{l},\,s_{\un{2}},\ldots,s_{\un{l'}},\,s_{\unn{k+1}},\ldots,s_{\unn{l''}}\bigr)\in\Seq_{R,\sv}$. We define 
\begin{equation}\label{substrings}
\bS_1:=(s_{0},s_{1},\dots,s_{l})\,,\ \ \ \bS_2:=(s_{0},s_{1},s_{\un{2}},\dots,s_{\un{l'}})\,,\ \ \ \bS_3:=(s_{0},s_{1},\dots,s_{k},s_{\unn{k+1}},\dots,s_{\unn{l''}})\,.
\end{equation}
and call them the substrings of $\bS$. Note that they appear naturally in the array representation  (\ref{string-array}) of $\bS$.

Assume that the sequence $\bS$ is a sequence of contents $\bS=\Seq(\bT)$, for a tableau $\bT=(\bT_1,\bT_2,\bT_3)$ of type $(R,\sv)$. Then, the substrings $\bS_1,\bS_2,\bS_3$ are, respectively, the sequences of contents of the tableaux $\bT_1,\bT_2,\bT_3$. According to Proposition \ref{prop-tab-A} and Corollary \ref{cor-tab-A}, we have:
\begin{itemize}
\item a sequence $\bS\in\Seq_{R,\sv}$ is the sequence of contents of a standard tableau of type $(R,\sv)$ if and only if the three substrings $\bS_1,\bS_2,\bS_3$ satisfy the condition $(\star)$ of Proposition \ref{prop-tab-A};
\item a standard tableau of type $(R,\sv)$ is uniquely characterised by its sequence of contents.
\end{itemize}

From now on, we very often identify a standard tableau $\bT$ of type $(R,\sv)$ with its sequence of contents $\Seq(\bT)\in\Seq_{R,\sv}$. Recall that we have, for $i\in\Ve_R$, operators $r_i$ and truncated operators $\ovr_i$ on $\Seq_{R,\sv}$. We will abuse notation and write $r_i(\bT)$ and $\ovr_i(\bT)$ instead of $r_i\bigl(\Seq(\bT)\bigr)$ and $\ovr_i\bigl(\Seq(\bT)\bigr)$. We will also write $\oO_{\bT}$ for $\oO_{\Seq(\bT)}$ and call it the truncated orbit of $\bT$. 

We stress that, by definition, $r_i(\bT)$ (and $\ovr_i(\bT)$ if not $\bze$) is a sequence in $\Seq_{R,\sv}$. It does not have to be a sequence of contents of a tableau of type $(R,\sv)$, and in general, it is not. Therefore, we make the following definition.

\begin{definition}\label{def-adm-tab}
A standard tableau $\bT\in\STab(R,\sv)$ of type $(R,\sv)$ is called \emph{admissible} if its truncated orbit $\oO_{\bT}$ only consists of sequences of contents $\Seq(\bT')$ with $\bT'\in\STab(R,\sv)$. We set:
\[\ASTab(R,\sv):=\{\bT\in\STab(R,\sv)\ \text{such that $\bT$ is admissible}\}\ .\]
If $\bT\in\STab(R,\sv)$ is admissible, then $\oO_{\bT}$ is called an admissible truncated orbit of type $(R,\sv)$. 
\end{definition}

\paragraph{\textbf{The action graphically.}}
The action of the operators $r_i$ (or $\ovr_i$), $i\in\Ve_R$, on a tableau $\bT=(\bT_1,\bT_2,\bT_3)$ of type $(R,\sv)$ is interpreted graphically as follows. This is a direct reformulation of the explicit action given in Subsection \ref{subsec-act-exp}. Let $u\in\{1,2,3\}$. To give how $\bT_u$ is transformed by the action of $r_i$, there are three different situations to consider:
\vspace{-0.1cm}
\begin{itemize}
\item If $\bT_{u}$ contains two placed nodes containing respectively $i$ and $i-1$, then they are exchanged.
\item If $\bT_{u}$ contains neither a placed node containing $i$ nor a placed node containing $i-1$, then nothing happens in $\bT_{u}$.
\item If $\bT_{u}$ contains a placed node containing $i-1$ but no placed node containing $i$, then, first, the placed node with $i-1$ moves to the diagonal with content $\qc_i(\bT)$; second, all the placed nodes $\alpha$ with numbers greater than $i-1$ move such that each ratio $\qc(\alpha)/\qc_{i-1}(\bT)$ (\emph{i.e.} the axial distance) is conserved.
\end{itemize}
Note that the third situation only occurs if $i\in\{2,\un{2},k+1,\unn{k+1}\}$. 

Finally, we note that the procedure above gives $\ovr_i(\bT)$ when it is different from $\bze$, and moreover, it is very easy to determine when $\ovr_i(\bT)=\bze$: this happens if and only if the placed nodes containing $i$ and $i-1$ are in adjacent diagonals in the same diagram. 

\begin{remark}\label{rem-act}
Starting from $\bT\in\STab(R,\sv)$ and its graphical representation, we note that there is no ambiguity in the graphical representation of $r_i(\bT)$ if the sequence $r_i(\bT)$ corresponds again to a standard tableau of type $(R,\sv)$ (because then it corresponds to a unique one, as recalled earlier). If $r_i(\bT)$ is not a sequence of contents of a standard tableau of type $(R,\sv)$, we can still interpret it and draw it as a sequence of placed nodes but there is no canonical way to do it (because fixing the content of a placed node only fixes the diagonal to which it belongs). In this situation, this ambiguity is not important since only the fact that $r_i(\bT)$ does not correspond to a standard tableau is relevant and this is easily seen from any representation of $r_i(\bT)$ (see examples below).\hfill$\triangle$
\end{remark}

\subsection{Examples of tableaux and truncated orbits}\label{subsec-ex}

\paragraph{\textbf{(a) $R=A_n$ and $\sv=1$ in the standard labelling.}} This is the situation corresponding to $GL_{n+1}$  as in Section \ref{subsec-GL}. The Weyl group is the symmetric group $S_{n+1}$. We already noted in Remark \ref{rem-conv} that, in this situation, a tableau of type $(R,\sv)$ is a usual tableau (of size $n+1$) as in Section \ref{sec-tab-cla}. For any $i\in\Ve_R=\{1,\dots,n\}$, the action of $r_i$ on a tableau $\bT$ is given simply by exchanging the placed nodes with number $i$ and $i-1$.

In this particular situation, the shape of a tableau remains invariant under the action of the Weyl group $S_{n+1}$. Recall that, given a placed skew shape $\bGamma$, the set $\STab(\bGamma)$ is the set of standard tableaux of shape $\bGamma$. The following is a reformulation, in our setting, of classical combinatorial properties of tableaux.
\begin{proposition}\label{prop-tab-A2}
Let $\bGamma$ be a placed skew shape and $\bT\in\STab(\bGamma)$. We have:
\[\oO_{\bT}=\STab(\bGamma)\ .\]
In particular, every standard tableau is admissible and the admissible truncated orbits are parametrised by placed skew shapes.
\end{proposition}
\begin{proof} Let $\bGamma$ be a placed skew shape of size $n+1$ and $\bT$ a standard tableau of shape $\bGamma$.

Let $w\in S_{n+1}$ such that $\overline{w}(\bT)\neq\bze$. We need to prove that $\overline{w}(\bT)$ is standard (it is obviously a tableau of shape $\bGamma$). Assume that this is not true. Write $\overline{w}(\bT)=\ovr_{a_m}\dots \ovr_{a_1}(\bT)$, where $r_{a_m}\dots r_{a_1}$ is a reduced expression of $w$ in terms of $r_1,\dots,r_n$. Take $p$ such that $\ovr_{a_{p-1}}\dots \ovr_{a_1}(\bT))$ is standard and $\ovr_{a_{p}}\dots \ovr_{a_1}(\bT)$ is not. Then $\ovr_{a_{p-1}}\dots \ovr_{a_1}(\bT)$ is a sequence satisfying Condition $(\star)$ in Proposition \ref{prop-tab-A} and $\ovr_{a_{p}}\dots \ovr_{a_1}(\bT)$ is not. This is possible only if $\qc_{a_p}(\bT)=q^{\pm2}\qc_{a_p-1}(\bT)$. However, this implies $\ovr_{a_{p}}\dots \ovr_{a_1}(\bT)=\bze$, and in turn $\overline{w}(\bT)=\bze$, which is a contradiction.

Now let $\bT'\in\STab(\bGamma)$. It remains to show that $\bT'\in\oO_{\bT}$. We will use induction on $n$. If $n=0$ this is immediate. Assume that $n>0$. Let $\btheta$ (respectively, $\btheta'$) be the placed node of $\bT$ (respectively, $\bT'$) containing $n$. As $\bT$ is standard, $\btheta$ must be the rightmost node of its line and the lowest node of its column, and similarly for $\btheta'$. In particular, $\btheta$ and $\btheta'$ can not be in adjacent diagonals, and therefore $\qc(\btheta')\neq q^{\pm2}\qc(\btheta)$. Recall that $\bT^{\downarrow n-1}$ denotes the standard tableau of size $n$ obtained from $\bT$ by removing $\btheta$, and $\bT'^{\downarrow n-1}$ the one obtained from $\bT'$ by removing $\btheta'$.
\begin{itemize}
\item If $\btheta=\btheta'$ then $\bT^{\downarrow n-1}$ and $\bT'^{\downarrow n-1}$  are of the same shape and, by induction hypothesis, there is $w\in S_n$ such that $\overline{w}(\bT^{\downarrow n-1})=\bT'^{\downarrow n-1}$, and in turn such that $\overline{w}(\bT)=\bT'$.
\item If $\btheta\neq\btheta'$ then we proceed in three steps:
\begin{itemize}
\item First, note that $\btheta'$ is a placed node of $\bT^{\downarrow n-1}$ which is the rightmost node of its line and the lowest node of its column. So there exists a standard tableau of the same shape than $\bT^{\downarrow n-1}$ with $n-1$ in $\btheta'$. By induction hypothesis, we can obtain it from $\bT^{\downarrow n-1}$ by applying $\overline{w_1}$ for some $w_1\in S_n$.
\item Then, in $\overline{w_1}(\bT)$, the number $n$ is contained in $\btheta$ and $n-1$ in $\btheta'$. So $\ovr_n\overline{w_1}(\bT)\neq\bze$, and in $\ovr_n\overline{w_1}(\bT)$, the number $n$ is contained in $\btheta'$.
\item Finally, $\ovr_n\overline{w_1}(\bT)^{\downarrow n-1}$ is of the same shape than $\bT'^{\downarrow n-1}$ and therefore, by induction hypothesis, we have  $\bT'^{\downarrow n-1}=\overline{w_2}\bigl(\ovr_n\overline{w_1}(\bT)^{\downarrow n-1}\bigr)$ for some $w_2\in S_n$.
\end{itemize}
We conclude that $\bT'=\overline{w_2}\,\ovr_n\overline{w_1}(\bT)$.
\end{itemize}
This concludes the verification of $\oO_{\bT}=\STab(\bGamma)$ and the proof of the proposition.
\end{proof}

\paragraph{\textbf{(b) $R=D_n$ ($n\geq4$) and $\sv=1$ in the standard labelling.}} The new labelling of the vertices of the Dynkin diagram is (with $l=n-1$):
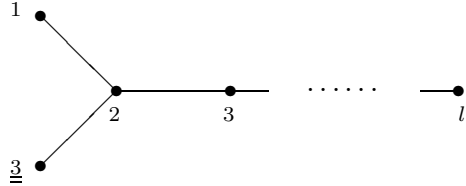
\begin{center}
\setlength{\unitlength}{0.01cm}
\begin{picture}(1550,210)(0,850)
\put(460,1050){\scriptsize{$1$}}
\put(500,1050){\circle*{15}}
\put(500,1050){\line(1,-1){100}}
\put(460,840){\scriptsize{$\unn{3}$}}
\put(500,850){\circle*{15}}
\put(500,850){\line(1,1){100}}
\put(600,950){\circle*{15}}
\put(590,910){\scriptsize{$2$}}
\put(600,950){\line(1,0){150}}
\put(750,950){\circle*{15}}
\put(740,910){\scriptsize{$3$}}
\put(750,950){\line(1,0){50}}
\put(850,950){$\ldots$}
\put(900,950){$\ldots$}
\put(1000,950){\line(1,0){50}}
\put(1050,950){\circle*{15}}
\put(1050,910){\scriptsize{$l$}}
\end{picture}
\end{center}
We show two examples of standard tableaux of type $(R,\sv)$:
\[
\left(\Biggl\{\begin{array}{c}
\gamma_1 \\[0.5em]
\begin{array}{cccc}
\scriptstyle{\times} & \hspace{-0.35cm}\cdot & \hspace{-0.35cm}\fbox{\scriptsize{$2$}} & \hspace{-0.35cm}\fbox{\scriptsize{$3$}} \\[-0.2em]
\fbox{\scriptsize{$0$}} & \hspace{-0.35cm}\fbox{\scriptsize{$1$}} & \hspace{-0.35cm}\fbox{\scriptsize{$4$}} &
\end{array}
\end{array}\Biggr\}\ ,\ 
\Biggl\{\begin{array}{c}
\gamma_1 \\[0.5em]
\begin{array}{ccc}
\scriptstyle{\times} & \hspace{-0.35cm}\cdot & \hspace{-0.35cm}\fbox{\scriptsize{$2$}} \\[-0.2em]
\fbox{\scriptsize{$0$}} & \hspace{-0.35cm}\fbox{\scriptsize{$1$}} & \hspace{-0.35cm}\fbox{\scriptsize{$\unn{3}$}}
\end{array}
\end{array}\Biggr\}
\right)\ ,\ \ 
\left(\Biggl\{\begin{array}{c}
\gamma_1 \\[0.5em]
\begin{array}{cccc}
\fbox{\scriptsize{$0$}} & \hspace{-0.35cm}\fbox{\scriptsize{$1$}} & \hspace{-0.35cm}\fbox{\scriptsize{$2$}} & \hspace{-0.35cm}\fbox{\scriptsize{$3$}}\\[-0.2em]
&&&\end{array}
\end{array},
\begin{array}{c}
\gamma_2 \\[0.5em]
\begin{array}{c}
\fbox{\scriptsize{$4$}}\\[-0.2em]
\ 
\end{array}
\end{array}\Biggr\}\ ,\ 
\Biggl\{\begin{array}{c}
\gamma_1 \\[0.5em]
\begin{array}{cccc}
\fbox{\scriptsize{$0$}} & \hspace{-0.35cm}\fbox{\scriptsize{$1$}} & \hspace{-0.35cm}\fbox{\scriptsize{$2$}} & \hspace{-0.35cm}\fbox{\scriptsize{$\unn{3}$}}\\[-0.2em]
&&&\end{array}
\end{array}\Biggr\}
\right)
\]
where $\gamma_1,\gamma_2\in\C(q)^{\times}$ are two places, and so satisfy by definition $\gamma_2\neq\gamma_1 q^{2a}$ for any $a\in\Z$. The associated sequences of contents are $\bigl(\gamma_1q^{-2},\gamma_1,\gamma_1 q^4,\!\begin{array}{ll}\gamma_1 q^6\,, &\!\!\!\!\gamma_1q^2\\[0.2em] \gamma_1 q^2 & \end{array}\!\!\!\bigr)$ and $\bigl(\gamma_1,\gamma_1q^2,\gamma_1 q^4,\!\begin{array}{ll}\gamma_1q^6\,, & \!\!\!\!\gamma_2\\[0.2em] \gamma_1 q^{6} & \end{array}\!\!\bigr)\,$.

Let $n=4$. We show graphically an example of a truncated orbit of an element of $\STab(R,\sv)$, which turns out to be admissible (see also Example \ref{ex-res}). To save space, we omit to indicate the place $\gamma_1$, which is constant here, and we omit the curly brackets. The non-indicated actions of $\ovr_{1},\ovr_{2},\ovr_{3},\ovr_{\unn{3}}$ give $\bze$.

\[
\begin{array}{ccccc}
 \!\!\!\left(
\begin{array}{cccc}
\scriptstyle{\times} & \hspace{-0.35cm}\cdot & \hspace{-0.35cm}\fbox{\scriptsize{$2$}} & \hspace{-0.35cm}\fbox{\scriptsize{$3$}} \\[-0.2em]
\fbox{\scriptsize{$0$}} & \hspace{-0.35cm}\fbox{\scriptsize{$1$}} & &
\end{array},
\begin{array}{ccc}
\scriptstyle{\times} & \hspace{-0.35cm}\cdot & \hspace{-0.35cm}\fbox{\scriptsize{$2$}} \\[-0.2em]
\fbox{\scriptsize{$0$}} & \hspace{-0.35cm}\fbox{\scriptsize{$1$}} & \hspace{-0.35cm}\fbox{\scriptsize{$\unn{3}$}}
\end{array}
\right)
& \!\!\!\stackrel{\ \ \ovr_{2_{}}\ \,}{\longleftrightarrow}\!\!\! &
 \! \left(
\begin{array}{cccc}
\scriptstyle{\times} & \hspace{-0.35cm}\cdot & \hspace{-0.35cm}\fbox{\scriptsize{$1$}} & \hspace{-0.35cm}\fbox{\scriptsize{$3$}} \\[-0.2em]
\fbox{\scriptsize{$0$}} & \hspace{-0.35cm}\fbox{\scriptsize{$2$}} & &
\end{array},
\begin{array}{ccc}
\scriptstyle{\times} & \hspace{-0.35cm}\cdot & \hspace{-0.35cm}\fbox{\scriptsize{$1$}} \\[-0.2em]
\fbox{\scriptsize{$0$}} & \hspace{-0.35cm}\fbox{\scriptsize{$2$}} & \hspace{-0.35cm}\fbox{\scriptsize{$\unn{3}$}}
\end{array}
\right)
& \!\!\!\stackrel{\ \ \ovr_{1_{}}\ \,}{\longleftrightarrow}\!\!\!\!\! &
\!\!\!\!\!\!\left(
\begin{array}{cccc}
\scriptstyle{\times} & \hspace{-0.35cm}\cdot & \hspace{-0.35cm}\fbox{\scriptsize{$0$}} & \hspace{-0.35cm}\fbox{\scriptsize{$3$}} \\[-0.2em]
\fbox{\scriptsize{$1$}} & \hspace{-0.35cm}\fbox{\scriptsize{$2$}} & &
\end{array},
\begin{array}{ccc}
\scriptstyle{\times} & \hspace{-0.35cm}\cdot & \hspace{-0.35cm}\fbox{\scriptsize{$0$}} \\[-0.2em]
\fbox{\scriptsize{$1$}} & \hspace{-0.35cm}\fbox{\scriptsize{$2$}} & \hspace{-0.35cm}\fbox{\scriptsize{$\unn{3}$}}
\end{array}
\right)\\[1.2em]
&& \ \text{\large{$\updownarrow$}\,\footnotesize{$\ovr_{3}$}}& &  \ \text{\large{$\updownarrow$}\,\footnotesize{$\ovr_{3}$}}\\[0.8em]
&& 
 \!\!\! \!\!\! \!\!\! \!\!\! \!\!\!\left(
\begin{array}{cccc}
\scriptstyle{\times} & \hspace{-0.35cm}\cdot & \hspace{-0.35cm}\fbox{\scriptsize{$1$}} & \hspace{-0.35cm}\fbox{\scriptsize{$2$}} \\[-0.2em]
\fbox{\scriptsize{$0$}} & \hspace{-0.35cm}\fbox{\scriptsize{$3$}} & &
\end{array},
\begin{array}{ccccc}
\scriptstyle{\times} & \hspace{-0.35cm}\cdot & \hspace{-0.35cm}\fbox{\scriptsize{$1$}}& \hspace{-0.35cm}\fbox{\scriptsize{$2$}} & \hspace{-0.35cm}\fbox{\scriptsize{$\unn{3}$}} \\[-0.2em]
\fbox{\scriptsize{$0$}} & & & &
\end{array}
\right)\!\!\!
&  \!\!\!\!\!\!\stackrel{\ \ \ovr_{1_{}}\ \,}{\longleftrightarrow}\!\!\!\!\!  &
\!\!\!\left(
\begin{array}{cccc}
\scriptstyle{\times} & \hspace{-0.35cm}\cdot & \hspace{-0.35cm}\fbox{\scriptsize{$0$}} & \hspace{-0.35cm}\fbox{\scriptsize{$2$}} \\[-0.2em]
\fbox{\scriptsize{$1$}} & \hspace{-0.35cm}\fbox{\scriptsize{$3$}} & &
\end{array},
\begin{array}{ccccc}
\scriptstyle{\times} & \hspace{-0.35cm}\cdot & \hspace{-0.35cm}\fbox{\scriptsize{$0$}}& \hspace{-0.35cm}\fbox{\scriptsize{$2$}} & \hspace{-0.35cm}\fbox{\scriptsize{$\unn{3}$}} \\[-0.2em]
\fbox{\scriptsize{$1$}} & & & &
\end{array}
\right)\\[1.2em]
&& & &  \ \text{\large{$\updownarrow$}\,\footnotesize{$\ovr_{2}$}}\\[0.6em]
&& & & \!\!\!\!\!\left(
\begin{array}{cccc}
\scriptstyle{\times} & \hspace{-0.35cm}\cdot & \hspace{-0.35cm}\fbox{\scriptsize{$0$}} & \hspace{-0.35cm}\fbox{\scriptsize{$1$}} \\[-0.2em]
\fbox{\scriptsize{$2$}} & \hspace{-0.35cm}\fbox{\scriptsize{$3$}} & &
\end{array},
\begin{array}{ccccc}
\scriptstyle{\times} & \hspace{-0.35cm}\cdot & \hspace{-0.35cm}\fbox{\scriptsize{$0$}}& \hspace{-0.35cm}\fbox{\scriptsize{$1$}} & \hspace{-0.35cm}\fbox{\scriptsize{$\unn{3}$}} \\[-0.2em]
\fbox{\scriptsize{$2$}} & & & &
\end{array}
\right)\\[1.2em]
&& & &  \ \text{\large{$\updownarrow$}\,\footnotesize{$\ovr_{\unn{3}}$}}\\[0.6em]
&& & & 
\!\!\!\!\!\!\!\!\!\left(
\begin{array}{cccccc}
\scriptstyle{\times} & \hspace{-0.35cm}\cdot & \hspace{-0.35cm}\fbox{\scriptsize{$0$}} & \hspace{-0.35cm}\fbox{\scriptsize{$1$}} & \hspace{-0.35cm}\fbox{\scriptsize{$2$}} & \hspace{-0.35cm}\fbox{\scriptsize{$3$}} \\[-0.2em]
& & & & & 
\end{array},
\begin{array}{ccccc}
\scriptstyle{\times} & \hspace{-0.35cm}\cdot & \hspace{-0.35cm}\fbox{\scriptsize{$0$}}& \hspace{-0.35cm}\fbox{\scriptsize{$1$}} & \hspace{-0.35cm}\fbox{\scriptsize{$2$}} \\[-0.2em]
\fbox{\scriptsize{$\unn{3}$}} & & & &
\end{array}
\right)
\end{array}
\]

\paragraph{\textbf{(c) $R=E_n$ ($n=6,7,8$) and $\sv=3$ in the standard labelling.}} The new labelling of the vertices of the Dynkin diagram is (with $l=n-2$):
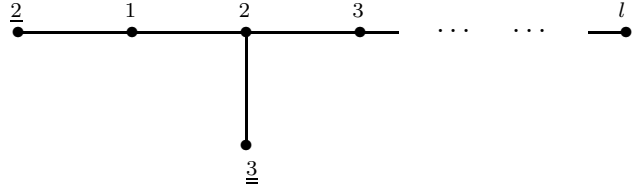
\begin{center}
\setlength{\unitlength}{0.01cm}
\begin{picture}(1550,220)(100,510)
\put(500,700){\circle*{15}}
\put(490,720){\scriptsize{$\un{2}$}}
\put(500,700){\line(1,0){150}}
\put(650,700){\circle*{15}}
\put(640,720){\scriptsize{$1$}}
\put(650,700){\line(1,0){150}}
\put(800,700){\circle*{15}}
\put(790,720){\scriptsize{$2$}}
\put(800,700){\line(0,-1){150}}
\put(800,550){\circle*{15}}
\put(800,510){\scriptsize{$\unn{3}$}}
\put(800,700){\line(1,0){150}}
\put(950,700){\circle*{15}}
\put(940,720){\scriptsize{$3$}}
\put(950,700){\line(1,0){50}}
\put(1050,700){$\ldots$}
\put(1150,700){$\ldots$}
\put(1250,700){\line(1,0){50}}
\put(1300,700){\circle*{15}}
\put(1290,720){\scriptsize{$l$}}
\end{picture}
\end{center}
We show an example of a standard tableau of type $(R,\sv)$:
\[
\bT=\left(\Biggl\{\begin{array}{c}
\gamma_1 \\[0.5em]
\begin{array}{ccc}
\fbox{\scriptsize{$0$}} & \hspace{-0.35cm}\fbox{\scriptsize{$1$}} & \hspace{-0.35cm}\fbox{\scriptsize{$3$}} \\[-0.2em]
\fbox{\scriptsize{$2$}} & \hspace{-0.35cm}\fbox{\scriptsize{$4$}} &
\end{array}
\end{array}\Biggr\}\ ,\ 
\Biggl\{\begin{array}{c}
\gamma_1 \\[0.5em]
\begin{array}{cc}
\fbox{\scriptsize{$0$}} & \hspace{-0.35cm}\fbox{\scriptsize{$1$}} \\[-0.2em]
\fbox{\scriptsize{$\un{2}$}} &\end{array}
\end{array}\Biggr\}\ ,\ 
\Biggl\{\begin{array}{c}
\gamma_1 \\[0.5em]
\begin{array}{ccc}
\fbox{\scriptsize{$0$}} & \hspace{-0.35cm}\fbox{\scriptsize{$1$}} & \hspace{-0.35cm}\fbox{\scriptsize{$\unn{3}$}} \\[-0.4em]
\fbox{\scriptsize{$2$}} &  &
\end{array}
\end{array}\Biggr\}
\right)\]
The associated sequence of contents is $\Seq(\bT)=\Bigl(\gamma_1,\gamma_1q^2,\begin{array}{cc}\gamma_1q^{-2} \,, & 
\begin{array}{cc}
\gamma_1q^4\,, &\gamma_1 \\[0.2em]
\gamma_1q^4 &  \end{array}\\ 
\gamma_1q^{-2} & 
\end{array}\!\Bigr)$. The substrings are $\bS_1=(\gamma_1,\gamma_1q^2,\gamma_1q^{-2},\gamma_1q^4,\gamma_1)$, $\bS_2=(\gamma_1,\gamma_1q^2,\gamma_1q^{-2})$ and $\bS_3=(\gamma_1,\gamma_1q^2,\gamma_1q^{-2},\gamma_1q^4)$ and they satisfy Condition $(\star)$ of Proposition \ref{prop-tab-A}.

We show graphically the action of the operators $\ovr_i$, $i\in\{2,3,4,\un{2},\unn{3}\}$, on the element $\bT$ considered above ($\ovr_1(\bT)=\bze$). In every tableaux below, the node with $0$ has coordinates $(1,1)$. Again, the place $\gamma_1$, which is constant here, and the curly brackets are omitted. As can be seen graphically, $\ovr_{\un{2}}(\bT)$ is a sequence of placed nodes which is not a tableau. Therefore, the truncated orbit $\oO_{\bT}$ is not admissible (and thus neither are the standard tableaux appearing in it).

\[
\begin{array}{cccc}
\left(
\begin{array}{ccc}
& & \\[-0.2em]
\fbox{\scriptsize{$0$}} & \hspace{-0.35cm}\fbox{\scriptsize{$1$}} & \hspace{-0.35cm}\fbox{\scriptsize{$2$}} \\[-0.2em]
\fbox{\scriptsize{$3$}} & \hspace{-0.35cm}\fbox{\scriptsize{$4$}} &
\end{array},
\begin{array}{cc}
 & \\[-0.2em]
\fbox{\scriptsize{$0$}} &  \hspace{-0.35cm}\fbox{\scriptsize{$1$}} \\[-0.2em]
\fbox{\scriptsize{$\un{2}$}} & \end{array},
\begin{array}{ccccc}
 & & & & \hspace{-0.35cm}\fbox{\scriptsize{$\unn{3}$}}\\[-0.2em]
\fbox{\scriptsize{$0$}} & \hspace{-0.35cm}\fbox{\scriptsize{$1$}} & \hspace{-0.35cm}\fbox{\scriptsize{$2$}} & \hspace{-0.35cm}\cdot &  \hspace{-0.35cm}\cdot \\[-0.2em]
&  & & &
\end{array}
\right) 
& &  & 
\hspace{-5cm}\left(
\begin{array}{ccc}
\fbox{\scriptsize{$0$}} & \hspace{-0.35cm}\fbox{\scriptsize{$2$}} & \hspace{-0.35cm}\fbox{\scriptsize{$3$}} \\[-0.2em]
\fbox{\scriptsize{$1$}} & \hspace{-0.35cm}\fbox{\scriptsize{$4$}} & \\[-0.2em]
$\ $ & & 
\end{array},
\begin{array}{cc}
 &  \hspace{-0.35cm}\fbox{\scriptsize{$0$}} \\[-0.2em]
 & \hspace{-0.35cm}\fbox{\scriptsize{$1$}} \\[-0.2em]
  \fbox{\scriptsize{$\un{2}$}} & 
  \end{array},
\begin{array}{ccc}
\fbox{\scriptsize{$0$}} & \hspace{-0.35cm}\fbox{\scriptsize{$2$}} & \hspace{-0.35cm}\fbox{\scriptsize{$\unn{3}$}} \\[-0.4em]
\fbox{\scriptsize{$1$}} &  & \\[-0.2em]
$\ $ & & 
\end{array}
\right)  \\[2em]

\hspace{2.5cm}\text{\footnotesize{$\ovr_{3}$}}\!\!\searrow\mspace{-18.0mu}\nwarrow & & \hspace{-1.5cm}\swarrow\mspace{-18.0mu}\nearrow\!\text{\footnotesize{$\ovr_{2}$}} & \\[0.8em]

& \hspace{-2cm}\left(
\begin{array}{ccc}
\fbox{\scriptsize{$0$}} & \hspace{-0.35cm}\fbox{\scriptsize{$1$}} & \hspace{-0.35cm}\fbox{\scriptsize{$3$}} \\[-0.2em]
\fbox{\scriptsize{$2$}} & \hspace{-0.35cm}\fbox{\scriptsize{$4$}} &
\end{array},
\begin{array}{cc}
\fbox{\scriptsize{$0$}} & \hspace{-0.35cm}\fbox{\scriptsize{$1$}} \\[-0.2em]
\fbox{\scriptsize{$\un{2}$}} &\end{array}
,
\begin{array}{ccc}
\fbox{\scriptsize{$0$}} & \hspace{-0.35cm}\fbox{\scriptsize{$1$}} & \hspace{-0.35cm}\fbox{\scriptsize{$\unn{3}$}} \\[-0.4em]
\fbox{\scriptsize{$2$}} &  &
\end{array}
\right)
& \!\!\!\stackrel{\ \ \ovr_{4_{}}\ \,}{\longleftrightarrow}\!\!\!\! &
\!\!\left(
\begin{array}{ccc}
\fbox{\scriptsize{$0$}} & \hspace{-0.35cm}\fbox{\scriptsize{$1$}} & \hspace{-0.35cm}\fbox{\scriptsize{$4$}} \\[-0.2em]
\fbox{\scriptsize{$2$}} & \hspace{-0.35cm}\fbox{\scriptsize{$3$}} &
\end{array},
\begin{array}{cc}
\fbox{\scriptsize{$0$}} & \hspace{-0.35cm}\fbox{\scriptsize{$1$}} \\[-0.2em]
\fbox{\scriptsize{$\un{2}$}} &\end{array}
,
\begin{array}{ccc}
\fbox{\scriptsize{$0$}} & \hspace{-0.35cm}\fbox{\scriptsize{$1$}} & \hspace{-0.35cm}\fbox{\scriptsize{$\unn{3}$}} \\[-0.4em]
\fbox{\scriptsize{$2$}} &  &
\end{array}
\right)\\[2em]
\hspace{3cm}\swarrow\mspace{-18.0mu}\nearrow\!\text{\footnotesize{$\ovr_{\unn{3}}$}} & & \hspace{-2.5cm}\!\text{\footnotesize{$\ovr_{\un{2}}$}}\!\!\searrow\mspace{-18.0mu}\nwarrow & \\[0.8em]

\left(
\begin{array}{ccccc}
&& & & \hspace{-0.35cm}\fbox{\scriptsize{$3$}}\\[-0.2em]
\fbox{\scriptsize{$0$}} & \hspace{-0.35cm}\fbox{\scriptsize{$1$}} & \hspace{-0.35cm}\fbox{\scriptsize{$2$}} &  \hspace{-0.35cm}\fbox{\scriptsize{$4$}} & \\[-0.2em]
&& & &
\end{array},
\begin{array}{cc}
& \\[-0.2em]
\fbox{\scriptsize{$0$}} & \hspace{-0.35cm}\fbox{\scriptsize{$1$}} \\[-0.2em]
\fbox{\scriptsize{$\un{2}$}} &\end{array}
,
\begin{array}{ccc}
 & &\\[-0.2em]
\fbox{\scriptsize{$0$}} & \hspace{-0.35cm}\fbox{\scriptsize{$1$}} & \hspace{-0.35cm}\fbox{\scriptsize{$2$}} \\[-0.2em]
\fbox{\scriptsize{$\unn{3}$}} &  &
\end{array}
\right)
& & &
\hspace{-4.5cm}\left(
\begin{array}{ccc}
 & \hspace{-0.35cm}\fbox{\scriptsize{$0$}} &  \\[-0.2em]
 & \hspace{-0.35cm}\fbox{\scriptsize{$1$}} &  \hspace{-0.35cm}\fbox{\scriptsize{$3$}}\\[-0.2em]
 \fbox{\scriptsize{$2$}} & \hspace{-0.35cm}\fbox{\scriptsize{$4$}} &
\end{array},
\begin{array}{cc}
\fbox{\scriptsize{$0$}} & \hspace{-0.35cm}\fbox{\scriptsize{$\un{2}$}} \\[-0.2em]
\fbox{\scriptsize{$1$}} &\\[-0.2em]
 & 
\end{array},
\begin{array}{ccc}
 & \hspace{-0.35cm}\fbox{\scriptsize{$0$}} &  \\[-0.2em]
 & \hspace{-0.35cm}\fbox{\scriptsize{$1$}} &  \hspace{-0.35cm}\fbox{\scriptsize{$\unn{3}$}}\\[-0.4em]
 \fbox{\scriptsize{$2$}} & &
\end{array}
\right)
\end{array}
\]
In terms of sequences, we have:
\[\Seq(\bT)=\Bigl(\gamma_1,\gamma_1q^2,\begin{array}{cc}\gamma_1q^{-2} \,, & 
\begin{array}{cc}
\gamma_1q^4\,, &\gamma_1 \\[0.2em]
\gamma_1q^4 &  \end{array}\\ 
\gamma_1q^{-2} & 
\end{array}\!\Bigr)\stackrel{\ \ \ovr_{\un{2}_{}}\ \,}{\longleftrightarrow}\Bigl(\gamma_1,\gamma_1q^{-2},\begin{array}{cc}\gamma_1q^{-6} \,, & 
\begin{array}{cc}
\gamma_1\,, &\gamma_1q^{-4} \\[0.2em]
\gamma_1 &  \end{array}\\ 
\gamma_1q^{2} & 
\end{array}\!\Bigr)\ ,\]
and, in $\ovr_{\un{2}}\bigl(\Seq(\bT)\bigr)$, the substrings $(\gamma_1,\gamma_1q^{-2},\gamma_1q^{-6},\gamma_1,\gamma_1q^{-4})$ and $(\gamma_1,\gamma_1q^{-2},\gamma_1q^{-6},\gamma_1)$ do not satisfy Condition $(\star)$ of Proposition \ref{prop-tab-A} and therefore do not correspond to any standard tableau.

\paragraph{\textbf{(d) $R=A_3$ and $\sv=2$ in the standard labelling.}} In the following examples, we exhibit phenomenons which were obviously not present in the classical situation of $GL_{n+1}$ (example \textbf{(a)}). Namely, we first show in an example that the values of the places have an influence on the size or the admissibility of a truncated orbit. In a second example, we show an admissible standard tableau $\bT$ such that $r_i(\bT)=r_j(\bT)$ for different $i,j$.

The new labelling of the vertices of the Dynkin diagram is:
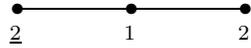
\begin{center}
\setlength{\unitlength}{0.01cm}
\begin{picture}(1550,50)(-100,1160)
\put(500,1200){\circle*{15}}
\put(490,1160){\scriptsize{$\un{2}$}}
\put(500,1200){\line(1,0){150}}
\put(650,1200){\circle*{15}}
\put(640,1160){\scriptsize{$1$}}
\put(650,1200){\line(1,0){150}}
\put(800,1200){\circle*{15}}
\put(790,1160){\scriptsize{$2$}}
\end{picture}
\end{center}
We take two places $\gamma_1,\gamma_2\in\C(q)^{\times}$ so that, by definition, $\gamma_2\neq\gamma_1q^{2a}$ for any $a\in\Z$. Consider the sequence $\bS=\Bigl(\gamma_1,\gamma_1q^{-2},\begin{array}{c}\gamma_2 \\[0.2em]
\gamma_2 
\end{array}\!\Bigr)$ (it corresponds to a standard tableau of type $(R,\sv)$ shown below). We have:

\[\bS=\Bigl(\gamma_1,\gamma_1q^{-2},\begin{array}{c}\gamma_2 \\[0.2em]
\gamma_2 
\end{array}\!\Bigr)\ \ \stackrel{\ \ \ovr_{2_{}}\ \,}{\longleftrightarrow}\ \ \Bigl(\gamma_1,\gamma_2,\begin{array}{c}\gamma_1q^{-2} \\[0.2em]
\gamma_1^{-1}\gamma_2^2q^2
\end{array}\!\Bigr)\]
Assume first that $\gamma_1^{-1}\gamma_2^2\notin\{\gamma_1q^{2a},\gamma_2 q^{2a}\}_{a\in\Z}$, which is equivalent, by assumption on $\gamma_1,\gamma_2$, to $\gamma_2\neq\pm\gamma_1q^a$ for any $a\in\Z$. Then, a new place is created, and a graphical representation is:
\[
\left(\Biggl\{\!\!\begin{array}{c}
\gamma_1 \\[0.5em]
\begin{array}{c}
\fbox{\scriptsize{$0$}}  \\[-0.2em]
\fbox{\scriptsize{$1$}}\end{array}
\end{array} ,
\begin{array}{c}
\gamma_2 \\[0.5em]
\begin{array}{c}
\fbox{\scriptsize{$2$}}\\[-0.2em]
\phantom{\fbox{\scriptsize{$2$}}}  \end{array}
\end{array}\!\!\Biggr\}\ ,\ 
\Biggl\{\!\!\begin{array}{c}
\gamma_1 \\[0.5em]
\begin{array}{c}
\fbox{\scriptsize{$0$}}  \\[-0.2em]
\fbox{\scriptsize{$1$}}\end{array}
\end{array} ,
\begin{array}{c}
\gamma_2 \\[0.5em]
\begin{array}{c}
\fbox{\scriptsize{$\un{2}$}}\\[-0.2em]
\  \end{array}
\end{array}\!\!\Biggr\}
\right)
\stackrel{\ \ \ovr_{2_{}}\ \,}{\longleftrightarrow}
\left(\Biggl\{\!\!\begin{array}{c}
\gamma_1 \\[0.5em]
\begin{array}{c}
\fbox{\scriptsize{$0$}}  \\[-0.2em]
\fbox{\scriptsize{$2$}}\end{array}
\end{array} ,
\begin{array}{c}
\gamma_2 \\[0.5em]
\begin{array}{c}
\fbox{\scriptsize{$1$}}\\[-0.2em]
\phantom{\fbox{\scriptsize{$1$}}}  \end{array}
\end{array}\!\!\Biggr\}\ ,\ 
\Biggl\{\!\!\begin{array}{c}
\gamma_1 \\[0.5em]
\begin{array}{c}
\fbox{\scriptsize{$0$}}  \\[-0.2em]
\ \end{array}
\end{array} ,
\begin{array}{c}
\gamma_2 \\[0.5em]
\begin{array}{c}
\fbox{\scriptsize{$1$}}\\[-0.2em]
\phantom{\fbox{\scriptsize{$1$}}}  \end{array}
\end{array} ,
\begin{array}{c}
\gamma_1^{-1}\gamma_2^2 \\[0.5em]
\begin{array}{c}
\scriptstyle{\times}\ \fbox{\scriptsize{$\un{2}$}}\\[-0.2em]
\  \end{array}
\end{array}\!\!\Biggr\}
\right)
\]
In this case, it can be checked that this is part of an admissible truncated orbit containing $12$ elements.

Now assume that $\gamma_2=\pm\gamma_1q^a$ for some $a\in\Z$ ($a$ is odd if $\gamma_2=\gamma_1q^{a}$). Then $\gamma_1^{-1}\gamma_2^2=\gamma_1q^{2a}$ and therefore, in this situation, no new place is created. Instead, the placed nodes with $\un{2}$ has to move in the diagram with place $\gamma_1$. Below, we only show how the second tableau (with numbers 0,1,$\un{2}$) is transformed by $\ovr_{2}$ in the situation, respectively, $a=-2$, $a=-1$, $a=0$ and $a=1$.
\[\Biggl\{\!\!\begin{array}{c}
\gamma_1 \\[0.5em]
\begin{array}{c}
\fbox{\scriptsize{$0$}}  \\[-0.2em]
\fbox{\scriptsize{$\un{2}$}}\end{array}
\end{array} ,
\begin{array}{c}
\gamma_2 \\[0.5em]
\begin{array}{c}
\fbox{\scriptsize{$1$}} \\[-0.2em]
\phantom{\fbox{\scriptsize{$1$}}}  \end{array}
\end{array}\!\!\Biggr\}
\ ,\ \ \ \ 
\Biggl\{\!\!\begin{array}{c}
\gamma_1 \\[0.5em]
\begin{array}{cc}
\fbox{\scriptsize{$0$}} & \hspace{-0.35cm}\cdot \\[-0.2em]
\cdot & \hspace{-0.35cm}\fbox{\scriptsize{$\un{2}$}}\end{array}
\end{array} ,
\begin{array}{c}
\gamma_2 \\[0.5em]
\begin{array}{c}
\fbox{\scriptsize{$1$}} \\[-0.2em]
\phantom{\fbox{\scriptsize{$1$}}}  \end{array}
\end{array}\!\!\Biggr\}
\ ,\ \ \ \ 
\Biggl\{\!\!\begin{array}{c}
\gamma_1 \\[0.5em]
\begin{array}{cc}
\fbox{\scriptsize{$0$}} & \hspace{-0.35cm}\fbox{\scriptsize{$\un{2}$}}  \\[-0.2em]
& \end{array}
\end{array} ,
\begin{array}{c}
\gamma_2 \\[0.5em]
\begin{array}{c}
\fbox{\scriptsize{$1$}} \\[-0.2em]
\phantom{\fbox{\scriptsize{$1$}}}  \end{array}
\end{array}\!\!\Biggr\}
\ ,\ \ \ \ 
\Biggl\{\!\!\begin{array}{c}
\gamma_1 \\[0.5em]
\begin{array}{ccc}
\fbox{\scriptsize{$0$}} & \hspace{-0.35cm}\cdot & \hspace{-0.35cm}\fbox{\scriptsize{$\un{2}$}}  \\[-0.2em]
& &\end{array}
\end{array} ,
\begin{array}{c}
\gamma_2 \\[0.5em]
\begin{array}{c}
\fbox{\scriptsize{$1$}} \\[-0.2em]
\phantom{\fbox{\scriptsize{$1$}}}  \end{array}
\end{array}\!\!\Biggr\}\]
We see at once that we have a non-admissible situation for $a=-1$. It can be checked, as before, that for any other values of $a$, we have an admissible truncated orbit. However, now, the truncated orbit still contains 12 elements if and only if $a\neq -2,0$. If $a=0$ then the truncated orbit contains 6 elements while if $a=-2$ the truncated orbit only contains $3$ elements. 

Indeed, we have in general $\ovr_{\un{2}}(\bS)=\Bigl(\gamma_1,\gamma_2,\begin{array}{c}\gamma_1^{-1}\gamma_2^2q^2 \\[0.2em]
\gamma_1q^{-2}
\end{array}\!\Bigr)$ and so it coincides with $\ovr_2(\bS)$ if $a=-2$. We write the truncated orbit when $a=-2$, that is when $\gamma_2=-\gamma_1q^{-2}$ (The non-indicated actions of $\ovr_{1},\ovr_{2},\ovr_{\un{2}}$ give $\bze$):

\[\Bigl(\gamma_1,\gamma_1q^{-2},\begin{array}{c}-\gamma_1q^{-2} \\[0.2em]
-\gamma_1q^{-2} 
\end{array}\!\Bigr)\ \ \stackrel{\ovr_{2_{}}\,,\,\ovr_{\un{2}_{}}}{\longleftrightarrow}\ \ \Bigl(\gamma_1,-\gamma_1q^{-2},\begin{array}{c}\gamma_1q^{-2} \\[0.2em]
\gamma_1q^{-2}
\end{array}\!\Bigr)\ \ \stackrel{\ovr_{1_{}}}{\longleftrightarrow}\ \ \Bigl(-\gamma_1q^{-2},\gamma_1,\begin{array}{c}\gamma_1q^{-2} \\[0.2em]
\gamma_1q^{-2}
\end{array}\!\Bigr)\]

\section{Construction of representations of $H(R,\sv)$ and $\oH(R,\sv)$}\label{sec-rep}

\subsection{Representations of $H(R,\sv)$}\label{subsec-rep}

We recall that the generators of the affine Hecke algebra $H(R,\sv)$ are $g_i$, $i\in\Ve_R$, and $X_j^{\pm1}$, $j\in\Ve_R\cup\{0\}$.

Let $\oO$ be an admissible truncated orbit in $\Seq_{R,\sv}$. By Definition \ref{def-adm-tab} of admissibility, the orbit $\oO$ consists only of elements of the form $\Seq(\bT)$ with $\bT\in\STab(R,\sv)$; we identify $\bT$ and $\Seq(\bT)$. 

Let $V_{\oO}$ be a $\C(q)$-vector space with basis $\{v_{\bT}\}_{\bT\in\oO}$ and set $v_{\bze}:=0$. We define linear operators $g_i$, $i\in\Ve_R$, and $X_j^{\pm1}$, $j\in\Ve_R\cup\{0\}$, on $V_{\oO}$ by setting, for any $\bT\in\oO$,
\begin{equation}\label{rep-X}
X_{j}(v_{\bT})=\qc_j(\bT)\,v_{\bT}\ ,\ \ \ \qquad\text{for $j\in\Ve_R\cup\{0\}$;}
\end{equation}
\begin{equation}\label{rep-g}
g_{i}(v_{\bT})=\frac{(q-q^{-1})\qc_{i}(\bT)}{\qc_{i}(\bT)-\qc_{i-1}(\bT)}\,v_{\bT}+\frac{q\,\qc_{i}(\bT)-q^{-1}\qc_{i-1}(\bT)}{\qc_{i}(\bT)-\qc_{i-1}(\bT)}\,v_{\ovr_{i}(\bT)}\ ,\ \ \ \quad\text{for $i\in\Ve_R$.}
\end{equation}
Note that $\qc_{i}(\bT)\neq\qc_{i-1}(\bT)$ for any $\bT\in\STab(R,\sv)$.
\begin{proposition}\label{prop-rep}
Formulas (\ref{rep-X}) and (\ref{rep-g}) define a representation of $\C(q)H(R,\sv)$ on $V_{\oO}$.
\end{proposition}
\begin{proof} We shall check that the defining relations of the algebra $H(R,\sv)$ are satisfied by the operators $g_i$, $i\in\Ve_R$, and $X_j^{\pm1}$, $j\in\Ve_R\cup\{0\}$, on $V_{\oO}$. First note that the commuting relation $X_iX_j=X_jX_i$ for any $i,j\in\Ve_R\cup\{0\}$ is obviously satisfied. Note also that the operators $X_j$, $j\in\Ve_R\cup\{0\}$, are invertible since $\qc_j(\bT)\neq 0$ for any $\bT\in\STab(R,\sv)$. We organize the remaining verifications in three steps. To save space, we denote during the proof $\qc^{\bT}_j:=\qc_j(\bT)$ for any $\bT\in\oO$ and any $j\in\Ve_R\cup\{0\}$.

\textbf{1.} Recall the subalgebras $\cH_u$, $u=1,2,3$, of $H(R,\sv)$ defined in Section \ref{subsec-GL} which are isomorphic to $H(GL_N)$ for some $N$. We consider first a defining relation of $H(R,\sv)$ involving only generators lying inside one of these subalgebras $\cH_u$. Let $\bT\in\oO$. Then Formulas (\ref{rep-X})-(\ref{rep-g}) only involve the contents of the standard tableaux $\bT_u$ inside $\bT$ and moreover, they coincide with the known formulas for the representations of $H(GL_n)$ given in \cite{Ra}. We conclude that we already know that such defining relations are satisfied on $V_{\oO}$. In particular, this takes care of the relations $g_i^2=(q-q^{-1})g_i+1$ for $i\in\Ve_R$, and $g_ig_jg_i=g_jg_ig_j$ for $i,j\in\Ve_R$ such that $m_{i,j}=3$. Therefore we do not repeat the calculations necessary to check these relations (instead, we refer to \cite{Ra}; see also \cite{Ho,AK} for similar calculations).

\textbf{2.} Now we will check the defining relations involving only generators $g_i$, $i\in\Ve_R$. Due to step \textbf{1}, it is enough to check the ones which are not already relations inside one of the subalgebras $\cH_u$, $u=1,2,3$. In particular, only some defining relations of the form $g_ig_j=g_jg_i$ remain to be checked. We are going to check all defining relations of the form $g_ig_j=g_jg_i$.

Let $i,j\in\Ve_R$ such that $m_{i,j}=2$, and let $\bT\in\oO$. A direct calculation gives:
\[\begin{array}{ll}
g_jg_i(v_{\bT})\ = & \displaystyle\frac{(q-q^{-1})\qc_{i}^{\bT}}{\qc_{i}^{\bT}-\qc_{i-1}^{\bT}}\,
\Biggl(\,\frac{(q-q^{-1})\qc_{j}^{\bT}}{\qc_{j}^{\bT}-\qc_{j-1}^{\bT}}\cdot v_{\bT}\ +\ \frac{q\,\qc_{j}^{\bT}-q^{-1}\qc_{j-1}^{\bT}}{\qc_{j}^{\bT}-\qc_{j-1}^{\bT}}\cdot v_{\ovr_{j}(\bT)}\,\Biggr)
\\[1.8em]
& \displaystyle \hspace{-1.5cm}+\ \ \frac{q\,\qc_{i}^{\bT}-q^{-1}\qc_{i-1}^{\bT}}{\qc_{i}^{\bT}-\qc_{i-1}^{\bT}}\,
\Biggl(\,\frac{(q-q^{-1})\qc_{j}^{r_i(\bT)}}{\qc_{j}^{r_i(\bT)}-\qc_{j-1}^{r_i(\bT)}}\cdot v_{\ovr_i(\bT)}\ +\ \frac{q\,\qc_{j}^{r_i(\bT)}-q^{-1}\qc_{j-1}^{r_i(\bT)}}{\qc_{j}^{r_i(\bT)}-\qc_{j-1}^{r_i(\bT)}}\cdot v_{\ovr_j\ovr_{i}(\bT)}\,\Biggr)\ .
\end{array} \]
Note that this equation is always valid since, first, if $\ovr_i(\bT)=\bze$ then $\ovr_j\ovr_i(\bT)=\bze$ as well and the second line above is in this case equal to $0$; and second, thanks to Lemma \ref{lem-mon}, we have 
\begin{equation}\label{rep-eq2}
\frac{\qc_{j-1}^{r_i(\bT)}}{\qc_{j}^{r_i(\bT)}}=\frac{\qc_{j-1}^{\bT}}{\qc_{j}^{\bT}}\ ,
\end{equation}
since $m_{i,j}=2$ (in particular, the denominators are different from $0$). Using (\ref{rep-eq2}) in the above formula for $g_jg_i(v_{\bT})$, we can replace each $\qc_{j}^{r_i(\bT)}$ by $ \qc_{j}^{\bT}$ and each $\qc_{j-1}^{r_i(\bT)}$ by $ \qc_{j-1}^{\bT}$. Moreover, from Proposition \ref{prop-mon}, we have $\ovr_j\ovr_{i}(\bT)=\ovr_i\ovr_{j}(\bT)$. By exchanging $i$ and $j$, it is now immediate that $g_jg_i(v_{\bT})=g_ig_j(v_{\bT})$.

\textbf{3.} It remains to check the defining relation of $H(R,\sv)$ involving both generators $g_i$ and $X_j$. They are given in general by Relations (\ref{rel-Lu}). Due to step \textbf{1}, it would be enough to check the ones which are not relations inside one of the subalgebras $\cH_u$, $u=1,2,3$. Nevertheless, we will check them in a uniform way by proving that the following relation is satisfied on $V_{\oO}$, for $i\in\Ve_R$ and $j\in\Ve_R\cup\{0\}$,
\begin{equation}\label{rep-eq1}
g_iX_{j}-X^{r_i(\delta_{j})}g_i=(q-q^{-1})\frac{X_{j}-X^{r_i(\delta_{j})}}{1-X^{-\alpha_i}}\ .
\end{equation}
First, recall that, by definition of the action of the Weyl group on sequences (Section \ref{sec-act}), we have $X^{r_i(\delta_{j})}(v_{\bT})=\qc_j^{r_i(\bT)}v_{\bT}$ for any $\bT\in\oO$. Then recall that $\alpha_i=\delta_i-\delta_{i-1}$ for any $i\in\Ve_R$ and thus that $X^{-\alpha_i}=X_{i-1}X_i^{-1}$.

Let $\bT\in\oO$. We apply the right hand side of (\ref{rep-eq1}) on $v_{\bT}$ and find:
\[(q-q^{-1})\frac{X_{j}-X^{r_i(\delta_{j})}}{1-X^{-\alpha_i}}(v_{\bT})\ =\ (q-q^{-1})\,\qc_i^{\bT}\,\frac{\qc_{j}^{\bT}-\qc_{j}^{r_i(\bT)}}{\qc_i^{\bT}-\qc_{i-1}^{\bT}}\cdot v_{\bT}\ .\]
On the other hand, we have
\[g_iX_{j}(v_{\bT})\ =\ \qc_{j}^{\bT}\cdot\Bigl(\,\frac{(q-q^{-1})\qc_{i}^{\bT}}{\qc_{i}^{\bT}-\qc_{i-1}^{\bT}}\cdot v_{\bT}\ +\ \frac{q\,\qc_{i}^{\bT}-q^{-1}\qc_{i-1}^{\bT}}{\qc_{i}^{\bT}-\qc_{i-1}^{\bT}}\cdot v_{\ovr_{i}(\bT)}\,\Bigr)\ ,\]
while, 
\[X^{r_i(\delta_{j})}g_i(v_{\bT})\ =\ \frac{(q-q^{-1})\qc_{i}^{\bT}}{\qc_{i}^{\bT}-\qc_{i-1}^{\bT}}\,\qc_{j}^{r_i(\bT)}\cdot v_{\bT}\ +\ \frac{q\,\qc_{i}^{\bT}-q^{-1}\qc_{i-1}^{\bT}}{\qc_{i}^{\bT}-\qc_{i-1}^{\bT}}\,\qc_{j}^{r_i^2(\bT)}\cdot v_{\ovr_{i}(\bT)}\ .\]
We have $r_i^2(\bT)=\bT$ and then, putting together the last three calculations, we check easily that Relation (\ref{rep-eq1}) applied on $v_{\bT}$ is satisfied.
\end{proof}

\subsection{Calibrated spectrum of $H(R,\sv)$}\label{subsec-cal}

We use the ordered basis 
\[\cB_{R,\sv}:=\left(\delta_{0},\delta_{1},\dots,\dots,\delta_{l},,\delta_{\un{2}},\dots,\delta_{\un{l'}},\delta_{\unn{k+1}},\dots,\delta_{\unn{l''}}\right)\ .\]
of the free $\Z$-module $L_{R,\sv}$ to identify the calibrated spectrum $\text{C-Spec}\bigl(H(R,\sv)\bigr)$ with the set $\text{C-Eig}(\cB_{R,\sv})$, as explained in Section \ref{subsec-prel3}. 

We recall that $X_j:=X^{\delta_j}$ and that an element of $\text{C-Eig}(\cB_{R,\sv})$ is a sequence of eigenvalues for these generators of $H(R,\sv)$. Therefore an element of $\text{C-Eig}(\cB_{R,\sv})$ is naturally a sequence $\bS\in\Seq_{R,\sv}$ such that there is a calibrated representation $V$ of $H(R,\sv)$ and a non-zero vector $v\in V$ satisfying 
\begin{equation}\label{eq-proof1}
X_{j}(v_{\bS})=\qc_j(\bS)\,v_{\bS}\ ,\ \ \ \qquad\text{for $j\in\Ve_R\cup\{0\}$\,.}
\end{equation}

\begin{theorem}\label{theo-spec}
The calibrated spectrum of $H(R,\sv)$ is in bijection with the set $\ASTab(R,\sv)$ of admissible standard tableaux of type $(R,\sv)$. More precisely, we have
\begin{equation}\label{spec-H}
\text{C-Eig}(\cB_{R,\sv})=\{\Seq(\bT)\ |\ \bT\in\ASTab(R,\sv)\}\ .
\end{equation}
\end{theorem}
\begin{proof} \textbf{1.} We first show that $\text{C-Eig}(\cB_{R,\sv})\subset\{\Seq(\bT)\ |\ \bT\in\STab(R,\sv)\}$. Let $\bS\in\Seq_{R,\sv}$ with substrings $\bS_1,\bS_2,\bS_3$ and assume that $\bS\in\text{C-Eig}(\cB_{R,\sv})$. Let $V$ be a calibrated $\C(q)H(R,\sv)$-module and $v_{\bS}\in V$ a non-zero vector such that (\ref{eq-proof1}) is satisfied.

Let $u\in\{1,2,3\}$ and consider the subalgebra $\C(q)\cH_u$, which is isomorphic to $\C(q)H(GL_{l_u+1})$, where $l_1:=l$, $l_2:=l'$ and $l_3:=l''$; see Section \ref{subsec-GL}.  We consider $V$ as a module for $\C(q)\cH_u$. The module $V$ is calibrated as a $\C(q)\cH_u$-module. We now use Theorem \ref{theo-GL}, where the calibrated spectrum of $\C(q)\cH_u$ is described. 

Equation (\ref{eq-proof1}) restricted to generators $X_j$ in $\cH_u$ involves only the substring $\bS_u$ of $\bS$. It follows immediately from the inclusion of $\cH_u$ in $H(R,\sv)$ and Theorem \ref{theo-GL} that there is a standard tableau $\bT_u\in\STab(l_u+1)$ such that $\Seq(\bT_u)=\bS_u$. 

We thus obtained a triplet $(\bT_1,\bT_2,\bT_3)$ satisfying the conditions of Definition \ref{def-tab-R}, due to the intersections of the subalgebras $\cH_u$. Moreover, $\bS_u=\Seq(\bT_u)$, $u=1,2,3$. We conclude that we have a standard tableau $\bT=(\bT_1,\bT_2,\bT_3)$ of type $(R,\sv)$ such that $\bS=\Seq(\bT)$.

\textbf{2.} We prove now that $\bT$ is admissible. To do so, we shall show that, for $i\in\Ve_R$ such that $\ovr_i(\bS)\neq\bze$, we have $r_i(\bS)\in\text{C-Eig}(\cB_{R,\sv})$. Indeed, due to the first part of the proof this will imply that $r_i(\bS)$ is the sequence of contents of a standard tableau of type $(R,\sv)$. In turn, this implies that the truncated orbit of $\bT$ is admissible (and so is $\bT$).

So, let $i\in\Ve_R$ such that $\ovr_i(\bS)\neq\bze$. For brevity, we set during the proof $\qc_j^{\bS}:=\qc_j(\bS)$ and $\qc_j^{r_i(\bS)}:=\qc_j\bigl(r_i(\bS)\bigr)$ for $j\in\Ve_R\cup\{0\}$. Let $v'_{\bS}$ be the following vector of $V$:
\begin{equation}\label{eq-proof1a}
v'_{\bS}:=g_{i}(v_{\bS})-\frac{(q-q^{-1})\qc_i^{\bS}}{\qc_i^{\bS}-\qc_{i-1}^{\bS}}\,v_{\bS}\ .
\end{equation}
First, we have $v'_{\bS}\neq 0$. Indeed, by the assumption $\ovr_i(\bS)\neq\bze$, we have $\qc_i^{\bS}\neq q^{\pm2}\qc_{i-1}^{\bS}$ and therefore $\displaystyle\frac{(q-q^{-1})\qc_i^{\bS}}{\qc_i^{\bS}-\qc_{i-1}^{\bS}}\notin\{q,-q^{-1}\}$. So $v'_{\bS}= 0$ would contradict the relation $g_i^2=1+(q-q^{-1})g_i$.

Now, in order to show that $r_i(\bS)$ is indeed in $\text{C-Eig}(\cB_{R,\sv})$, we are going to prove that:
\begin{equation}\label{eq-proof1b}
X_{j}(v'_{\bS})=\qc_j^{r_i(\bS)}\,v'_{\bS}\ ,\ \ \ \qquad\text{for $j\in\Ve_R\cup\{0\}$\,.}
\end{equation}
Let $j\in\Ve_R\cup\{0\}$. First, recall that, by definition of the action of $W_0(R)$ on sequences (Section \ref{sec-act}), we have $X^{r_i(\delta_{j})}(v_{\bS})=\qc_j^{r_i(\bS)}v_{\bS}$, and also that $X^{-\alpha_i}=X_{i-1}X_i^{-1}$ since $\alpha_i=\delta_i-\delta_{i-1}$. We calculate, using the defining relation between $g_i$ and $X_j$,
\[\begin{array}{ll}
X_jg_i(v_{\bS}) & = g_iX^{r_i(\delta_j)}(v_{\bS})-(q-q^{-1})\displaystyle\frac{X^{r_i(\delta_j)}-X_j}{1-X^{-\alpha}}(v_{\bS})\\[1em]
& = \qc_j^{r_i(\bS)}g_i(v_{\bS}) - (q-q^{-1})\displaystyle\frac{(\qc_j^{r_i(\bS)}-\qc_j^{\bS})\qc_i^{\bS}}{\qc_i^{\bS}-\qc_{i-1}^{\bS}}v_{\bS} \\[1em]
& =  \qc_j^{r_i(\bS)}v'_{\bS} + (q-q^{-1})\displaystyle\frac{\qc_j^{\bS}\qc_i^{\bS}}{\qc_i^{\bS}-\qc_{i-1}^{\bS}}v_{\bS}\ .
\end{array}\]
Combining this calculation with Formula (\ref{eq-proof1a}) for $v'_{\bS}$, we obtain the desired formula (\ref{eq-proof1b}).

\textbf{3.} The first two steps show that $\text{C-Eig}(\cB_{R,\sv})\subset\{\Seq(\bT)\ |\ \bT\in\ASTab(R,\sv)\}$. The reverse inclusion is an immediate consequence of the construction of representations $V_{\oO}$, for admissible truncated orbits $\oO$ of type $(R,\sv)$, in Section \ref{subsec-rep}. Indeed, for any $\bT\in\ASTab(R,\sv)$, the defining formula (\ref{rep-X}) for the action of the generators $X_j$ shows that $V_{\oO_{\bT}}$ is calibrated and that $\Seq(\bT)\in \text{C-Eig}(\cB_{R,\sv})$.
\end{proof}

\subsection{Representations and calibrated spectrum of $\oH(R,\sv)$}

The algebra $\oH(R,\sv)$ was defined in Section \ref{subsec-quot}. We recall that we have a central element $C_{R,\sv}$ in $H(R,\sv)$:
\begin{equation}\label{form-C2}
C_{R,\sv}=X^{\Delta_{R,\sv}}=\prod_{j\in\Ve_R\cup\{0\}} X_j^{\kappa_j}\ ,
\end{equation}
where $\Delta_{R,\sv}$ is an element of the $\Z$-module $L_{R,\sv}$ invariant under the action of the Weyl group $W_0(R)$. The powers $\kappa_j$ are the coefficients of the expansion of $\Delta_{R,\sv}$ in the basis $\cB_{R,\sv}$. They are given in (\ref{form-Delta}).

Then the algebra $\oH(R,\sv)$ is the quotient of $H(R,\sv)$ over the relation $C_{R,\sv}=1$. We showed that in fact $\oH(R,\sv)$ is an affine Hecke algebra associated to the sublattice $Q_R+\Z\omega_{\sv}$ of the weight lattice.

\begin{definition}\label{delta-inv}
\textbf{(i)} Let $\bS\in\Seq_{R,\sv}$. We define:
\begin{equation}\label{Delta-bS}
\Delta(\bS):=\prod_{j\in\Ve_R\cup\{0\}} \qc_j(\bS)^{\kappa_j}\ .
\end{equation}
\noindent\textbf{(ii)} We define a subset of admissible standard tableaux of type $(R,\sv)$:
\[\overline{\ASTab(R,\sv)}:=\{\bT\in\ASTab(R,\sv)\ \text{such that}\ \Delta(\bT)=1\}\ .\]
\end{definition}
As we show in the proposition below, it follows directly from the definitions that $\Delta(\bS)$ is actually an invariant for the action of the Weyl group $W_0(R)$ on sequences.
\begin{proposition}\label{prop-delta-inv}
Let $\bS\in\Seq_{R,\sv}$. We have $\Delta\bigl(w(\bS)\bigr)=\Delta(\bS)$ for any $w\in W_0(R)$. In particular, the value $\Delta(\bS)$ is constant along the truncated orbits.
\end{proposition}
\begin{proof}
Recall from Section \ref{subsec-trunc} that to $\bS\in\Seq_{R,\sv}$ corresponds a character $\chi_{\bS}\in\text{Hom}\bigl(L_{R,\sv},\C(q)^{\times}\bigr)$ of the Abelian group $L_{R,\sv}$ which is identified with the group formed by elements $X^x$, $x\in L_{R,\sv}$. Then, the definition of $\Delta(\bS)$ becomes
\[\Delta(\bS)=\chi_{\bS}(C_{R,\sv})=\chi_{\bS}(X^{\Delta_{R,\sv}})\ .\]
From the definition of the action of $W_0(R)$ on sequences, we have, for $w\in W_0(R)$,
\[\Delta\bigl(w(\bS)\bigr)=\chi_{\bS}(X^{w^{-1}(\Delta_{R,\sv})})\ .\]
The result then follows from the fact that $\Delta_{R,\sv}$ is invariant under the action of $W_0(R)$.
\end{proof}
Let $\oO$ be a truncated orbit in $\Seq_{R,\sv}$. The proposition allows to define:
\[\Delta(\oO):=\Delta(\bS)\ ,\ \ \ \ \quad\text{where $\bS\in\oO$}.\]
\begin{proposition}\label{prop-rep-quot}
\textbf{(i)} Let $\oO$ be an admissible truncated orbit of type $(R,\sv)$. Then $V_{\oO}$ passes through the quotient and becomes a representation of $\oH(R,\sv)$ if and only if $\Delta(\oO)=1$.

\textbf{(ii)} The calibrated spectrum of $\oH(R,\sv)$ is in bijection with $\overline{\ASTab(R,\sv)}$.
\end{proposition}
\begin{proof} Let $\oO$ be an admissible truncated orbit of type $(R,\sv)$. From the definitions and Formula (\ref{rep-X}) giving the action of the generators $X_j$ on $V_{\oO}$, it follows immediately that
\[C_{R,\sv}(v_{\bT})=\Delta(\bT)\,v_{\bT}=\Delta(\oO)\,v_ {\bT}\ ,\ \ \ \quad\text{for any $\bT\in\oO$.}\]
So the central element $C_{R,\sv}$ acts on $V_{\oO}$ by multiplication by $\Delta(\oO)$. Item \textbf{(i)} follows.

Item \textbf{(ii)} is a direct consequence of the general result stated in Proposition \ref{prop-c-spec}.
\end{proof}

\begin{example}\label{ex-C-GL2}
We consider again the example of the affine Hecke algebra $H(GL_{N+1})$; see Section \ref{subsec-GL}. We recall from Section \ref{subsec-ex} that in this situation, the admissible truncated orbits are parametrised by placed skew shapes of size $N+1$ and the notion of standard tableaux is the classical one. Moreover, we have seen in Example \ref{ex-C-GL} that we have $C_{A_N,1}=Y_0Y_1\dots Y_N$. Therefore, in this case, $\Delta(\bT)$ is simply the product of the contents of the placed nodes of a standard tableau $\bT$. Obviously here, $\Delta(\bT)$ only depends on the shape of $\bT$.

The quotient algebra $\oH(A_N,1)$ is the affine Hecke algebra associated to the weight lattice, that is, the affine Hecke algebra associated to $SL_{N+1}$. Propositions \ref{prop-rep} and \ref{prop-rep-quot} then provide a construction of representations of $\oH(A_N,1)$ associated to placed skew shapes $\bGamma$ of size $N+1$ such that the product of the contents of the placed nodes of $\bGamma$ is equal to 1.\hfill$\triangle$
\end{example}

\subsection{Minimal idempotents and irreducibility}
We recall that standard tableaux of type $(R,\sv)$ are characterized by their sequence of contents. Let $\oO$ be an admissible truncated orbit in $\Seq_{R,\sv}$ and let $\bT\in\oO$. We define the following elements of $\C(q)H(R,\sv)$:
\begin{equation}\label{def-ET}
E_{\bT}:=\prod_{j\in\Ve_R\cup\{0\}} \Bigl(\!\!\!\prod_{\text{\scriptsize{$\begin{array}{c}\bT'\in\oO\\\qc_j(\bT')\neq\qc_j(\bT)\end{array}$}}}\!\!\!\frac{X_j-\qc_j(\bT')}{\qc_j(\bT)-\qc_j(\bT')}\Bigr)\ .
\end{equation}
We have, on the representation $V_{\oO}$, that:
\[E_{\bT}(v_{\bT})=v_{\bT}\ \ \qquad\text{and}\qquad\ \ E_{\bT}(v_{\bT'})=0\ \ \ \text{if $\bT'\in\oO\backslash\{\bT\}$\ .}\]
This follows from Formulas (\ref{rep-X}) defining the action of the generators $X_j$ on $V_{\oO}$, together with the fact that for any two different $\bT,\bT'\in\STab(R,\sv)$, there is some $j\in\Ve_R\cup\{0\}$ such that $\qc_j(\bT)\neq\qc_j(\bT')$.

As operators on $V_{\oO}$, the set $\{E_{\bT}\}_{\bT\in\oO}$ is a complete set of pairwise orthogonal idempotents of $\text{End}(V_{\oO})$.

\begin{proposition}\label{prop-rep2}
The representations $V_{\oO}$, where $\oO$ runs over the set of admissible truncated orbits of type $(R,\sv)$, are pairwise non-isomorphic irreducible representations of $\C(q)H(R,\sv)$.
\end{proposition}
\begin{proof} Let $\oO$ and $\oO'$ be two distinct admissible truncated orbits of type $(R,\sv)$. Let $\bT\in\oO$. From Formulas (\ref{rep-X}), we have first that the vector $v_{\bT}$ in $V_{\oO}$ is a common eigenvector for the operators $X_j$ with eigenvalues given by $\Seq(\bT)$, and second that, in $V_{\oO'}$, there is no common eigenvector of the operators $X_j$ with the same eigenvalues. Indeed, as $\oO$ and $\oO'$ are disjoint, we have that $\bT'\neq\bT$ for every $\bT'\in\oO'$, and in turn, $\Seq(\bT')\neq\Seq(\bT)$ for every $\bT'\in\oO'$. We conclude that $V_{\oO}$ and $V_{\oO'}$ are not isomorphic.

Let $\oO$ an admissible truncated orbits of type $(R,\sv)$ and $W\neq\{0\}$ an invariant subspace of $V_{\oO}$ for the action of $\C(q)H(R,\sv)$.

Let $v$ be a non-zero vector of $W$. It is a linear combination of the basis vectors $v_{\bT}$, $\bT\in\oO$, of $V_{\oO}$ with at least one non-zero coefficient. By applying the idempotents $E_{\bT}$ on $v$ and using that $W$ is stable by $\C(q)H(R,\sv)$, we obtain that $v_{\bT}\in W$ for some $\bT\in\oO$.

Now let $i\in\Ve_R$ such that $\ovr_i(\bT)\neq\bze$. Recall that it means $\qc_i(\bT)\neq q^{\pm2}\qc_{i-1}(\bT)$. From Formula (\ref{rep-g}), we obtain that the coefficient in front of $v_{\ovr_i(\bT)}$ in $g_i(v_{\bT})$ is different from $0$, and in turn that $v_{\ovr_i(\bT)}\in W$. As all elements of the truncated orbit $\oO$ can be obtained from $\bT$ by repeated applications of truncated operators $\ovr_i$, $i\in\Ve_R$, we conclude that $v_{\bT'}\in W$ for any $\bT'\in\oO$. This shows that $W=V_{\oO}$ and concludes the proof of the proposition.
\end{proof}

\section{Restriction to finite Hecke algebras and classical limit}\label{sec-res-lim}

\subsection{Restriction to $H_0(R)$}

We recall that the finite Hecke algebra $H_0(R)$ of type $R$ is isomorphic to the subalgebra of $H(R,\sv)$ generated by $g_i$, $i\in\Ve_R$.

Let $V$ be a representation of $\C(q)H(R,\sv)$. We denote by $V^{\text{fin}}$ the restriction of $V$ to the subalgebra $\C(q)H_0(R)$. In particular, for $\oO$ an admissible truncated orbit of type $(R,\sv)$, the representation $V_{\oO}^{\text{fin}}$ is the restriction of the representation $V_{\oO}$ constructed in the preceding section.

The action of $H_0(R)$ on $V_{\oO}^{\text{fin}}$ is given explicitly by Formulas (\ref{rep-g}). In general, the representation $V_{\oO}^{\text{fin}}$ is not irreducible. We are going to formulate a sufficient condition on $\oO$ so that $V_{\oO}^{\text{fin}}$ is an irreducible representation of $\C(q)H_0(R)$ and then interpret it combinatorially.

\begin{remark}\label{rem-pl}
A particular situation is when in all elements of $\oO$, a single place $\gamma_1$ appears (this will be in particular satisfied for truncated orbits of level 1 to be defined below). Then the representation $V_{\oO}^{\text{fin}}$ is given purely in terms of moving boxes in skew diagrams, and the value of the place is not relevant. Indeed, in Formulas (\ref{rep-g}) giving the action of the generators $g_i$, we can see that $\gamma_1$ disappears by simplification
\hfill$\triangle$
\end{remark}

\paragraph{\textbf{Truncated orbits of level 1.}} We consider the following definition for arbitrary truncated orbits in $\Seq_{R,\sv}$. However we will use it only for admissible truncated orbits (the ones corresponding to representations).

\begin{definition}\label{def-lev}
Let $\oO$ be a truncated orbit in $\Seq_{R,\sv}$. We say that $\oO$ is of level 1 if the first coefficient of the sequences in $\oO$ remain constant, namely if
\[\qc_0(\bS)=\qc_0(\bS')\ ,\ \ \ \ \quad\text{for any $\bS,\bS'\in\oO$\,.}\]
\end{definition}

\begin{proposition}\label{prop-fin}
Let $\oO$ be an admissible truncated orbit of type $(R,\sv)$. If $\oO$ is of level 1 then $V_{\oO}^{\text{fin}}$ is an irreducible representation of $\C(q)H_0(R)$.
\end{proposition}
\begin{proof}
Let $\oO$ be an admissible truncated orbit of level 1 and set $c:=\qc_0(\bT)$ for $\bT\in\oO$. The representation $V_{\oO}$ is an irreducible representation of $\C(q)H(R,\sv)$ and on $V_{\oO}$, the generator $X_0$ acts as $c\cdot\text{Id}_{V_{\oO}}$ by assumption on $\oO$.

Let $H(R,\sv)^c$ denote the quotient of $\C(q)H(R,\sv)$ by the relation $X_0-c=0$. Then the representation $V_{\oO}$ factors through this quotient and becomes an irreducible representation of $H(R,\sv)^c$. Let $\pi^c$ be the canonical surjective map from $\C(q)H(R,\sv)$ to its quotient $H(R,\sv)^c$ and consider the following diagram:
\[\begin{array}{ccc}\C(q)H(R,\sv) & \stackrel{\pi^c}{\longrightarrow} & H(R,\sv)^c \\
\cup & &\cup\\
\C(q)H_0(R) & \stackrel{\pi^c}{\longrightarrow} & H_0(R)^c\end{array} \]
where we denoted $H_0(R)^c:=\pi^c\bigl(\C(q)H_0(R)\bigr)$. Considering the restriction to $\C(q)H_0(R)$, the representation $V_{\oO}^{\text{fin}}$ factors through $H_0(R)^c$, which is a subalgebra of $H(R,\sv)^c$.

We claim that we have in fact $H_0(R)^c=H(R,\sv)^c$. Indeed this follows immediately from $\pi^c(X_0)=c$ and the relations $X_i=g_iX_{i-1}g_i$ for $i\in\Ve_R$ showing that $\pi^c(X_j)\in H_0(R)^c$ for any $j\in\Ve_R\cup\{0\}$. Therefore $V_{\oO}^{\text{fin}}$ is irreducible for $H_0(R)^c$ (since $V_{\oO}$ is irreducible for $H(R,\sv)^c$) and in turn, is irreducible for $\C(q)H_0(R)$.
\end{proof}

\begin{remarks}\label{rem-res}
\textbf{(i)} As we just saw in the proof, the quotient $H(R,\sv)^c$ of the affine Hecke algebra $H(R,\sv)$ by a relation of the form $X_0=c$ is in fact a quotient of the finite Hecke algebra $H_0(R)$. This is implied by the first set of defining relations (\ref{rel-H-gX1}) between the generators $g_i$ and $X_j$. The other defining relations (\ref{rel-H-gX2})--(\ref{rel-H-gX5}) may imply other relations between the generators $g_i$ in $H(R,\sv)^c$. So, in general, $H(R,\sv)^c$ is a non-trivial quotient of $H_0(R)$. This is related to the fact that, in general, the irreducible representations obtained in Proposition \ref{prop-fin} do not exhaust the set of irreducible representations of $\C(q)H_0(R)$. Note however, that one could ask how much of the set of irreducible representations of $\C(q)H_0(R)$ is obtained by fixing $R$ and varying $\sv$.

\textbf{(ii)} The terminology ``level 1" in Definition \ref{def-lev} refers to the fact that $X_0$ has only one eigenvalue in $V_{\oO}$. Admissible truncated orbits of higher levels would lead to the study of quotients of $H(R,\sv)$ by characteristic equations for $X_0$ of higher degrees (in some contexts, quotients of this sort are called cyclotomic quotients).

\textbf{(iii)} A particular situation is when $H(R,\sv)$ is the affine Hecke algebra $H(GL_{N+1})$. In this case, it is known that we have $H(R,\sv)^c=\C(q)H_0(R)$ (equivalently, all irreducible representations of $\C(q)H_0(R)$ are obtained via Proposition \ref{prop-fin}). Moreover, the quotients of higher levels  in this case are well-known: they are the so-called cyclotomic Hecke algebras, or Ariki--Koike algebras \cite{AK}.

\textbf{(iv)} Finally we note that the converse of Proposition is not true. It is already seen in the $GL_{3}$ situation, where the representation associated to the skew diagram $\begin{array}{cc}
  & \hspace{-0.45cm}\Box \\[-0.7em]
\Box & \hspace{-0.45cm}\Box
\end{array}$, restricted to $H_0(A_2)$, is an irreducible representation, while $X_0$ has two different eigenvalues.
\hfill$\triangle$\end{remarks}

\paragraph{\textbf{Combinatorial characterisation of admissible truncated orbits of level 1.}} Recall that a placed skew shape is an equivalence class of placed skew diagrams (Section \ref{sec-tab-cla}). 

We will from now on make the following slight abuse of terminology. A skew partition $\lambda/\mu$ is a ``usual partition" when $\mu$ is empty. We will say that a placed skew shape is a usual partition if it is the equivalence class of a placed skew diagram of the form $\{(\Gamma_1,\gamma_1)\}$, where $\Gamma_1$ is the diagram of a usual partition (in particular, there is a single place; see Remark \ref{rem-pl}).

Let $\Gamma$ be a skew diagram and $\theta$ a node of $\Gamma$. We call $\theta$ a \emph{top left node} of $\Gamma$ if the positions above $\theta$ in the same column and the positions to the left of $\theta$ in the same line are all empty. A placed node $(\theta,\gamma)$ of a placed skew diagram is called a \emph{top left placed node} if $\theta$ is a top left node of its diagram. Finally, we note that the notion of being a top left placed node is compatible with the equivalence relation of placed skew diagram, and therefore is well-defined for a placed skew shape.

Let $\bGamma$ be a placed skew diagram with a unique top left placed node. It means that there is only one place: $\bGamma=\{(\Gamma_1,\gamma_1)\}$, and moreover the skew diagram $\Gamma_1$ has only one top left node. Using the equivalence relation, we can translate $\Gamma_1$ such that the top left node moves to position $(1,1)$ (the place $\gamma_1$ may have to change; see example (\ref{ex-eq})). We thus see that $\bGamma$ is equivalent to a placed skew diagram of the form $\{(\Gamma'_1,\gamma'_1)\}$, where $\Gamma'_1$ is the diagram of a usual partition.

As a conclusion, we have explained that a placed skew shape is a usual partition if and only if it contains a unique top left placed node.

\begin{proposition}\label{prop-lev2}
Let $\oO$ be an admissible truncated orbit of type $(R,\sv)$. Then $\oO$ is of level 1 if and only if for every $\bT=(\bT_1,\bT_2,\bT_3)\in\oO$, the shapes of $\bT_1,\bT_2$ and $\bT_3$ are usual partitions.
\end{proposition}
\begin{proof}
Assume first that for every $\bT=(\bT_1,\bT_2,\bT_3)\in\oO$, the shapes of $\bT_1,\bT_2,\bT_3$ are usual partitions. Thus, for every $\bT=(\bT_1,\bT_2,\bT_3)\in\oO$, there is a unique top left placed node in the shapes of $\bT_1,\bT_2,\bT_3$. By standardness of $\bT_1,\bT_2,\bT_3$, the only possibility for $0$ is to be contained in this placed node. This shows that $\ovr_1(\bT)=\bze$ for every $\bT\in\oO$ (otherwise, in $\ovr_1(\bT)$, $0$ would be contained in a placed node which is not top left). As $r_1$ is the only simple reflection in $W_0(R)$ which modifies the position of $0$, we conclude that $\qc_0(\bT)$ remains constant along $\oO$.

Assume that we have $\bT=(\bT_1,\bT_2,\bT_3)\in\oO$ such that, for some $u\in\{1,2,3\}$, the shape $\bGamma_u$ of $\bT_u$ is not a usual partition. It means there are at least two top left placed nodes in $\bGamma_u$. So let $\btheta$ be the one containing $0$ in $\bT_u$ and let $\btheta'$ be a different one. It is clear that $\btheta$ and $\btheta'$ can not lie in the same diagonal of the same diagram, and therefore $\qc(\btheta)\neq\qc(\btheta')$.

Let $\bT'_u$ be a standard tableau of shape $\bGamma_u$ such that $0$ is contained in $\btheta'$ (such a standard tableau exists as $\btheta'$ is a top left placed node). Let $l_u+1$ be the size of $\bGamma_u$ ($l_u$ is $l$, $l'$ or $l''$ depending on $u$). We have seen in Proposition \ref{prop-tab-A2} that $\bT'_u$ can be obtained from $\bT_u$ by repeated applications of truncated operators of 
the symmetric group $S_{l_u+1}$. So we have $\bT_u'=\overline{w}(\bT_u)$ for some $w\in S_{l_u+1}$. 
The symmetric group $S_{l_u+1}$ is a subgroup of $W_0(R)$, and so, considering $w$ as an element of $W_0(R)$, we set $\bT':=\overline{w}(\bT)$.

From the definition of truncated operators, we have that $\bT'\neq\bze$ since we already had $\overline{w}(\bT_u)\neq\bze$. So we obtained $\bT'\in\oO$ such that $\qc_0(\bT')=\qc(\btheta')$. From the fact that $\qc_0(\bT)=\qc(\btheta)\neq\qc(\btheta')$, we conclude that $\oO$ is not of level $1$.
\end{proof}

\begin{example}\label{ex-res}
We take $R=D_4$ and $\sv=1$ in the standard labelling, as in Section \ref{subsec-ex}\textbf{(b)}. We show an example of an admissible truncated orbits of level 1:
\[\left(
\begin{array}{cccc}
\fbox{\scriptsize{$0$}} & \hspace{-0.35cm}\fbox{\scriptsize{$1$}} & \hspace{-0.35cm}\fbox{\scriptsize{$2$}}  & \hspace{-0.35cm}\fbox{\scriptsize{$3$}} \\[-0.2em]
 & & & 
\end{array},
\begin{array}{ccc}
\fbox{\scriptsize{$0$}} & \hspace{-0.35cm}\fbox{\scriptsize{$1$}} & \hspace{-0.35cm}\fbox{\scriptsize{$2$}} \\[-0.2em]
\fbox{\scriptsize{$\un{3}$}} & &
\end{array}
\right)
\ \ \stackrel{\ \ \ovr_{3_{}}\ \,}{\longleftrightarrow}\ \ \left(
\begin{array}{cc}
\fbox{\scriptsize{$0$}} & \hspace{-0.35cm}\fbox{\scriptsize{$1$}} \\[-0.2em]
\fbox{\scriptsize{$2$}} & \hspace{-0.35cm}\fbox{\scriptsize{$3$}} 
\end{array},
\begin{array}{ccc}
\fbox{\scriptsize{$0$}} & \hspace{-0.35cm}\fbox{\scriptsize{$1$}} & \hspace{-0.35cm}\fbox{\scriptsize{$\un{3}$}} \\[-0.3em]
\fbox{\scriptsize{$2$}} & &
\end{array}
\right)
\ \ \stackrel{\ \ \ovr_{2_{}}\ \,}{\longleftrightarrow}\ \ 
\left(
\begin{array}{cc}
\fbox{\scriptsize{$0$}} & \hspace{-0.35cm}\fbox{\scriptsize{$2$}} \\[-0.2em]
\fbox{\scriptsize{$1$}} & \hspace{-0.35cm}\fbox{\scriptsize{$3$}} 
\end{array},
\begin{array}{ccc}
\fbox{\scriptsize{$0$}} & \hspace{-0.35cm}\fbox{\scriptsize{$2$}} & \hspace{-0.35cm}\fbox{\scriptsize{$\un{3}$}} \\[-0.3em]
\fbox{\scriptsize{$1$}} & &
\end{array}
\right)\]
The non-indicated actions of $\ovr_{1},\ovr_2,\ovr_{3},\ovr_{\un{3}}$ give $\bze$. This truncated orbit gives rise to an irreducible representation of $H_0(D_4)$.
\hfill$\triangle$\end{example}

\subsection{Representations of $W_0(R)$ as classical limits}

Let $\bT$ be a tableau of type $(R,\sv)$ and, for $i\in\Ve_R$, let $\btheta_i=(\theta_i,\gamma_i)$ be the placed node of $\bT$ with number $i$. For $i,j\in\Ve_R$, we write $\pl_i(\bT)=\pl_j(\bT)$ if $\gamma_i=\gamma_j$ and $\pl_i(\bT)\neq\pl_j(\bT)$ otherwise. We recall that tableaux are only considered up to equivalence and we note that the values of the $\gamma_i$'s are therefore not uniquely defined. Nevertheless, the property $\pl_i(\bT)=\pl_j(\bT)$ (or its negation) is well-defined since compatible with the equivalence relation.

Let $\oO$ be an admissible truncated orbit of type $(R,\sv)$. In this subsection, we will make the following assumption. If $\gamma_1$ and $\gamma_2$ are two different places which appear in an element $\bT\in\oO$ then we assume that:
\begin{equation}\label{hyp-pl}
\left(\frac{\gamma_1}{\gamma_2}\right)\Bigr\rvert_{q=\pm1}\neq 1\ ,
\end{equation}
where, for $x\in\C(q)$, $x\bigr\rvert_{q=\pm1}$ denotes the evaluation of $x$ at $q=\pm1$.

Let $V$ be a representation of the finite Hecke algebra $\C(q)H_0(R)$ and let $\cB$ be a basis of $V$. 
Let $M_i$ be the matrix representing the generator $g_i$, $i\in\Ve_R$, in the basis $\cB$. Let $\epsilon\in\{-1,1\}$ and assume that all coefficients of $M_i$ are non-singular when evaluated at $q=\epsilon$, and this for all $i\in\Ve_R$. Then, sending the simple reflection $r_i$, $i\in\Ve_R$, to the evaluation of the matrix $M_i$ at $q=\epsilon$, we obtain a complex representation of the Weyl group $W_0(R)$ on a complex vector space with basis $\cB$. In this situation, we say that the evaluation at $q=\epsilon$ of the representation $V$ in the basis $\cB$ exists.

\begin{proposition}\label{prop-clas}
Let $\oO$ be an admissible truncated orbit of type $(R,\sv)$ satisfying (\ref{hyp-pl}) and let $\epsilon\in\{-1,1\}$. The evaluation at $q=\epsilon$ of the representation $V_{\oO}^{\text{fin}}$ in the basis $\{v_{\bT}\}_{\bT\in\oO}$ exists, and the action of $W_0(R)$ is given by 
\begin{itemize}
\item if $\pl_i(\bT)\neq\pl_{i-1}(\bT)$,
\begin{equation}\label{rep-r1}
r_{i}(v_{\bT})= \epsilon\,v_{\ovr_i(\bT)}\ ,
\end{equation}
\item if $\pl_i(\bT)=\pl_{i-1}(\bT)$,
\begin{equation}\label{rep-r2}
r_{i}(v_{\bT})=\epsilon\,\Bigl(\displaystyle\frac{1}{\cc_{i}(\bT)-\cc_{i-1}(\bT)}\,v_{\bT}+\bigl(1+\displaystyle\frac{1}{\cc_{i}(\bT)-\cc_{i-1}(\bT)}\bigr)v_{\ovr_{i}(\bT)}\Bigr)\ ,
\end{equation}\end{itemize}
where $\bT\in\oO$ and $i\in\Ve_R$ (we recall that $v_{\bze}:=0$).
\end{proposition}
\begin{proof} The proof is a straightforward verification, consisting in the evaluation at $q=\epsilon$ of the coefficients appearing in Formulas (\ref{rep-g}). We only recall that $\qc_i(\bT)=\gamma_iq^{2\cc_i(\bT)}$ if $i$ is contained in a placed node with place $\gamma_i$ and that we use Condition (\ref{hyp-pl}) on the places.
\end{proof}

Regarding Formulas (\ref{rep-r1})-(\ref{rep-r2}), we note that if $\pl_i(\bT)\neq\pl_{i-1}(\bT)$ then $\ovr_i(\bT)\neq\bze$, and that if $\pl_i(\bT)=\pl_{i-1}(\bT)$ then $\cc_{i}(\bT)\neq\cc_{i-1}(\bT)$ since $\bT$ is standard.

Let $U_{\oO}$ be the representation of the preceding proposition obtained when $q$ is evaluated at $1$, and $U_{\oO}^-$ the one obtained at $-1$. From Formulas (\ref{rep-r1})-(\ref{rep-r2}), we see that:
\[U^-_{\oO}\cong U_{\oO}\otimes\text{sign}\ ,\]
where $\text{sign}$ denotes the signature representation of $W_0(R)$.

\begin{remarks}\label{rem-hyp-pl}
\textbf{(i)} A remark similar to Remark \ref{rem-pl} applies here: when in all elements of $\oO$, a single place $\gamma_1$ appears, the representation $U_{\oO}$ is given purely in terms of moving boxes in skew diagrams, and the value of the place $\gamma_1$ is not relevant. Outside of this situation, even though the values of the places do not appear in Formulas (\ref{rep-r1})-(\ref{rep-r2}), these values determine the structure of $\oO$; see examples in Section \ref{subsec-ex}.

\textbf{(ii)} Condition \ref{hyp-pl} on $\oO$ was assumed for simplicity because it allowed us to perform the evaluation at $q=\pm1$ and to obtain Formulas (\ref{rep-r1})-(\ref{rep-r2}) as ``classical limits" of representations of $H_0(R)$. Therefore it followed immediately that we obtained representations of $W_0(R)$. Alternatively, one could define Formulas (\ref{rep-r1})-(\ref{rep-r2}) for any truncated admissible orbit $\oO$, and then check directly that the defining relations of $W_0(R)$ are indeed satisfied, without any reference to $H_0(R)$.
\hfill$\triangle$\end{remarks}

\end{document}